\documentclass{amsart}

\def\Xint#1{\mathchoice
   {\XXint\displaystyle\textstyle{#1}}%
   {\XXint\textstyle\scriptstyle{#1}}%
   {\XXint\scriptstyle\scriptscriptstyle{#1}}%
   {\XXint\scriptscriptstyle\scriptscriptstyle{#1}}%
   \!\int}
\def\XXint#1#2#3{{\setbox0=\hbox{$#1{#2#3}{\int}$}
     \vcenter{\hbox{$#2#3$}}\kern-.5\wd0}}
\def\dashint{\Xint-}

\DeclareFontFamily{U}  {MnSymbolC}{}
\DeclareFontShape{U}{MnSymbolC}{m}{n}{
    <-6>  MnSymbolC5
   <6-7>  MnSymbolC6
   <7-8>  MnSymbolC7
   <8-9>  MnSymbolC8
   <9-10> MnSymbolC9
  <10-12> MnSymbolC10
  <12->   MnSymbolC12}{}
\DeclareFontShape{U}{MnSymbolC}{b}{n}{
    <-6>  MnSymbolC-Bold5
   <6-7>  MnSymbolC-Bold6
   <7-8>  MnSymbolC-Bold7
   <8-9>  MnSymbolC-Bold8
   <9-10> MnSymbolC-Bold9
  <10-12> MnSymbolC-Bold10
  <12->   MnSymbolC-Bold12}{}
\DeclareSymbolFont{MnSyC}         {U}  {MnSymbolC}{m}{n}
\DeclareMathSymbol{\restrict}{\mathbin}{MnSyC}{'267}

\newcommand\restr[2]{\left.#1\right|_{#2}}

\usepackage{hyperref}

\newtheorem{theorem}{Theorem}[section]
\newtheorem{proposition}[theorem]{Proposition}
\newtheorem{lemma}[theorem]{Lemma}
\newtheorem{corollary}[theorem]{Corollary}

\theoremstyle{definition}
\newtheorem{definition}[theorem]{Definition}
\newtheorem{notation}[theorem]{Notation}

\theoremstyle{remark}
\newtheorem{remark}[theorem]{Remark}

\numberwithin{equation}{section}

\newcommand{\Lip}{\operatorname{Lip}}
\newcommand{\dom}{\operatorname{Dom}}
\newcommand{\im}{\operatorname{Im}}
\newcommand{\graph}{\operatorname{Graph}}
\renewcommand{\llcorner}{\restrict}

\begin{document}

\begin{abstract}
We prove the equivalence of two seemingly very different ways of
generalising Rademacher's theorem to metric measure spaces.  One such
generalisation is based upon the notion of forming partial derivatives along
a very rich structure of Lipschitz curves in a way analogous to the
differentiability theory of Euclidean spaces.  This approach to
differentiability in this generality appears here for the first time
and by examining this structure further, we naturally arrive to
several descriptions of Lipschitz differentiability spaces.
\end{abstract}

\thanks{I would like to thank David Preiss for his dedicated reading
  of this manuscript and for our insightful conversations throughout
  my time as his student.  I would also like to thank Guy David (UCLA)
  for spotting a technical error in the first version of Section
  \ref{sec:spanreps}.  This work was supported by the EPSRC}

\date{11th June 2013}
\author{David Bate}
\address{Mathematics Institute\\
Zeeman Building\\
University of Warwick\\
Coventry \\
CV4 7AL
UK}
\curraddr{}
\email{david@bate.org.uk}

\title{Structure of measures in Lipschitz differentiability spaces}
\maketitle

\section{Introduction}\label{sec:intro}
One of our goals is to show the equivalence of two seemingly very
different ways of generalising Rademacher's theorem on the almost
everywhere differentiability of (real valued) Lipschitz functions
defined on Euclidean spaces to metric measure spaces $(X,d,\mu)$.  The
first stems from the work of Cheeger \cite{cheeger-diff} and further
developed into the concept of a \emph{Lipschitz differentiability
  space} following the work of several authors, most notably Keith
\cite{keith}.  It is based upon the notion of differentiability with
respect to suitable \emph{charts} - Lipschitz functions $\varphi\colon
X \to \mathbb R^n$.  The second originates in some ideas of Alberti
\cite{alberti-rankone} and is based on the immediate consequence of
Lebesgue's theorem that one may always differentiate a Lipschitz
function almost everywhere along a Lipschitz curve.  For measures
represented by integration over such curves, this leads to a notion
akin to partial differentiability $\mu$ almost everywhere and, if
$\mu$ has more such representations, can lead to a notion of
differentiability.

To describe our results in more detail, we now informally introduce
the two main concepts studied in this paper.

An \emph{$n$-dimensional Lipschitz differentiability space} may be
thought of as a metric measure space $(X,d,\mu)$ for which there is a
Lipschitz $\varphi\colon X\to\mathbb R^n$ (which is termed a
\emph{chart}) with respect to which every real valued Lipschitz
function is differentiable $\mu$ almost everywhere (a precise
definition that allows a countable decomposition into such charts is
given in Definition \ref{def:chart}).  Here differentiability of a
function $f$ at $x_0$ requires the approximation of $f(x)-f(x_0)$ by a
linear form of $\varphi(x)-\varphi(x_0)$ for $x$ close to $x_0$.  More
precisely, it requires the existence of a linear $L:\mathbb R^n\to
\mathbb R$ such that
\[f(x)=f(x_0)+L(\varphi(x)-\varphi(x_0))+o(d(x,x_0)).\]

Perhaps the best known non-trivial example of Lipschitz
differentiability spaces is the Heisenberg group (see
\cite{heinonen}).  In fact, by the result of \cite{cheeger-diff}, any
doubling space satisfying the Poincar\'e inequality is a Lipschitz
differentiability space.  However the converse does not hold, for
example subspaces of Lipschitz differentiability spaces are
Lipschitz differentiability spaces, while an analogous statement fails
for spaces satisfying the Poincar\'e inequality.

An \emph{Alberti representation} of a measure $\mu$ (see Definition
\ref{defn:albrep}) is an integral combination of 1-rectifiable
measures, each supported on a Lipschitz curve in $X$ (whose domain is
not necessarily connected), that calculates
the $\mu$ measure of any Borel subset of $X$.  Further, we distinguish
different Alberti representations by considering their
\emph{direction}; given a Lipschitz function $\varphi\colon
X\to\mathbb R^n$ we say that an Alberti representation is in the
direction of a cone $C\subset \mathbb R^n$ if, for almost every
Lipschitz curve $\gamma$ (with respect to the integrating measure) and
almost every $t_0$, $(\varphi\circ\gamma)'(t_0)\in C$.  We then say
that a collection of $n$ Alberti representations are
\emph{independent} if there exists a Lipschitz $\varphi\colon
X\to\mathbb R^n$ and a collection of $n$ independent
cones in $\mathbb R^n$ such that each Alberti representation is in the
direction of its own cone (see Definition \ref{defn:phidir}).

Fubini's theorem provides a simple example of an Alberti
representation: it represents the $n$-dimensional Lebesgue measure as
the integral of one dimensional Lebesgue measures on lines parallel to
the $x$-axis.  In a general metric measure space, as for this example,
if the measure $\mu$ has an Alberti representation, it is possible to
define the notion of a partial derivative of a Lipschitz function in
the direction of this representation at almost every point.
Therefore, if a measure has $n$ independent Alberti representations,
there exists $n$ independent partial derivatives.  This can lead to a
new notion of the derivative of a Lipschitz function.

In this generality, the approach to differentiability via
representations of the underlying measure by integration over curves
appears here for the first time.  However, we will not consider it as
a separate notion, but rather as a description of the usual notion of
Lipschitz differentiability spaces.  One advantage of this is that it
naturally leads to several descriptions of Lipschitz differentiability
spaces, and emphasises the key point that such spaces possess a very
rich structure of curves.

The research presented here began with an observation of Preiss that,
from a theorem of Cheeger and Kleiner, one may obtain an Alberti
representation of any doubling Lipschitz differentiability space that
satisfies the Poincar\'e inequality (see Section \ref{sec:cheeger} for
details).  We significantly strengthen this observation by showing
that (even in an arbitrary Lipschitz differentiability space) there
exist so many Alberti representations that they completely describe
the derivative of a Lipschitz function.  This gives a natural
realisation of the possibility of describing such derivatives by
partial derivatives which was first pointed out in
\cite{cheegerkleiner-rnp} (see Corollary \ref{cor:tangentcurves}).

We note that the notion of an Alberti representation is similar to,
but more general than, the notion of a \emph{test plan} recently
studied by Ambrosio, Gigli and Savar\'e in \cite{ags-calculusheatflow} and many other
papers by the same authors.  Indeed, it is easy to see that any test
plan defines an Alberti representation.

In the first half of this paper we show that an $n$-dimensional
Lipschitz differentiability space possesses $n$ independent Alberti
representations and that the derivative obtained from the partial
derivatives agrees with the existing derivative.  Further, this
implies that, for every Lipschitz function, at almost every point one
of the partial derivatives must be comparable to the pointwise
Lipschitz constant.  That is, this collection of Alberti
representations is \emph{universal} (see Definition
\ref{def:universal}).

Since our only hypothesis on Lipschitz differentiability spaces is
that Lipschitz functions are differentiable almost everywhere, we
prove these results by first reducing the problem to showing that a
certain class of sets must all have measure zero.  Then, for each set
$S$ in this class, we construct a Lipschitz function that
is differentiable almost nowhere on $S$, so that $S$ has measure zero.

In the second half of the paper we characterise Lipschitz
differentiability spaces via Alberti representations and related null
sets.  First we see that a collection of Alberti representations being
universal is exactly the condition required for the partial
derivatives to form a derivative and so obtain our first
characterisation.  We observe that many subsets of Lipschitz
differentiability space must necessarily have measure zero, namely the
set of non-differentiability points of any Lipschitz function.  One may then
ask if there exists a collection of subsets of a metric measure space
such that, if each of these sets has measure zero, the space is a
Lipschitz differentiability space.  We answer this positively by
giving several characterisations of Lipschitz differentiability spaces
via null sets that arise naturally when constructing many independent
Alberti representations of a measure.

The existence of Alberti representations, and the
ability to calculate the derivative from them, gives us a strong
connection to the differentiability theory in Euclidean spaces.
However, the theory of Lipschitz differentiability spaces remains more
exotic than the Euclidean case as this theory cannot be improved beyond
representations by 1-rectifiable measures.  Indeed, it is known that
the Heisenberg group is a Lipschitz differentiability space consisting
of a single chart of dimension two.  By our characterisations, we
may therefore conclude that there exists a large collection of pairs
of independent Alberti representations.  However, it is also known that
the Heisenberg group is 2-purely unrectifiable (see \cite{ambker-rectifiablesets},
Theorem 7.2) and so there cannot exist a representation of its measure
by 2-rectifiable measures.  We therefore see that, although there
exists a very rich collection of 1 dimensional representations, they
cannot be combined to produce a higher dimensional structure.

The paper is organised as follows.  In Section \ref{sec:albertireps}
we define an Alberti representation and deduce the elementary
properties of measures with an Alberti representation.  In section
\ref{sec:firstanalysis} we recall the notion of a derivative with
respect to a chart given in \cite{cheeger-diff} and \cite{keith}.  We
also give a characterisation, that first appeared in
\cite{porousdiff}, describing the uniqueness of such a derivative
as a property inherent to the chart.  While this observation is
simple, it's consequences will be essential to our construction of a
non-differentiable Lipschitz function.

We then proceed to prove the existence of $n$ independent Alberti
representations of any $n$-dimensional Lipschitz differentiability
space (see Theorem \ref{thm:spanrep}).  This proof consists of the
following three constructions which are valid for any chart in a metric
measure space.

In Section \ref{sec:construction} we present a general construction of
a non-differentiable Lipschitz function (see Proposition
\ref{prop:cons}).  Given a set $S$ and a sequence of 1- Lipschitz
functions that behave, at some scale, like non-differentiable
functions at points of $S$, we are able to produce a Lipschitz
function that is not differentiable almost everywhere on $S$.  More
precisely, suppose that for each function $f_m$ of this sequence and
each $x\in S$, there exist $y,z\in X$ with $0<d(y,x),d(z,x)<1/m$ such
that
\[\frac{|f_m(x)- f_m(y)|}{d(x,y)} \leq \frac{1}{m} \text{ and }
\frac{|f_m(x)- f_m(z)|}{d(x,z)} \geq \frac{1}{2}.
\]
Then we modify and combine the functions of this sequence to construct
a Lipschitz function that is not differentiable almost everywhere on $S$.

In Section \ref{sec:reps} we define a class of subsets of a metric
measure space and show that the measure has many independent Alberti
representations if and only if each of these sets has measure zero
(see Theorem \ref{thm:manyrep}).  This class is a natural
generalisation of the class of sets considered in
\cite{acp-structurenullsets} (see \cite{acp-diffoflip} for an announcement), in
which the authors investigate measures in Euclidean space with respect
to which every Lipschitz function is differentiable almost everywhere.
In this paper, such sets will play a fundamental role in our
description of Lipschitz differentiability spaces, via Alberti
representations.

In Section
\ref{sec:spanreps}, we provide a
construction of a sequence of Lipschitz functions that satisfy the
hypotheses of the construction from Section \ref{sec:construction} for
any set in the above class.
In particular, there exists a Lipschitz function that is
differentiable almost nowhere on such a set, establishing the
existence of many independent Alberti representations of a Lipschitz
differentiability space.

We use the existence of Alberti representations to give complete
descriptions of Lipschitz differentiability spaces.  In Section
\ref{sec:albertichar} we show how the partial derivatives obtained from a
universal collection of Alberti representations form a derivative.
For example, suppose that a Lipschitz function has all partial derivatives
with respect to a universal collection of Alberti representations
equal to zero.  Then it must have derivative zero almost everywhere.
By using the linearity of the derivative and investigating universal
Alberti representations further, we are able to generalise this
argument to show that the partial derivatives of any Lipschitz
function form a derivative almost everywhere.  By combining this with
our previous results, we characterise Lipschitz differentiability
spaces as those metric measure spaces with a universal collection of
Alberti representations (see Theorem \ref{thm:albchar}).

In Section \ref{sec:char} we show that, for any metric measure space
with $n$ independent Alberti representations, there exists a
$\delta>0$ (depending on the properties of the Alberti
representations) such that, for almost every $x$ and every
sufficiently small $r$, there exists a $\delta r$ separated subset of
$B(x,r)$ of size $n$.  In particular, if the metric measure space is
doubling, there exists a bound on the total number of independent
Alberti representations.  When combined with the results from Section
\ref{sec:reps} that produce many independent Alberti representations,
such a bound imposes a derivative of every Lipschitz function at
almost every point.  In particular, we characterise Lipschitz
differentiability spaces as those doubling metric measure spaces in
which certain subsets have measure zero (see Theorem
\ref{thm:doublingchar}).

In Section \ref{sec:arbreps} we prove the existence of many additional
Alberti
representations of a Lipschitz differentiability space.  In fact, for
any chart $\varphi$ in an $n$-dimensional differentiability space and
any cone $C\subset\mathbb R^n$, we show the existence of an Alberti
representation in the direction of $C$.

In Section \ref{sec:cheeger} we relate our characterisations to the
previous works of Cheeger and Keith.

Throughout this paper we will denote by $(X,d)$ a complete,
separable metric space and by $(X,d,\mu)$ the metric measure space
obtained by equipping $(X,d)$ with a finite Borel regular measure
$\mu$, so that $\mu$ is Radon.  However, by applying the results
proved here to any compact subset, the main results hold for any
metric space equipped with a Radon measure (that must necessarily be
doubling).  Further, since a derivative is not defined at an isolated
point, we also suppose that the set of isolated points has measure
zero.

For $(X,d)$ a metric space, $Y\subset X$, $x\in X$ and $r>0$, we
denote by $B(x,r)$ and $\overline B(x,r)$ the open and closed balls
centred at $x$ of radius $r$ and by $B(Y,r)$ and $\overline B(Y,r)$
the open and closed $r$-neighbourhood of $Y$ of radius $r$.  For a
function $f\colon X\to\mathbb R$ we write $\Lip(f,x)$ for the
pointwise Lipschitz constant of $f$ at $x$ given by
\[\Lip(f,x):= \limsup_{r\to  0}\left\{\frac{|f(x)-f(y)|}{d(x,y)} : 0<d(x,y)<r\right\}.\]
For two measures $\mu$ and $\nu$ on $X$ we say that
$\mu\ll\nu$ if $\mu$ is absolutely continuous with respect to $\nu$.



\section{Alberti representations}\label{sec:albertireps}
This section first defines an Alberti representation and
proves some of their basic properties.  We then demonstrate how the
existence of an Alberti representation provides a metric measure space
with a notion of almost everywhere differentiability of Lipschitz
functions.

\subsection{Basic definitions and properties}
Before defining an Alberti representation we define the set of curves
that form a representation.

\begin{definition}\label{defn:curvespace}
Let $(X,d)$ be a metric space.  We denote by $\Gamma(X)$ the set of
bi-Lipschitz functions
\[\gamma\colon \dom\gamma\to X\]
with $\dom\gamma\subset \mathbb R$ non-empty and compact.  For a
function $f\colon Y\to Z$ we denote
\[\graph f=\{(y,z)\in Y\times Z:f(y)=z\}.\]
Finally, for $Y$ non-empty and compact, $\graph f$ is non-empty and
compact and so we may equip $\Gamma(X)$ with the Hausdorff metric:
\[d(\gamma,\gamma')=\text{Hausdorff}(\graph\gamma,\graph\gamma').\]
We also define $\Pi(X)$ to be the set of all $\gamma\in\Gamma(X)$ with
$\dom\gamma$ a closed interval.
\end{definition}

We will refer to elements of $\Gamma(X)$ as \emph{Lipschitz curves}.
Note that, as in most metric spaces, $\Pi(X)$ may only consist of
the Lipschitz curves whose domain is a single point.

Suppose that $(X,d)$ is a metric space, $f\colon X\to\mathbb R^n$ is
Lipschitz and $\gamma\in\Gamma(X)$.  Then $f\circ\gamma\colon
\dom\gamma\subset \mathbb R\to\mathbb R^n$ is Lipschitz and so may be
extended to a Lipschitz function $F$ defined on an interval.
Therefore, by Lebesgue's theorem, $F$ is differentiable almost
everywhere.  However, at any non-isolated point $t_0$ of $\dom\gamma$,
if $F'(t_0)$ exists then it's value is determined by the value of
$f\circ\gamma$ near to $t_0$, independently of the extension $F$.  We
may therefore define the derivative of $f\circ\gamma$ almost
everywhere in $\dom\gamma$ to be equal to the derivative of any Lipschitz
extension to an interval.

With these constructions, we define an Alberti representation as follows.

\begin{definition}\label{defn:albrep}
Let $(X,d,\mu)$ be a metric measure space, $\mathbb P$ a Borel
probability measure on $\Gamma(X)$ and, for each $\gamma\in\Gamma(X)$,
let $\mu_\gamma$ be a Borel measure on $X$ that is absolutely
continuous with respect to $\mathcal H^1\llcorner \im\gamma$, the
restriction of $\mathcal H^1$ to $\im\gamma$.  For measurable
$A\subset X$ we say that $(\mathbb P,\{\mu_\gamma\})$ is an
Alberti representation of $\mu\llcorner A$ if, for every Borel
$Y\subset A$, $\gamma\mapsto \mu_\gamma(Y)$ is Borel measurable and
\[\mu(Y)=\int_{\Gamma(X)}\mu_\gamma(Y)\text{d}\mathbb P(\gamma).\]

Given an Alberti representation $\mathcal A=(\mathbb
P,\{\mu_\gamma\})$ we will write ``almost every $\gamma\in\mathcal
A$'' to mean ``$\mathbb P$-almost every $\gamma\in\Gamma(X)$'' and
given a curve $\gamma\in\Gamma(X)$ we write ``almost every
$t\in\dom\gamma$'' to mean ``$\gamma^{-1}_{\#}(\mu_\gamma)$-almost every
$t\in\dom\gamma$''.
\end{definition}

First observe that, if $\mu\llcorner A$ has an Alberti representation
$\mathcal A=(\mathbb P, \{\mu_\gamma\})$ and $f\colon X\to\mathbb R$ is a positive simple function,
then
\[\int_A f\text{d}\mu = \int_{\Gamma(X)}\int_{\operatorname{Im}\gamma} f\text{d}\mu_\gamma \text{d}\mathbb P(\gamma).\]
Therefore, by the monotone convergence theorem, this formula holds for
any positive Borel function $f\colon X\to \mathbb R$.  Also note that,
if $N\subset X$ is $\mu$-null, then $\mathcal A$ is also an Alberti
representation of $\mu\llcorner A\cup N$.  Finally, given a measurable
set $B\subset A$, we may define $\mathcal A\llcorner B$, the
restriction of $\mathcal A$ to $B$, given by $(\mathbb
P,\{\mu_\gamma\llcorner B\})$.

We may also form new Alberti representations from existing representations.

\begin{lemma}\label{lem:albertinikodym}
Let $(X,d,\mu)$ be a metric measure space and $A\subset X$ measurable.
Suppose that there exists a finite measure $\mathbb M$ on $\Gamma(X)$
and for each $\gamma\in\Gamma(X)$ a measure $\nu_\gamma\ll \mathcal
H^1\llcorner \im\gamma$ such that $\mu\llcorner A$ is absolutely continuous with respect to the Borel measure
\[\nu(B)=\int_{\Gamma(X)}\nu_\gamma (B) \text{d}\mathbb M\]
and such that $\nu$ is finite.  Then $\mu\llcorner A$ has an Alberti representation.
\end{lemma}

\begin{proof}
By dividing by $\mathbb M(\Gamma(X))$ we may suppose that $\mathbb M$
is a probability measure.  Since $\mu$ and $\nu$ are both finite, by
the Radon-Nikodym Theorem there exists a Borel measurable $F\colon
A\to\mathbb R$ such that, for each Borel $B\subset A$,
\[\mu(B)=\int_B F\text{d}\nu=\int_{\Gamma(X)}\int_\gamma \mathbf{1}_B F\text{d}\nu_\gamma\text{d}\mathbb M.\]
Therefore, if we define $\text{d}\mu_\gamma=F \text{d}\nu_\gamma$,
\[\mu(B)=\int_{\Gamma(X)}\mu_\gamma(B)\text{d}\mathbb P\]
and so $\mu\llcorner A$ has an Alberti representation.
\end{proof}

\begin{lemma}\label{lem:sumreps}
Let $(X,d,\mu)$ be a metric measure space and for each $k\in\mathbb N$ let
$A_k\subset X$ be measurable such that $\mu\llcorner A_k$ has an Alberti
representation.  Then $\mu\llcorner \cup_k A_k$ has an Alberti
representation.
\end{lemma}

\begin{proof}
For each $k\in\mathbb N$ let $\mathcal A_k=(\mathbb
P_k,\{\mu_{\gamma,k}\})$ be an Alberti representation of $\mu\llcorner
A_k$ and set $\mathbb P = \sum_k 2^{-k}\mathbb P_k$, a Borel
probability measure on $\Gamma(X)$.  For each $\gamma\in\Gamma(X)$ and
each $k\in\mathbb N$ let $F_{\gamma,k}$ be the Radon-Nikodym
derivative of $\gamma^{-1}(\mu_\gamma)$ with respect to Lebesgue
measure and set
\[S_\gamma=\{t\in\dom\gamma : \exists\ k\in\mathbb
N,\ F_{\gamma,k}(t)>0\}.
\]
Then $S_\gamma$ is a
bounded subset of $\mathbb R$ of positive measure and so we may
define the measure
\[\nu_\gamma(Y) := \begin{cases}\mathcal H^1\llcorner
\gamma(S_\gamma)/\mathcal H^1(\gamma(S_\gamma)) & \text{if } \mathcal
H^1(\gamma(S_\gamma))>0\\
0 & \text{otherwise}\end{cases}
\]
on $X$.  Further, we
may define the finite measure $\nu$ on $X$ given by
\[\nu(Y) = \int_{\Gamma(X)} \nu_\gamma(Y) \text{d}\mathbb P\]
for each Borel $Y\subset X$.

Now let $Y\subset X$ be Borel with $\nu(Y)=0$ and let $k\in\mathbb N$.
Then for $\mathbb P_k$-a.e.\ $\gamma\in\Gamma(X)$ we have
$\nu_\gamma(Y)=0$.  However, by definition, $\mu_{\gamma,k} \llcorner
S_\gamma \ll \nu_\gamma$ for every $\gamma\in\Gamma(X)$.  Therefore,
for $\mathbb P_k$-a.e.\ $\gamma\in\Gamma(X)$, $\mu_{\gamma,k}(Y)=0$.
Since $\mu\llcorner A_k$ has the above Alberti representation, we must
have $\mu(A_k\cap Y)=0$.  In particular $\mu(Y\cap \cup_k A_k)=0$ and
so $\mu\llcorner \cup_k A_k$ is absolutely continuous with respect to
$\nu$.  By applying the previous Lemma we obtain the required Alberti
representation $(\mathbb P,\{\mu_\gamma\})$ of $\mu\llcorner\cup_k
A_k$.
\end{proof}

The previous two Lemmas preserve many of the properties of the Alberti
representations in their hypotheses.  More precisely, suppose under
the hypotheses of Lemma \ref{lem:sumreps} that $\tilde \Gamma \subset
\Gamma(X)$ is a set of full $\mathbb P_k$-measure for each
$k\in\mathbb N$ and that, for each $\gamma\in \tilde\Gamma$,
$\tilde\gamma \subset \im\gamma$ is a set of full
$\mu_{\gamma,k}$-measure, for each $k\in\mathbb N$.  Then
$\tilde\Gamma$ is a set of full $\mathbb P$-measure and for each
$\gamma\in \tilde\Gamma$, $\tilde\gamma$ is a set of full
$\mu_\gamma$-measure.  We will make particular use of this fact in the
following form.  Suppose that $f\colon X\to\mathbb R$ is Lipschitz
such that, for each $k\in\mathbb N$, almost every $\gamma\in\mathcal
A_k$ and almost every $t\in\dom\gamma$, $(f\circ\gamma)'(t)>0$.  Then
for almost every $\gamma\in\mathcal A$ and almost every
$t\in\dom\gamma$, $(f\circ\gamma)'(t)>0$.

In the definition of an Alberti representation we only require the
representation to hold for all Borel subsets of $X$.  However, as we will
now see, it is easy to extend this representation to all
$\mu$-measurable subsets of $X$.

\begin{lemma}\label{lem:measurablerep}
Let $(X,d,\mu)$ be a metric measure space and $A\subset X$ measurable such
that $\mu\llcorner A$ has an Alberti representation $(\mathbb
P,\{\mu_\gamma\})$.  Then for every $\mu$-measurable $Y\subset A$,
$\gamma\mapsto \mu_\gamma (Y)$ is $\mathbb P$-measurable and
\[\mu(Y)=\int_{\Gamma(X)}\mu_\gamma(Y)\text{d}\mathbb P.\]
\end{lemma}

\begin{proof}Let $Y\subset A$ be $\mu$-measurable and $B\subset Y\subset C$ be Borel with $\mu(B)=\mu(C)$.  Then for every $\gamma\in\Gamma(X)$,
\[\mu_\gamma(B)\leq\mu_\gamma(Y)\leq\mu_\gamma(C).\]
However, since $\mu(B)$ and $\mu(C)$ are given by the
Alberti representation, we know that $\mu_\gamma(B)=\mu_\gamma(C)$ for
almost every $\gamma\in\mathcal A$.  Therefore
$\gamma\mapsto\mu_\gamma(Y)$ is $\mathbb P$-measurable and
\begin{align*}\mu(Y)=\mu(B)&=\int_{\Gamma(X)}\mu_\gamma(B)\text{d}\mathbb P\\
&\leq \int_{\Gamma(X)}\mu_\gamma(Y)\text{d}\mathbb P\\
&\leq \int_{\Gamma(X)}\mu_\gamma(C)\text{d}\mathbb P=\mu(C)=\mu(Y).\end{align*}
Thus $\mu(Y)$ is given by the Alberti representation.
\end{proof}

\subsection{Relationship with differentiability}
We now demonstrate how an Alberti representation provides a metric
measure space with a notion of almost everywhere differentiability of
Lipschitz functions.

\begin{definition}\label{def:curvespace}
For a metric space $(X,d)$ we define $H(X)$ to be
the collection of non-empty compact subsets of $\mathbb R\times X$
with the Hausdorff metric, so that $H(X)$ is complete and separable.  We
also identify $\Gamma(X)$ with it's isometric image in $H(X)$ via
$\gamma\mapsto\graph\gamma$ and set
\[A(X)=\{(x,\gamma)\in X\times\Gamma(X):\exists\ t\in\dom\gamma,\ \gamma(t)=x\}.\]
Finally, for any $K\subset X$, we define the set
\[DP(K):=\{(x,\gamma)\in A(X): \gamma^{-1}(x) \text{ is a density
  point of } \gamma^{-1}(K)\}.\]
\end{definition}

\begin{lemma}\label{lem:gammaborel}
Let $(X,d)$ be a metric space and $L>0$.  Then the set of all
$L$-bi-Lipschitz $\gamma\in\Gamma(X)$ is a closed subset of $H(X)$, and
$\Gamma(X)$ is a Borel subset of $H(X)$.  Further, $A(X)$ is a Borel
subset of $X\times H(X)$.  Finally, for any
$a_1=(\gamma_1(t_1),\gamma_1)$ and $a_2=(\gamma_2(t_2),\gamma_2)$ in
$A(X)$,
\[|t_1-t_2|\leq (2\min_{i=1, 2}\operatorname{biLip}\gamma_i+1)d(a_1,a_2)).\]
\end{lemma}

\begin{proof}
For any $L>0$ suppose that
$\gamma_m\in\Gamma(X)$ are $L$-bi-Lipschitz and $\gamma_m\to\gamma$,
for some $\gamma\in H(X)$.  We will show that $\gamma$ is the graph of
some $L$-bi-Lipschitz function.

Indeed, let $(s,x)$ and $(t,y)$ belong to $\gamma$ and, for each
$m\in\mathbb N$, $(s_m,x_m)$ and $(t_m,y_m)\in\gamma_m$ with
$(s_m,x_m)\to (s,m)$ and $(t_m,y_m)\to (t,x)$.  Then since each
$\gamma_m$ is $L$-bi-Lipschitz,
\[d(x,y)=\lim_{m\to\infty}d(\gamma_m(s_m),\gamma_m(t_m))\leq \lim_{m\to\infty} L|s_m-t_m| = L|s-t|\]
and similarly
\[d(x,y)\geq |s-t|/L.\]
Therefore $\gamma$ is the graph of some $L$-bi-Lipschitz function
$\gamma\colon\dom\gamma\to X$.  In particular, the set of
$L$-bi-Lipschitz elements of $\Gamma(X)$ is a closed subset of $H(X)$.  By
taking a union over $L\in\mathbb N$ we conclude that $\Gamma(X)$ is a
Borel subset of $H(X)$.

Now suppose that $(x,\gamma)\in X\times \Gamma(X)$ does not belong to
$A(X)$.  Then since $\im\gamma$ is a compact set there exists a
$\delta>0$ such that $d(x,\im\gamma)>\delta$.  In particular
$B((x,\gamma),\delta/2)$ is disjoint from $A(X)$ and so $A(X)$ is a
closed subset of $X\times\Gamma(X)$.

Finally, given $a_1=(\gamma_1(t_1),\gamma_1),a_2=(\gamma_2(t_2),\gamma_2)\in A(X)$, suppose that
\[L=\operatorname{biLip}\gamma_1\leq \operatorname{biLip}\gamma_2.\]
There exists a $t\in\dom\gamma_1$ with
\[\max\{|t_2-t|, d(\gamma_1(t),\gamma_2(t_2))\}\leq d(a_1,a_2)\]
and so
\begin{align*}|t_1-t_2| &\leq |t_1-t|+|t-t_2|\\
&\leq L d(\gamma_1(t_1),\gamma_1(t))+ |t-t_2|\\
&\leq Ld(\gamma_1(t_1),\gamma_2(t_2))+Ld(\gamma_2(t_2),\gamma_1(t_1))+|t-t_2|\\
&\leq (2L+1)d(a_1,a_2)\end{align*}
as required.
\end{proof}

We will describe an Alberti representation using a Borel set $B\subset
A(X)$ whose projection into $\Gamma(X)$ is a set of full measure.  We
now demonstrate the particular Borel sets that we will use.

\begin{lemma}\label{lem:measurability}
Let $(X,d,\mu)$ be a metric measure space.
\begin{enumerate}\item For $g\colon X\to\mathbb R$ continuous, the
function $f \colon A(X)\to\mathbb R\cup \{\infty\}$ defined by
\[f(x,\gamma)=\begin{cases}(g\circ\gamma)'(\gamma^{-1}(x)) & \text{if it exists}\\
\infty & \text{otherwise}\end{cases}\]
is Borel measurable.
\item \label{eq:compactdp} For any compact $K\subset X$, $DP(K)$ is a Borel subset of $A(X)$.
\end{enumerate}
\end{lemma}

\begin{proof}
Let $\delta,\epsilon>0$ and $q\in\mathbb R$.  Then the set of
$(\gamma(t_0),\gamma)\in A(X)$ with
\[|g(\gamma(t))-g(\gamma(t_0))-q(t-t_0)|\leq\epsilon |t-t_0|\]
for all $t\in \dom\gamma\cap B(t_0,\delta)$, is closed.  Therefore,
after taking suitable countable intersections and unions of such sets,
the set where $(g\circ\gamma)'(\gamma^{-1}(x))$ exists and
belongs to some open subset of $\mathbb R$ is Borel.  Thus the set of
points where the derivative does not exist is Borel as is the set of
points where the derivative belongs to some given Borel set.

To prove the second statement, let $\gamma_m\to\gamma$ in $H(X)$,
$t\in\mathbb R$ and $r>0$.  Then $\dom\gamma_m\to\dom\gamma$ in the
Hausdorff metric.  In particular, for every $\epsilon>0$,
$\dom\gamma_m\subset \overline B(\dom\gamma,\epsilon)$ for
sufficiently large $m$.  Therefore
\begin{align}\nonumber\limsup_{m\to\infty}\mathcal L^1(\overline B(t,r)\cap\dom\gamma_m)&\leq \lim_{\epsilon\to 0}\mathcal L^1(\overline B(t,r)\cap \overline B(\dom\gamma,\epsilon))\\
\label{eq:albreps1}&=\mathcal L^1(\overline B(t,r)\cap\dom\gamma)
\end{align}
since $\dom\gamma$ is compact.

Now let $K\subset X$ be compact.  We first show that
\[\gamma\mapsto
\mathcal L^1(\overline B(t,r)\cap\gamma^{-1}(K))
\]
is upper semicontinuous on $\Gamma(X)$.  For this let
$\gamma_m\to\gamma$ in $H(X)$ such that each $\gamma_m^{-1}(K)$ is
non-empty and for every $m\in\mathbb N$ define
$\dom\tilde\gamma_m=\gamma_m^{-1}(K)$ and $\tilde\gamma_m$ to be the
restriction of $\gamma_m$ to $\dom\tilde\gamma_m$.  Then since
$\gamma_m\to\gamma$ we know that all of the $\dom\tilde\gamma_m$
belong to some compact subset $A$ of $\mathbb R$.  In particular each
$\tilde\gamma_m$ is contained within $A\times K$, a compact set,
inside which the Hausdorff metric on its non-empty subsets is also
compact.  Therefore, for any sequence $m(k)\to\infty$ there exists a
subsequence $m(k_i)\to\infty$ such that the $\gamma_{m(k_i)}$ converge
to some non-empty compact $B\subset A\times K$.

However, since $\gamma_m\to\gamma$, we know that $B$ must be a subset
of $\gamma$ restricted to $K$ and so, by equation \eqref{eq:albreps1},
\[\limsup_{i\to\infty}\mathcal L^1(\overline B(t,r)\cap\dom\tilde\gamma_{m(k_i)})\leq \mathcal L^1(\overline B(t,r)\cap\gamma^{-1}(K)).\]
This is true for any sequence $m(k)\to\infty$ and so
\[\gamma\mapsto\mathcal L^1(\overline B(t,r)\cap\gamma^{-1}(K))\]
is upper semicontinuous on $\Gamma(X)$.

Further, if $(\gamma_m(t_m),\gamma_m)\to (\gamma(t),\gamma)$ in
$A(X)$, then by Lemma \ref{lem:gammaborel} we know that $|t_m-t|\to 0$ and
so
\begin{align*}\limsup_{m\to\infty}\mathcal L^1(\overline B(t_m,r)\cap \gamma_m^{-1}(K))&\leq \limsup_{m\to\infty}\mathcal L^1(\overline B(t,r)\cap\gamma_m^{-1}(K))+2|t_m-t|\\
&\leq \mathcal L^1(\overline B(t,r)\cap\gamma^{-1}(K)). \end{align*}
In particular
\[(x,\gamma)\mapsto \mathcal L^1(B(\gamma^{-1}(x),r)\cap \gamma^{-1}(K))\]
is upper semicontinuous on $A(X)$.  By taking a suitable countable
collection of intersections and unions we see that $DP(K)$ is a Borel
subset of $A(X)$.
\end{proof}

By combining these results we now show, for almost every point $x$ in
a metric measure space with an Alberti representation, the existence
of a Lipschitz curve of which $x$ is a density point.  Moreover, such
a Lipschitz curve may be taken from a set of full measure, with respect
to the integrating measure of the Alberti representation.

\begin{proposition}\label{prop:curvefullmeas}
Let $(X,d,\mu)$ be a metric measure space and $M\subset X$ measurable
such that $\mu\llcorner M$ has an Alberti representation $\mathcal A$.
Suppose that $(Y,\rho)$ is a complete, separable metric space,
$B\subset Y$ is Borel and $f\colon A(X)\to Y$ is a Borel function such
that, for almost every $\gamma\in\mathcal A$ and almost every
$t\in\dom\gamma$, $f(\gamma(t),\gamma)\in B$.  Then the set
\[P(M):=\{x\in M:\exists\ \gamma\in\Gamma(X),\ (x,\gamma)\in DP(M),\ f(x,\gamma)\in B\}
\]
is a set of full measure in $M$ and for each $x\in P(M)$ there exists a
$\gamma^x\in\Gamma(X)$ that satisfies the conditions given in the
definition of $P(M)$ such that $x\mapsto f(x,\gamma^x)$ is measurable.
\end{proposition}

\begin{proof}
We first prove the statement under the additional hypotheses that $M$
is compact.  If so, by Lemma \ref{lem:measurability}
\eqref{eq:compactdp}, the set
\[G:=\{(f(x,\gamma),x,\gamma)\in B\times A(X):x\in M,\ (x,\gamma)\in DP(M)\}\]
is the graph of a Borel function restricted to a Borel set, and so is
Borel.  In particular, it's projection onto $M$ is a Suslin set and so
is measurable.  This projection equals $P(M)$.  Further, the
projection of $G$ onto $B\times M$ given by
\[\{(y,x)\in B\times M:\exists\ \gamma\in\Gamma(X),\ (x,\gamma)\in DP(M),\ y=f(x,\gamma)\in B\}\]
is also a Suslin set.  This set is the graph of a function defined on
$M$ and so, by the Jankov-von Neumann Selection Theorem (see
\cite{classicaldescset}, Theorem 18.1), for every $x\in P(M)$ there
exists a $\gamma^x\in \Gamma(X)$ such that $(x,\gamma^x)\in DP(M)$,
$f(x,\gamma^x)\in B$ and such that $x\mapsto f(x,\gamma^x)$ is
measurable.

Suppose that $\mathcal A=(\mathbb P,\{\mu_\gamma\})$ and let us
consider $C=M\setminus P(M)$, a measurable set.  Then for any
$\gamma\in\Gamma(X)$ and $t\in\dom\gamma$ with $\gamma(t)\in C$, since
$(\gamma(t),\gamma)\in A(X)$, we must have either
$f(\gamma(t),\gamma)\not\in B$ or $t$ not a density point of
$\gamma^{-1}(C)$.  Therefore, from our hypotheses on $\mathcal A$,
there exists a $\mathbb P$-null set $N$ such that either $\gamma \in N$
or $\mathcal L^1(\gamma^{-1}(C))=0$.  By Lemma \ref{lem:measurablerep},
\[\mu(C)=\int_{\Gamma(X)\setminus N}\mu_\gamma(C)\text{d}\mathbb P +
\int_{N}\mu_\gamma(C)\text{d}\mathbb P=0\]
so that $P(M)$ is a set of full measure in $M$.

Now suppose that $M\subset X$ is measurable and let $N$ be a
$\mu$-null set such that $M\setminus N$ is a countable union of compact
sets $K_i$.  Then by the above, $\cup_i P(K_i)$ is a set of full
measure in $M$.  Moreover, for every $i\in\mathbb N$, $K_i\subset M$
and so $P(K_i)\subset P(M)$.  In particular $\mu(M\setminus P(M))=0$
so that such a $\gamma^x$ exists for almost every $x\in M$.
\end{proof}

In particular we may use this Proposition to find a \emph{partial
  derivative} of a Lipschitz function at almost every point of a
metric measure space with an Alberti representation.  In fact, we may
find a \emph{gradient} of partial derivatives, whenever a metric
measure space has many \emph{independent} Alberti representations,
each distinguished in the following way.

\begin{definition}\label{defn:phidir}
For $w\in\mathbb S^{n-1}$ and $0<\theta<1$, define the \emph{cone of
  width $\theta$ centred on $w$} to be the set
\[C(w,\theta)=\{v\in \mathbb R^n: v\cdot w\geq (1-\theta)\|v\|\}\]
and the \emph{open cone of width $\theta$ centred on $w$} to be
\[C^\circ(w,\theta)=\{v\in\mathbb R^n: v\cdot w>(1-\theta)\|v\|\}.\]
Also, for a cone $C$ we denote the open cone with the same centre and
width by $C^\circ$.

Now let $(X,d,\mu)$ be a metric measure space and
$\varphi\colon\ X\to\mathbb R^n$ Lipschitz.  For a cone $C\subset
\mathbb R^n$ we say that a Lipschitz curve
$\gamma\in\Gamma(X)$ is in the
\emph{$\varphi$-direction of $C$} if, for almost every
$t\in\dom\gamma$,
\[(\varphi\circ\gamma)'(t)\in C\setminus\{0\}.\]
Further, we say that an Alberti representation $\mathcal A$ is in the
\emph{$\varphi$-direction of $C$} if almost every $\gamma\in\mathcal
A$ is in the $\varphi$-direction of $C$.

Finally, we say that closed cones $C_1,\ldots,C_m\subset\mathbb R^n$
are \emph{independent} if, for any choice of $v_i\in
C_i\setminus\{0\}$, the $v_i$ are linearly independent and that a
collection $\mathcal A_1,\ldots,\mathcal A_m$ of Alberti
representations is \emph{$\varphi$-independent} if there exists
independent cones $C_1,\ldots,C_m$ such that each $\mathcal A_i$ is in
the $\varphi$-direction of $C_i$.
\end{definition}

\begin{definition}\label{def:gradient}
Let $(X,d,\mu)$ be a metric measure space, $x_0\in X$ and $f\colon
X\to\mathbb R$ and $\varphi\colon X\to\mathbb R^n$ Lipschitz.  Suppose
that $\gamma_1,\ldots,\gamma_n\in\Gamma(X)$ such that each
$\gamma_i^{-1}(x_0)=0$ is a density point of $\dom\gamma_i$ and that the
$(\varphi\circ\gamma_i)'(0)$ exist and form a linearly independent
set.  We define the \emph{gradient of $f$ at $x_0$ with respect to $\varphi$
  and $\gamma_1,\ldots,\gamma_n$} to be the unique $\nabla
f(x_0)\in\mathbb R^n$ such that
\[(f\circ\gamma_i)'(0)= \nabla f(x_0)\cdot
(\varphi\circ\gamma_i)'(0).\]

Further, we say that $\nabla f(x_0)\in\mathbb R^n$ is \emph{a gradient
  of $f$ at $x_0$ with respect to $\varphi$} if there exist such
$\gamma_1,\ldots,\gamma_n\in\Gamma(X)$ such that $\nabla f(x_0)$
is the gradient of $f$ at $x_0$ with respect to $\varphi$ and
$\gamma_1,\ldots,\gamma_n$.
\end{definition}

\begin{remark}A very easy construction of a Lipschitz function on the
  plane with both partial derivatives equal to zero and a non-zero
  directional derivative at the origin shows that a gradient with
  respect to a fixed $\varphi$ need not be unique.
\end{remark}

We obtain the following Corollary of Proposition \ref{prop:curvefullmeas}.

\begin{corollary}
Let $(X,d,\mu)$ be a metric measure space and $\varphi\colon
X\to\mathbb R^n$ Lipschitz such that $\mu$ has $n$
$\varphi$-independent Alberti representations.  Then for any Lipschitz
$f\colon X\to\mathbb R$ there exists a gradient $\nabla f$ of $f$
almost everywhere with respect to $\varphi$.  Further, this gradient
may be chosen such that $x_0\mapsto \nabla f(x_0)$ is measurable.
\end{corollary}


\section{A first analysis of Lipschitz differentiability spaces}\label{sec:firstanalysis}
We begin by recalling the notion of differentiability in metric spaces
introduced by Cheeger \cite{cheeger-diff} and Keith \cite{keith}.

\begin{definition}\label{def:chart}
Let $(X,d)$ be a metric space and $n\in\mathbb N$.  We say that a
Borel set $U\subset X$ and a Lipschitz function $\varphi\colon
X\to\mathbb R^n$ form a \emph{chart of dimension $n$}, $(U,\varphi)$
and that a function $f\colon X\to\mathbb R$ is \emph{differentiable at
  $x_0\in U$} with respect to $(U,\varphi)$ if there exists a unique
$Df(x_0)\in\mathbb R^n$ (the \emph{derivative of $f$ at $x_0$}) such
that
\[\limsup_{X\ni x\to x_0}\frac{|f(x)-f(x_0)-Df(x_0)\cdot(\varphi(x)-\varphi(x_0))|}{d(x,x_0)}=0.\]

Further, we say that a metric measure space $(X,d,\mu)$ is a
\emph{Lipschitz differentiability space} if there exists a countable
decomposition of $X$ into charts such that any Lipschitz function
$f\colon X\to\mathbb R$ is differentiable at almost every point of
every chart.
\end{definition}

Whenever it will not cause confusion, we will say ``a chart in a
Lipschitz differentiability space'' to mean a chart in a Lipschitz
differentiability space with respect to which every real valued
Lipschitz function is differentiable almost everywhere.

For a survey on the existing theory of Lipschitz differentiability
spaces, the reader is referred to the primer written by Kleiner and
Mackay, \cite{kleinermackay-diffstructures}.  (Note that, in this paper and the papers
of Cheeger and Keith, a Lipschitz differentiability space is referred
to as a \emph{metric measure space that admits a measurable
  differentiable structure}.)

\begin{remark}
One may wonder about the notion of a \emph{zero dimensional chart},
where every Lipschitz function has derivative zero almost everywhere.
In Corollary \ref{cor:zerodim} we see that such a chart must have
measure zero.
\end{remark}

We first give a characterisation of the uniqueness of derivatives in a
chart that will lead us to necessary conditions for a function to be
differentiable at a point.  This characterisation first appeared in
\cite{porousdiff}.

\begin{lemma}[\cite{porousdiff}, Lemma 2.1]\label{lem:discuniq}
Let $(U,\varphi)$ be an $n$-dimensional chart in a metric measure space $(X,d,\mu)$ and $x_0\in U$.  The following are equivalent:
\begin{enumerate}\item \label{item:uniqueness1}There exists a $\lambda>0$ and $X\ni x_m\to x_0$ such that, for any $v\in\mathbb S^{n-1}$,
\begin{equation}\label{eq:uniqueness}\liminf_{m\to\infty}\max_{0\leq
    i<n}\frac{|(\varphi(x_{mn+i})-\varphi(x_0))\cdot
    v|}{d(x_{mn+i},x_0)}\geq \lambda.\end{equation}
\item\label{item:uniqueness2} There exists a $\lambda>0$ such that, for any $v\in\mathbb
  S^{n-1}$,
\[\limsup_{x\to x_0} \frac{|(\varphi(x)-\varphi(x_0))\cdot
  v|}{d(x,x_0)}\geq \lambda.\]
\item\label{item:uniqueness3} For any $f\colon X\to\mathbb R$, if there exists a $Df(x_0)\in\mathbb R^n$ such that
\[\limsup_{x\to x_0}\frac{|f(x)-f(x_0)-Df(x_0)\cdot(\varphi(x)-\varphi(x_0))|}{d(x,x_0)}=0,\]
then it is unique.
\end{enumerate}
In particular, the fact that such a $Df(x_0)$ is unique depends only
upon the chart and is independent of $f$.
\end{lemma}

\begin{proof}
Equation \eqref{item:uniqueness1} implies \eqref{item:uniqueness2}.
Now suppose that \eqref{item:uniqueness2} holds and for a function
$f\colon X\to \mathbb R$ there exist $Df(x_0)$ and $Df'(x_0) \in
\mathbb R^n$ that satisfy the hypotheses of \eqref{item:uniqueness3}.
Then by the triangle inequality
\[\limsup_{x\to x_0}\frac{|(\varphi(x)-\varphi(x_0))\cdot (D-D')|}{d(x,x_0)}=0\]
and so $\|D-D'\|=0$.

Finally suppose that \eqref{item:uniqueness3} holds for some $f\colon
X\to\mathbb R$.  Then for any $v_0\in\mathbb S^{n-1}$,
\[\limsup_{x\to x_0}\frac{|(\varphi(x)-\varphi(x_0))\cdot v_0|}{d(x,x_0)}>0.\]
Therefore, there exists $x_m^0\to x_0$ and $w_0\in \overline
B(0,\Lip\varphi)$ such that $w_0\cdot v_0>0$ and
\[\lim_{m\to\infty}\frac{\varphi(x_m^0)-\varphi(x_0)}{d(x_m^0,x_0)}=w_0.\]
For each $1\leq i <n$ inductively choose $v_i\in\mathbb S^{n-1}$ such that $v_i\cdot
w_j=0$ for each $0\leq j <i$ and let $x_m^i\to x_0$ and $w_i\in \overline
B(0,\Lip\varphi)$ such that $w_i\cdot v_i>0$ and
\[\lim_{m\to\infty}\frac{\varphi(x_m^i)-\varphi(x_0)}{d(x_m^i,x_0)}=w_i.\]
Then for any $\alpha_1,\ldots,\alpha_n\in\mathbb R$, suppose that
$\alpha_0 w_0+\ldots + \alpha_{n-1} w_{n-1}=0$.  By taking the
inner product with $w_{n-1}$ we see that $\alpha_{n-1}=0$.  Repeating
we see that $\alpha_i=0$ for each $0\leq i <n$ and so the $w_i$ form a
basis of $\mathbb R^n$.  In particular, there exists a $\lambda>0$ such
that, for each $v\in\mathbb S^{n-1}$, there exists a $0\leq j<n$ with
$|w_j\cdot v|\geq \lambda$.

Therefore, if we set $x_{mn+i}=x_m^i$ for each $m\in\mathbb N$ and
$0\leq i<n$, for any $v\in\mathbb S^{n-1}$ there exists a $0\leq j<n$
such that
\[\liminf_{m\to\infty}\max_{0\leq
    i<n}\frac{|(\varphi(x_{mn+i})-\varphi(x_0))\cdot
    v|}{d(x_{mn+i},x_0)}\geq
\lim_{m\to\infty}\frac{|(\varphi(x_m^j)-\varphi(x_0))\cdot
  v|}{d(x_m^j,x_0)}\geq \lambda\]
as required.
\end{proof}

Using this we may give a necessary condition for a function to be
differentiable at a point in a chart.

\begin{lemma}\label{lem:nondiffcond}
Let $(U,\varphi)$ be an $n$-dimensional chart in a metric space
$(X,d)$, $x_0\in U$ and $x_m\to x_0$ satisfying \eqref{eq:uniqueness}
for some $\lambda>0$.  Then for any $f\colon
X\to\mathbb R$ that is differentiable at $x_0$ we have
\[\|Df(x_0)\|\leq \Lip(f,x_0)/\lambda\]
and
\[\frac{\lambda}{\Lip\varphi} \Lip(f,x_0) \leq \liminf_{m\to\infty}\max_{0\leq i<n}\frac{|f(x_{mn+i})-f(x_0)|}{d(x_{mn+i},x_0)}.\]
\end{lemma}

\begin{proof}
If $f\colon X\to\mathbb R$ is differentiable at $x_0\in X$ then by the
triangle inequality
\begin{align*}\liminf_{m\to\infty}\max_{0\leq i<n}\frac{|f(x_{mn+i})-f(x_0)|}{d(x_{mn+i},x_0)}&\geq \liminf_{m\to\infty}\max_{0\leq i<n} \frac{|(\varphi(x_{mn+i})-\varphi(x_0))\cdot D|}{d(x_{mn+i},x_0)}\\
&\geq \lambda \|D\|.\end{align*}
This proves the first inequality.  Secondly, by another application of the triangle inequality,
\begin{align*}\Lip(f,x_0)&=\limsup_{x\to x_0}\frac{|f(x)-f(x_0)|}{d(x,x_0)}\\
&\leq \|Df(x_0)\|\limsup_{x\to x_0}\frac{\|\varphi(x)-\varphi(x_0)\|}{d(x,x_0)}\\
&\leq \|Df(x_0)\|\Lip\varphi.\end{align*}
Combining this with the first inequality gives the required result.
\end{proof}

A very easy application of this Lemma gives the following Corollary on
singletons in a Lipschitz differentiability space.  This will be used
without any specific reference.

\begin{corollary}\label{cor:singletons}
Any singleton in a Lipschitz differentiability space has measure zero.
\end{corollary}

\begin{proof}
Let $(U,\varphi)$ be any $n$-dimensional chart in a Lipschitz
differentiability space $(X,d,\mu)$.  For almost any $x_0\in U$ let
$x_m\to x_0$ and $\lambda>0$ be obtained from Lemma
\ref{lem:discuniq} and by taking a suitable subsequence if necessary, we
may suppose that
\[\max_{0\leq i < n}d(x_{mn+i},x_0)\leq \min_{0\leq i <n}d(x_{(m+1)n+i},x_0)/2\]
for each $m\in\mathbb N$.  Then the 1-Lipschitz function
\[f(x)=\sum_{m\in 2\mathbb N} \max_{0\leq i <n}(d(x_{mn+i},x_0)/4 - d(x_{mn+i},x))^+\]
satisfies
\[\frac{f(x_{mn+i})-f(x_0)}{d(x_{mn+i},x_0)}=\begin{cases}1/4 & \text{if } m \text{ even}\\
0 & \text{if } m \text{ odd}\end{cases}\]
for each $0\leq i <n$.  In particular, the condition given in Lemma
\ref{lem:nondiffcond} does not hold at $x_0$ and so $f$ is not
differentiable at $x_0$, as required.
\end{proof}

When constructing a non-differentiable Lipschitz function, we will ensure that
\[\liminf_{m\to\infty}\max_{0\leq i<n}\frac{|f(x_{mn+i})-f(x_0)|}{d(x_{mn+i},x_0)}\]
is sufficiently small by first bounding $d(x_{mn+i},x_0)$ from below,
for a fixed $m$, all $0\leq i<n$ and all $x_0$ in a certain set of
large measure.  We now do this uniformly across the chart to simplify
the construction.

\begin{definition}\label{defn:structuredchart}
Let $(X,d)$ be a metric space and $\lambda>0$.  We say that $U\subset
X$ and a Lipschitz function $\varphi\colon X\to\mathbb R^n$ form a
\emph{$\lambda$-structured chart} of dimension $n$ if $U$ is compact
and there exists a $U'\subset U$ of full measure such that, for every
$R>0$ there exists an $r>0$ and for every $x_0\in U'$ points
$x_1,\ldots,,x_n\in U$ with each $r<d(x_i,x_0)<R$ and
\[\max_{1\leq i\leq n}\frac{|(\varphi(x_i)-\varphi(x_0))\cdot v|}{d(x_i,x_0)}\geq \lambda\]
for every $v\in\mathbb S^{n-1}$.

We say that $(U,\varphi)$ is a structured chart if it is a
$\lambda$-structured chart for some $\lambda>0$.
\end{definition}

Note that a structured chart is also a chart and so we may consider
the derivative of a real valued function with respect to a structured
chart.  The following Lemma shows that we may just consider structured
charts when working with a Lipschitz differentiability space.

\begin{lemma}\label{lem:decompstruct}
Let $(U,\varphi)$ be an $n$-dimensional chart in a metric measure space
$(X,d,\mu)$ such that
\[\Lip(v\cdot\varphi,x_0)>0\]
for every $v\in\mathbb S^{n-1}$ and almost every $x_0\in U$.  Then for
any $\epsilon>0$ there exists a $U'\subset U$ with
$\mu(U\setminus U')<\epsilon$ and a countable set $N\subset X$ such
that $(U'\cup N,\varphi)$ is a structured chart.

In particular, we may decompose any Lipschitz differentiability space
into a $\mu$-null set and a countable collection of structured charts
with respect to which any real valued Lipschitz function is
differentiable almost everywhere.
\end{lemma}

\begin{proof}
For $x_0,x_1,\ldots,x_n\in X$ let us denote by $P(x_0;x_1,\ldots,x_n)$
the property
\[\max_{1\leq i\leq n}\frac{|(\varphi(x_i)-\varphi(x_0))\cdot
  v|}{d(x_i,x_0)}\geq \lambda \text{ for every } v\in\mathbb S^{n-1}.
\]
Then for any $R,\lambda,r>0$, the set of $x_0\in U$ for which
there exist $x_1,\ldots,x_n\in X$ with $r< d(x_i,x_0)< R$ for
each $1\leq i \leq n$ and such that $P(x_0;x_1,\ldots,x_n)$ holds is an open
set.  Therefore, the set of those $x_0\in U$ that satisfy
\eqref{eq:uniqueness} for a given $\lambda>0$ is a Borel set.  In
particular, for any $\epsilon>0$ there exists a compact $U'\subset U$
with $\mu(U\setminus U')<\epsilon$, so that $(U',\varphi)$ is still a
chart, and a $\lambda>0$ such that every point of $U'$ satisfies
\eqref{eq:uniqueness} for $\lambda$.

Now let $R>0$ and $\lambda$ be as found above.  The function
\[x_0\mapsto \sup\left\{\min_{1\leq i\leq n} d(x_i,x_0):d(x_i,x_0)\leq
R \text{ and } P(x_0;x_1,\ldots,x_n)\right\}
\]
is positive and lower semicontinuous on $U'$.  Therefore, there exists
an $r>0$ such that this function is bounded below by $r$ on $U'$.  Let
$y^1,\ldots, y^M$ be a finite $\lambda r/2\Lip\varphi$-net of $U'$
and for each $1\leq j\leq M$ let $x_1^j,\ldots,x_n^j\in X$ with
$r\leq d(x_i^j,y^j)\leq R$ for each $1\leq i \leq n$ and such that
$P(y^j;x_1^j,\ldots,x_n^j)$ for each $1\leq j\leq M$.  We set
\[N_R=\{x_i^j: 1\leq i\leq n,\ 1\leq j\leq M\}.\]
Then for any $x_0\in U'$ and $v\in\mathbb S^{n-1}$, there exists a
$y^j$ with $d(x_0,y^j)<\lambda r/2\Lip\varphi$ and a $1\leq i
\leq n$ such that
\begin{align*}\max_{1\leq i\leq n}|(\varphi(x_i^j)-\varphi(x_0))\cdot v|
&\geq  |(\varphi(x_i^j)-\varphi(y^j))\cdot v|-\|\varphi(y^j)-\varphi(x_0)\|\\
&\geq \lambda d(x_i^j,y^j)-\lambda d(x_i^j,y^j)/2\\
&\geq \lambda d(x_i^j,x_0)/(2+\lambda/\Lip\varphi)\\
&:= \lambda' d(x_i^j,x_0).
\end{align*}
Therefore, if we define $N=\cup_i N_{1/i}$
and $V=U'\cup N$, for any $x_0\in U'$ there exists a sequence
$V\ni x_m\to x_0$ that satisfies \eqref{eq:uniqueness} for $\lambda'$.
Moreover, since each $N_{1/k}$ is a finite subset of
$B(U',\lambda\Lip\varphi/2k)$, for any sequence $x_m\in V$ either
there exists a constant subsequence or $d(x_m,U')\to 0$.  In
particular, $V$ is compact and so $(V,\varphi)$ is a structured chart.

Finally, in a Lipschitz differentiability space $N$ is a $\mu$-null
set and so $(V,\varphi)$ is a chart with respect to which any
Lipschitz function is differentiable almost everywhere in $V$.
\end{proof}

To conclude this section we highlight a key fact that will be used
when investigating Alberti representations in Lipschitz
differentiability spaces.  This result should be compared to the
concept of a gradient (see Definition \ref{def:gradient}).  We will
use this without any specific reference.

\begin{lemma}
Let $(U,\varphi)$ be a chart in a metric measure space $(X,d,\mu)$ and
$\gamma\in\Gamma(X)$.  Suppose that for some non-isolated
$t_0\in\dom\gamma$ and $f\colon X\to\mathbb R$, $f$ is differentiable
with respect to $(U,\varphi)$ at $\gamma(t_0)$ and that
$(\varphi\circ\gamma)'(t_0)$ exists.  Then
\[(f\circ\gamma)'(t_0)=Df(x_0)\cdot (\varphi\circ\gamma)'(t_0).\]
\end{lemma}

\begin{proof}
If $\gamma$ is $L$ Lipschitz and $\gamma(t_0)=x_0$, use the triangle
inequality and the fact that
\[\limsup_{\dom\gamma\ni t\to t_0}\frac{|P(\gamma(t))-P(\gamma(t_0))|}{|t-t_0|}\leq L\limsup_{\im\gamma\ni x\to x_0}\frac{|P(x)-P(x_0)|}{d(x,x_0)}\]
for any $P\colon X\to\mathbb R$.
\end{proof}

\section{Construction of a non-differentiable Lipschitz function}\label{sec:construction}
In the following section we will see that a metric measure space
possesses many Alberti representations whenever a certain class of
subsets have measure zero.  Following this we will show that such
subsets do have measure zero in a Lipschitz differentiability space by
constructing a Lipschitz function that is differentiable almost
nowhere on such a set.

In this section we present a general method for constructing a
Lipschitz function that is not differentiable on a given subset of a
structured chart with certain properties.  Note that the construction
does not use the fact that we work inside a Lipschitz
differentiability space, just that we have a chart structure.

We will construct such a Lipschitz function from a given sequence of
Lipschitz functions.  The first step will be to modify each function
in such a sequence in the following way.

\begin{lemma}\label{lem:modifications}
Let $(X,d,\mu)$ be a metric measure space, $h>4\epsilon>0$ and $L>0$.
Then for any $L$-Lipschitz $f\colon X\to\mathbb R$ and Borel $S\subset
X$ there exists an $L$-Lipschitz $\tilde f\colon X\to\mathbb R$ and
Borel $\tilde S\subset S$ with $\mu(\tilde S)\geq
(1-4\epsilon/h)\mu(S)$ such that:
\begin{enumerate}\item\label{item:mod1}The support of $\tilde f$ is contained within $B(S,2h/L)$.
\item\label{item:mod2}$0\leq \tilde f \leq h$.
\item\label{item:mod3}For every $x, y\in B(S,h/L)$ with $x\neq y$,
\[\frac{|\tilde f(x)-\tilde f(y)|}{d(x,y)}\leq\frac{|f(x)-f(y)|}{d(x,y)}.\]
\item\label{item:mod4}For every $x_0\in \tilde S$ and $x\in X$ with $0<d(x,x_0)\leq\epsilon/L$,
\[\frac{|\tilde f(x)-\tilde f(x_0)|}{d(x,x_0)}=\frac{|f(x)-f(x_0)|}{d(x,x_0)}.\]
\end{enumerate}
\end{lemma}

\begin{proof}
Let $M$ be the greatest integer less than $h/4\epsilon$ and for each
$0\leq k < M$ define $Z_k(t)=d(t,2\epsilon k+ 2h\mathbb Z)$ and
$\hat f_k=Z_k\circ f\colon X\to\mathbb R$.  Then each $\hat f_k$
satisfies \eqref{item:mod2} and \eqref{item:mod3}.  Now define
\[F_k=\{x_0\in S: d(f(x_0),2\epsilon k + 2h \mathbb Z)<\epsilon\}.\]
Then for integer $0\leq k<M$, the $F_k$ are disjoint Borel subsets of
$S$.  Indeed, if $x_0\in F_k\cap F_{k'}$ then there exists an
$n\in\mathbb N$ with $|k-k'+hn/\epsilon|<1$.  Since
$|k-k'|<h/4\epsilon$ and $h/\epsilon>4$, this can only happen if $n=0$
and $k=k'$ and so such $F_k$ and $F_{k'}$ are either disjoint or
equal.  Therefore, there exists an $m$ with
\[\mu(F_m)\leq \mu(S)/M\leq 4\epsilon\mu(S)/h.\]
We set $\tilde S = S\setminus F_m$ so that, for $x_0\in \tilde S$ and
$x\in X$ with $d(x,x_0)\leq\epsilon/L$,
\[|\hat f_m(x)-\hat f_m(x_0)|<\epsilon \text{ and } d(f(x_0),2\epsilon m+2h\mathbb Z)\geq \epsilon.\]
In particular
\[|\hat f_m(x)-\hat f_m(x_0)|=|f(x)-f(x_0)|\]
and so \eqref{item:mod4} holds.

Finally we define
\begin{align*}\tilde f &\colon B(S,h/L)\cup X\setminus B(S,2h/L)\to\mathbb R\\
\tilde f(x) &= \begin{cases}\hat f_m & x\in B(S,h/L)\\
0 & x\in X\setminus B(X,2h/L).\end{cases}\end{align*}
Then, since $0\leq \tilde f\leq h$, $\tilde f$ is $L$-Lipschitz and so we may extend it to a function on the whole of $X$ that satisfies the required properties.
\end{proof}

We now give the main Lemma required for the construction of a
non-differentiable Lipschitz function.  It takes a set $S$ and a
sequence of Lipschitz functions that, in some neighbourhood of $S$,
witness both large and small difference quotients at all points of $S$
and combines them to construct a Lipschitz function that witnesses
such difference quotients in all neighbourhoods of $S$.  Later in this
section we will see that such a function cannot be differentiable at any
point of $S$.

\begin{lemma}\label{lem:prenondiff}
Let $(X,d,\mu)$ be a metric measure space, $S\subset X$ Borel and
$L,\beta>0$.  Suppose that there exists a sequence of $L$-Lipschitz
functions $f_m\colon X\to \mathbb R$ such that:
\begin{enumerate}
\item \label{item:steep}For every $x_0\in S$ there exists an $M\in\mathbb N$ such that for each $m\geq M$ there exists an $x\in X$ with $0<d(x,x_0)<1/m$ and
\[|f_m(x)-f_m(x_0)|\geq L\beta d(x,x_0).\]
\item\label{item:flat}For every $m\in\mathbb N$ there exists a $\rho_m>0$ such that, for every $x_0\in S$ and any $y,z\in B(x_0,\rho_m)$,
\[|f_m(y)-f_m(z)|\leq d(y,z)/m.\]
\end{enumerate}
Then for any $\alpha>0$ and $R_i>r_i>0$ such that $R_i,r_i\to 0$, there
exists $i(k)\to\infty$ and a Lipschitz function $F\colon X\to\mathbb
R$ such that:
\begin{itemize}
\item For almost every $x_0\in S$, $\Lip(F,x_0)\geq L\beta-\alpha$.
\item For every $k\in \mathbb N$, $x_0\in S$ and $x\in X$ with
  $r_{i(k)}<d(x,x_0)<R_{i(k)}$,
\[\frac{|f(x)-f(x_0)|}{d(x,x_0)}\leq \alpha.\]
\end{itemize}
\end{lemma}

\begin{proof}
By dividing by $L$ if necessary, we may suppose that $L=1$ and by
taking a suitable subsequence we may suppose that $R_i\leq r_{i-1}$
for each $i\in\mathbb N$.  We define sequences $m(k),\ i(k)\to \infty$
inductively as follows, choosing $m(0)=i(0)=1$.  Given $m(k)$ and $i(k)$
choose $i(k+1)>i(k)$ such that
\[R_{i(k+1)}<\rho_{m(k)} \text{ and } r_{i(k+1)} \leq 2^{-(k+1)}\alpha
r_{i(j)}\]
for every $j\leq k$.  Then choose $m(k+1)$ such that
\[\frac{1}{m(k+1)}\leq 2^{-(k+1)}r_{i(k+1)} \text{ and } m(k+1)\geq
2^{k+1}/\alpha.\]
Note that these conditions imply, for every $j\in\mathbb N$,
\[\sum_{k> j}r_{i(k)}\leq \alpha r_{i(j)} \text{ and } \sum_{k\in\mathbb N} 1/m(k)\leq \alpha.\]
We define 1-Lipschitz functions $g_k\colon
X\to \mathbb R$ and a Borel set $S_k\subset S$ by applying Lemma
\ref{lem:modifications} to $f_{m(k)}$ with $h=r_{i(k)}/2$ and
$\epsilon=1/m(k)$.  Then $\mu(S_k)\geq (1-2^{3-k})\mu(S)$ and the
$g_k$ have the following properties:
\begin{enumerate}\item The support of $g_k$ is contained within $B(S,r_{i(k)})$.
\item $0\leq g_k\leq r_{i(k)}/2$.
\item \label{item:prop3}For every $x_0\in S_k$ there exists an $x$ with $0<d(x,x_0)<1/m(k)$ such that
\[\frac{|g_k(x)-g_k(x_0)|}{d(x,x_0)}\geq \beta.\]
\item \label{item:prop4}For every $x_0\in S$ and any $x, y\in
  B(x_0,R_{i(k)})$ with $x\neq y$,
\[\frac{|g_k(x)-g_k(y)|}{d(x,y)}\leq \frac{1}{m(k)}.\]
\end{enumerate}
Finally we define $F=\sum_k g_k$ and
\[S'=\bigcap_{m\in\mathbb N}\bigcup_{k\geq m} S_k,\]
a set of full measure in $S$.

We first show that $F$ is Lipschitz.  Let $x\neq y\in X$ with
\[r_{i(k+1)}\leq d(x,y)<r_{i(k)}\]
and suppose that, for some $k'\leq k$,
\[|g_{k'}(x)-g_{k'}(y)|>d(x,y)/m_{k'}.\]
Then one of $x$ or $y$ must necessarily belong to $B(S,r_{i(k)})$.
However,
\[d(x,y)<r_{i(k)}\leq r_{i(k')}\]
and so there exists an $x_0\in S$ such that
\[x,y\in B(x_0, 2r_{i(k')})\subset B(x_0,R_{i(k'-1)}).\]
Since the $R_{i(k)}$ strictly decrease we have, for all $j<k'$,
\[|g_j(x)-g_j(y)|\leq d(x,y)/m(j).\]
In particular this can only happen for at most one value of $k'\leq k$.  If such a $k'$ does not exist set $k'=k$. Then we have
\begin{align*}|F(x)-F(y)|&\leq \sum_{k'\neq j\leq k}|g_j(x)-g_j(y)|+|g_{k'}(x)-g_{k'}(y)|+\sum_{j\geq k+1}|g_j(x)-g_j(y)|\\
&\leq \sum_{k'\neq j\leq k}d(x,y)/m(j) + 2d(x,y) + \sum_{j > k+1} r_{i(j)}/2\\
&\leq (\alpha+2)d(x,y) + \alpha r_{i(k+1)}\\
&\leq 2(\alpha+1)d(x,y)\end{align*}
and so $F\colon X\to\mathbb R$ is Lipschitz.

Now let $x_0\in S'$ and $K\in\mathbb N$ such that $x_0\in S_k$ for each
$k\geq K$.  Then for each $k\geq K$ there exists an $x\in X$ with
$0<d(x,x_0)<1/m(k)$ such that
\[\frac{|g_k(x)-g_k(x_0)|}{d(x,x_0)}\geq \beta.\]
Note that we also have $x\in B(S,R_{i(j)})$ for any $j<k$ and so
\begin{align*}|F(x)-F(x_0)|&\geq |g_k(x)-g_k(x_0)|-\sum_{j<k}|g_j(x)-g_j(x_0)|-\sum_{j>k}|g_j(x)-g_j(x_0)|\\
&\geq \beta d(x,x_0) - \sum_{j<k}d(x,x_0)/m(j) -\sum_{j>k}
  r_{i(j)}/2\\ &\geq \beta d(x,x_0) - \alpha d(x,x_0) - \alpha
  r_{i(k)}\\ &\geq (\beta-2\alpha)d(x,x_0).\end{align*}
Such an $x$ exists for each $k\geq K$ and so $\Lip(F,x_0)\geq
\beta-2\alpha$.

Now let $x_0\in S$ and for any $k\in\mathbb N$ let $x\in X$ with
$r_{i(k)}<d(x,x_0)<R_{i(k)}$.  Then
\begin{align*}|F(x_i)-F(x_0)|&\leq \sum_{j\leq k}|g_j(x)-g_j(x_0)|+\sum_{j>k}|g_j(x)-g_j(x_0)|\\
&\leq \alpha d(x,x_0) + \sum_{j>k}r_{i(j)}/2\\
&\leq 2\alpha d(x,x_0).\end{align*}
Therefore, $F$ satisfies the conclusion of the Lemma for $2\alpha$.
\end{proof}

By combining this construction with the definition of a structured
chart, we show that the constructed function is differentiable
almost nowhere on such a set.

\begin{proposition}\label{prop:cons}
Let $(U,\varphi)$ be a structured chart in a metric measure space
$(X,d,\mu)$, $S\subset U$ be Borel and $L,\beta>0$.  Suppose that there
exists a sequence of $L$-Lipschitz functions $f_m\colon X\to \mathbb
R$ such that:
\begin{enumerate}
\item For every $x_0\in S$ there exists an $M\in\mathbb N$ such that, for
  each $m\geq M$, there exists an $x\in X$ with $0<d(x,x_0)<1/m$ and
\[|f_m(x)-f_m(x_0)|\geq \beta d(x,x_0).\]
\item There exists a $\rho_m>0$ such that, for every $x_0\in S$ and any $y,z\in B(x_0,\rho_m)$,
\[|f_m(y)-f_m(z)|\leq d(y,z)/m.\]
\end{enumerate}
Then there exists a Lipschitz function $F\colon X\to \mathbb R$ that is
differentiable $\mu$-almost nowhere on $S$.  In particular, if $X$ is a
Lipschitz differentiability space, then $S$ is $\mu$-null.
\end{proposition}

\begin{proof}
For $\lambda>0$ let $(U,\varphi)$ be a $\lambda$-structured chart and
choose $\alpha <\lambda\beta/(\Lip\varphi+1)$.  Then by the definition
of a structured chart, there exist
positive $R_i>r_i\to 0$ such that, for almost every $x_0\in U$ and every
$i\in\mathbb N$, there exist $x_1,\ldots,x_n\in U$ with
$r_i<d(x_j,x_0)<R_i$ and
\[\max_{1\leq j\leq n}\frac{|(\varphi(x_j)-\varphi(x_0))\cdot v|}{d(x_j,x_0)}\geq \lambda\]
for every $v\in\mathbb S^{n-1}$.  We apply Lemma \ref{lem:prenondiff}
to obtain a Lipschitz function $F\colon X\to\mathbb R$ and
$i(k)\to\infty$ such that, for almost every $x_0\in S$, $\Lip(F,x_0)\geq
\beta-\alpha$ and for every $k\in\mathbb N$ and $x\in X$ with
$r_{i(k)}<d(x,x_0)<R_{i(k)}$,
\[\frac{|F(x)-F(x_0)|}{d(x,x_0)}\leq \alpha.\]
In particular, there exists a sequence $x_m\to x_0$ satisfying
\eqref{eq:uniqueness} for $\lambda$ such that
\[\liminf_{m\to\infty}\max_{0\leq i< n}\frac{|(f(x_{mn+i})-f(x_0))\cdot v|}{d(x_{mn+i},x_0)}\leq \alpha.\]
Therefore, by Lemma \ref{lem:nondiffcond} and our choice of $\alpha$,
$F$ is not differentiable at almost every point of $S$.
\end{proof}

The following result first appeared in \cite{porousdiff}, however we
now give a short proof as this result will be required when
establishing our characterisations of Lipschitz differentiability
spaces.  We first require some notation.

\begin{definition}\label{def:porous}
Let $(X,d)$ be a metric space, $S\subset X$ Borel and $x_0\in S$.  We
say that $S$ is \emph{porous at $x_0$} if there exists an $\eta>0$ and
$X\ni x_m\to x_0$ such that
\[B(x_m,\eta d(x_m,x_0))\cap S =\emptyset.\]

Further, we say that $S$ is \emph{porous} if it is porous at $x_0$ for
each $x_0\in S$.
\end{definition}

\begin{corollary}[\cite{porousdiff}, Theorem 2.4]\label{cor:porousnull}
Porous sets in a Lipschitz differentiability space $(X,d,\mu)$ have
measure zero.
\end{corollary}

\begin{proof}
For any porous set $S\subset X$ there exists a countable Borel decomposition
$S=\cup_i S_i\cup N$ and a sequence $\eta_i>0$ such that $\mu(N)=0$,
each $S_i$ is compact and contained within a structured chart and for
each $i\in\mathbb N$ and $x_0\in S_i$, there exists $x_m\to x_0$ such
that
\[B(x_m,\eta_i d(x_m,x_0))\cap S_i = \emptyset.\]
Then for any $r>0$ and $i\in\mathbb N$, the function
\[x_0\mapsto \sup\{d(x,x_0)<r: d(x,S_i) > \eta_i d(x,x_0)/2\}\]
is well defined, strictly positive and lower semicontinuous on $S_i$ and
so there exists an $\epsilon_{r,i}>0$ such that it is bounded below by
$\epsilon_{r,i}$.  Then the functions
\[f_m(x)= \min\{d(x,S_i)-\epsilon_{1/m,i},0\}\]
satisfy the hypotheses of Proposition \ref{prop:cons} for $S_i$ and so
$S_i$ is $\mu$-null.  In particular, $S$ is also $\mu$-null.
\end{proof}

\subsection{Additional remarks on Lipschitz differentiability spaces}
We give several simple consequences of Proposition \ref{prop:cons}.

\begin{corollary}[\cite{porousdiff}, Corollary 2.7]\label{cor:subsetdiff}
For any measurable subset $Y$ of a Lipschitz differentiability
space $(X,d,\mu)$, $(Y,d,\mu)$ is a Lipschitz differentiability space
with respect to the same chart structure.

Moreover, the derivative of a real
valued Lipschitz function in $Y$ agrees with the derivative of any
Lipschitz extension of $f$ to $X$, almost everywhere.
\end{corollary}

\begin{proof}
Let $f\colon Y\to\mathbb R$ be Lipschitz and $(U,\varphi)$ a
$\lambda$-structured chart of dimension $n$ in $X$.  We may extend $f$
to a Lipschitz function defined on $X$ and so, for almost every
$x_0\in Y\cap U$, there exists a unique $Df(x_0)$ such that
\[\limsup_{X\ni x\to x_0}\frac{|f(x)-f(x_0)-Df(x_0) \cdot(\varphi(x)-\varphi(x_0))|} {d(x,x_0)}=0.\]
This $Df(x_0)$ will also be a derivative for the metric measure
space $(Y,d,\mu)$, provided it is unique.  As seen in Lemma
\ref{lem:discuniq} such a $Df(x_0)$ is not unique if and only if there
exists a $v\in\mathbb S^{n-1}$ such that
\[\limsup_{Y\ni x\to x_0}\frac{|(\varphi(x)-\varphi(x_0))\cdot
  v|}{d(x,x_0)}=0.\]
However, since $(U,\varphi)$ is a $\lambda$-structured chart in $X$,
there exist $X\ni x_m\to x_0$ with
\[\lim_{m\to\infty}\frac{|(\varphi(x_m)-\varphi(x_0))\cdot
  v|}{d(x_m,x_0)}\geq \lambda.
\]
By combining these two relations and using the fact that $\varphi$ is
Lipschitz, we see that
\[B(x_m, \lambda d(x_m,x_0)/2\Lip\varphi)\cap Y = \emptyset
\]
for sufficiently large $m$.  In particular, for each $v\in\mathbb
S^{n-1}$, the set of such $x_0$ is a porous set in $X$.  By taking the
union of such $x_0$ over a countable dense subset of $\mathbb
S^{n-1}$, we see that the set of those $x_0$ where $Df(x_0)$ is not
unique is $\mu$-null.  Therefore the derivative of $f$ in $X$ is also
the derivative of $f$ in $Y$, almost everywhere.
\end{proof}

We also show how we may apply the construction of a non-differentiable
Lipschitz function when we only have control over the infinitesimal
behaviour of such a sequence of Lipschitz functions.

\begin{corollary}\label{cor:cons}
Let $(X,d,\mu)$ be a Lipschitz differentiability space, $S\subset X$
Borel and $L,\beta>0$.  Suppose that there exists a sequence of
$L$-Lipschitz functions $f_m\colon X\to\mathbb R$ such that, for
almost every $x_0\in S$:
\begin{itemize}
\item There exists an $M\in\mathbb N$ such that, for each $m\geq M$, there
  exists an $x\in X$ with $0<d(x,x_0)<1/m$ and
\[|f_m(x)-f_m(x_0)|\geq \beta d(x,x_0).\]
\item For each $m\in\mathbb N$
\[\Lip(f_m,x_0)\leq 1/m.\]
\end{itemize}
Then $S$ is $\mu$-null.
\end{corollary}

\begin{proof}
By Lemma \ref{lem:decompstruct}, it suffices to prove the result
whenever $S$ is contained within a structured chart $(U,\varphi)$.
For this, we will apply Proposition \ref{prop:cons} to a suitable
subset of $S$.

For any $\epsilon>0$ there exist a Borel $S' \subset S$ with
$\mu(S')\geq \mu(S)-\epsilon$, an $M\in\mathbb N$ and a sequence
$\rho_m>0$ such that:
\begin{itemize}
\item For every $m\geq M$, $x_0\in S'$ and $y\in X$ with
  $0<d(y,x_0)<\rho_m$,
\[\frac{|f_m(y)-f_m(x_0)|}{d(y,x_0)}\leq 2/m.\]
\item For every $m\geq M$ and $x_0\in S'$ there exists an
$x\in X$ with $0<d(x,x_0)<1/m$ and
\[\frac{|f_m(x)-f_m(x_0)|}{d(x,x_0)}\geq \beta.\]
\end{itemize}
In particular, for any $x_0\in S'$ and $y,z\in B(x_0,\rho_m)\cap S'$,
\[\frac{|f_m(y)-f_m(z)|}{d(y,z)}\leq 2/m.\]

Also observe that, for any $x_0\in S'$, there either exists an $M'\geq
M$ and for each $m\geq M'$ an $x\in S'$ with
\[\frac{|f_m(x)-f_m(x_0)|}{d(x,x_0)}\geq \beta/2\]
or there exists $X\ni x_j\to x_0$ such that $B(x_j,\beta
d(x_j,x_0)/2L)\cap S'=\emptyset$.  Therefore, there exists a Borel
decomposition $S'=S''\cup P$ where $P$ is a porous subset of $X$ and
$S''$ satisfies the hypotheses of Proposition \ref{prop:cons} for the
Lipschitz differentiability space $(S',d,\mu)$.  Therefore $S''$ and so $S'$
are $\mu$-null.  Finally, since $\epsilon>0$ was arbitrary, $S$ is
also $\mu$-null.
\end{proof}

Using the previous Corollary, we give an example of metric spaces that
can only be Lipschitz differentiability spaces when they have zero
measure.

\begin{corollary}
Let $(X,d,\mu)$ be a Lipschitz differentiability space, $Y\subset X$
Borel and for each $m\in\mathbb N$ let $N_m$ be a $1/m$-net of
$Y$.  Suppose that, for each $m\in\mathbb N$ and $q\in N_m$,
\[\limsup_{X\ni x\to x_0}\frac{|d(x,q)-d(x_0,q)|}{d(x,x_0)}=0\]
for almost every $x_0\in Y\setminus \cup_m N_m$.  Then $\mu(Y)=0$.

In particular, for any metric measure space $(X,d,\mu)$ and
$0<\delta<1$, $(X,d^\delta,\mu)$ is a Lipschitz differentiability
space if and only if $\mu(X)=0$.
\end{corollary}
\begin{proof}
Suppose that $(U,\varphi)$ is a structured chart in $X$ and $S\subset
Y\cap U\setminus \cup_m N_m$ is compact.  Then for each $m\in\mathbb
N$ there exists a finite $N'_m\subset N_m$ such that $B(N'_m,1/m)\supset
S$.  We define the 1-Lipschitz function
\[f_m(x)=d(x,N'_m)=\min\{d(x,q):q\in N'_m\}.\]
Then by the hypothesis on the $N_m$, $\Lip(f_m,x_0)=0$ for almost every $x_0\in S$.

Further, for any $x_0\in S$ let $q\in N'_m$ with $d(x_0,q)$ minimal.
Then $d(x_0,q)<1/m$ and
\[\frac{|f_m(x_0)-f_m(q)|}{d(x,q)}=1.\]
Therefore the hypotheses of Corollary \ref{cor:cons} are satisfied for
$S'$, so that $S'$ and hence $S$ are $\mu$-null.  Since $S\subset
U\setminus \cup_m N_m$ was an arbitrary compact set and each $N_m$
is separated, $Y\cap U$ must also be $\mu$-null.

Now let $(X,d,\mu)$ be a metric measure space and $0<\delta<1$.  For
any $a>0$, $|(a+r)^\delta-a^\delta|/r^\delta\to 0$ as $r\to 0$.
In particular, for any $x_0, z\in (X,d^\delta)$ with $x_0\neq z$, we may
use the triangle inequality in $(X,d)$ to deduce
\[\limsup_{x\to x_0}
\frac{|d^\delta(x,z)-d^\delta(x_0,z)|}{d^\delta(x,x_0)}
\leq \frac{|(d(x_0,z)+d(x,x_0))^\delta -d^\delta(x_0,z)|}{d^\delta(x,x_0)}
=0.
\]
Therefore, by choosing $N_m$ to be any $1/m$-net of $X$, if
$(X,d^\delta,\mu)$ is a Lipschitz differentiability space we must have
$\mu(X)=0$.
\end{proof}

Finally, we address the notion of a zero dimensional chart in a
Lipschitz differentiability space.

\begin{corollary}\label{cor:zerodim}
Let $(X,d,\mu)$ be a metric measure space and $U\subset X$ Borel.
Suppose that, for any Lipschitz $f\colon X\to\mathbb R$,
\[\limsup_{X\ni x\to x_0}\frac{|f(x)-f(x_0)|}{d(x,x_0)}=0\]
for almost every $x_0\in U$.  Then $\mu(U)=0$.
\end{corollary}

\begin{proof}
First observe that, for any $x_0\in X$, the function $x\mapsto d(x,x_0)$
satisfies $\Lip(f,x_0)>0$ and so singletons must have measure
zero in $U$.  In particular, almost every point of $U$ is a limit
point of $U$.

Now, for any $m\in\mathbb N$ let $N_m$ be a $1/m$-net of
$U$, $S\subset U\setminus N_m$ be compact and $N'_m\subset N_m$ be
finite with $B(N'_m,1/m)\supset S$.  We define
\[f_m(x)=d(x,N'_m)=\min\{d(x,q):q\in N'_m\}\]
and let $S'\subset S$ be a compact and $\rho>0$ such that
\[\frac{|f_m(x)-f_m(x_0)|}{d(x,x_0)}\leq 1/m\]
for each $x_0\in S'$ and $0<d(x,x_0)<\rho$.  In particular, for any
$x_0\in S'$ and $y,z\in B(x_0,\rho)\cap S'$,
\[\frac{|f_m(y)-f_m(z)|}{d(y,z)}\leq 1/m.\]
Moreover, for any $x_0\in S'$ there exists a $x\in U$ with
$0<d(x,x_0)<1/m$ and
\[\frac{|f_m(x)-f_m(x_0)|}{d(x,x_0)}=1.\]
We apply Lemma \ref{lem:prenondiff} to obtain a
Lipschitz function $F\colon X\to\mathbb R$ with $\Lip(F,x_0)>0$ for
almost every $x_0\in S'$.  Therefore $S'$ and hence $S$ and $X$ are
$\mu$-null.
\end{proof}

\begin{remark}\label{rmk:keithrmk}
Suppose we choose to define the $\limsup$ of a function at an
isolated point to equal zero (as is the case in \cite{keith}).  Then
this Corollary shows that a set in a Lipschitz differentiability
space is a chart of dimension zero if and only if almost every point
is isolated.  This answers the question raised in Remark 2.1.3 of
~\cite{keith} on the nature of charts of dimension zero.
\end{remark}

\section{The existence of Alberti representations}\label{sec:reps}
\subsection{A single representation}
We now give a characterisation of the metric measure spaces with
many independent Alberti representations (see Definition
\ref{defn:phidir}).  Our method relies upon the Gliksberg-K\"onig-Seever
Decomposition Theorem, which we state first.

\begin{theorem}[\cite{rudin-unitball}, Theorem 9.4.4]\label{thm:gks}
Let $(X,d,\mu)$ be a compact metric measure space and $K$ a
weak$^*$-compact, convex and non-empty set of regular Borel probability
measures on $X$.  Then there exists a Borel decomposition $X=A\cup S$
such that $\mu\llcorner A \ll\nu$ for some $\nu\in K$ and $k(S)=0$ for
every $k \in K$.
\end{theorem}

We would like to apply this Theorem with $K$ a set of normalised
$\mathcal H^1$ measures supported on elements of a compact subset of
$\Gamma(X)$ (see Definition \ref{defn:curvespace}).  However, this set
need not be compact.  Instead, in the following Lemma, we apply it
with $K$ a set of normalised $\mathcal H^1$ measures supported on
elements of a compact subset of $\Pi(X)$ (see Definition
\ref{defn:curvespace}).  Following this, we will use this Lemma
to obtain a conclusion in terms of $\Gamma(X)$.

\begin{lemma}\label{lem:generalrep}
Let $(X,d,\mu)$ be a compact metric measure space and $J\subset\Pi(X)$
closed.  Then there exists a Borel decomposition $X=A\cup S$ and an
Alberti representation $\mathcal A$ of $\mu\llcorner A$ such that:
\begin{enumerate}\item \label{item:rep} Almost every $\gamma\in\mathcal A$ belongs to $J$.
\item \label{item:null} For every $\gamma\in J$, $\mathcal
  H^1(\gamma\cap S)=0$.
\end{enumerate}
\end{lemma}

\begin{remark}
If $\mu$ gives measure zero to any set $S$ that satisfies
\eqref{item:null} then $\mathcal A$ is an Alberti representation of
$\mu$.  Observe that, from the definition of an Alberti
representation, this condition is also necessary for a
representation of $\mu$ satisfying \eqref{item:rep} to exist.
Therefore, this Lemma gives a characterisation of the existence of such
an Alberti representation.
\end{remark}

\begin{proof}
If $J$ is empty then setting $A=\emptyset$ and $\mathcal
A=(0,\{\mathcal H^1\llcorner\gamma\})$ completes the proof.
Otherwise, for each $k\in\mathbb N$ let $J_k$ be the set of
$k$-bi-Lipschitz $\gamma\in J$ with $\dom\gamma\subset [-k,k]$ and
$\mathcal L^1(\dom\gamma) \geq 1/k$.  Then by Lemma \ref{lem:gammaborel},
since $J\subset \Pi(X)$ is closed, each $J_k$ is isometrically
equivalent to a closed subset of the compact subsets of $[-k,k]\times
X$ with the Hausdorff metric, and so is compact.

Now fix $k\in\mathbb N$ and for each $\gamma\in J_k$ let $\mathcal
H_\gamma$ be the pushforward of the Lebesgue measure on $\dom\gamma$
under $\gamma$, multiplied by a suitable scalar such that $\mathcal
H_\gamma$ is a probability measure.  Then since $J_k$ is compact, the
set of such measures is weak$^*$-compact and so $K$, its
weak$^*$-closed convex hull, satisfies the hypothesis of Theorem
\ref{thm:gks}.  Moreover, probability measures on $J_k$ are
weak$^*$-compact and so any measure $\nu\in K$ is of the form
\[\nu(A)=\int_{\Pi(X)} \mathcal H_\gamma (A)\text{d}\mathbb P'(\gamma),\]
for some Borel probability measure $\mathbb P'$ on $\Pi(X)$
with $\mathbb P'(J_k)=1$.

Therefore, Theorem \ref{thm:gks} gives a decomposition $X=A_k\cup S_k$
where $\mu\llcorner A_k$ is absolutely continuous with respect to some
$\nu\in K$ and $\mathcal H_\gamma(S_k)=0$ for every $\gamma\in J_k$.
In particular, an application of Lemma \ref{lem:albertinikodym}
constructs an Alberti representation $\mathcal A_k$ of $\mu\llcorner
A_k$ from $\nu$ such that almost every $\gamma\in\mathcal A_k$ belongs
to $J_k$ and $\mathcal H^1(\gamma\cap S_k)=0$ for every $\gamma\in
J_k$.  Finally, if we define $A=\cup_k A_k$ and $S=\cap_k S_k$, then
$X=A\cup S$ is a Borel decomposition of $X$ such that, by Lemma
\ref{lem:sumreps}, $\mu\llcorner A$ has an Alberti representation
$\mathcal A$ such that almost every $\gamma\in\mathcal A$ belongs to
$J$, and such that $S$ satisfies $\mathcal H^1(\gamma\cap S)=0$ for
every $\gamma\in J$ with $\mathcal L^1(\dom\gamma)>0$.  Since $\mathcal
H^1(\gamma\cap S)=0$ for any $\gamma\in J$ with $\mathcal
L(\dom\gamma)=0$, this decomposition is of the required form.
\end{proof}

We summarise the exact setting in which we will use the above Lemma.
Not only does the resulting Alberti representation have a direction
but we may also control the magnitude of the partial derivatives of a
finite number of Lipschitz functions.  Later (in Definition
\ref{def:speed}), this will be developed into the notion of the
\emph{speed} of an Alberti representation.

\begin{corollary}\label{cor:basicrep}
Let $(X,d,\mu)$ be a compact metric measure space and, for some
$n\in\mathbb N$, let $\varphi,\psi\colon X\to\mathbb R^n$ be
Lipschitz.  Then for any cone $C\subset\mathbb
R^n$ and $\delta_1,\ldots,\delta_n>0$, there exists a Borel
decomposition $X=A\cup S$ such that:
\begin{itemize}
\item There exists an Alberti representation $\mathcal A$ of
  $\mu\llcorner A$ in the $\varphi$-direction of $C$ such that, for
  almost every $\gamma\in\mathcal A$, $\gamma$ belongs to $\Pi(X)$ and
\[\|(\psi_i\circ\gamma)'(t_0)\|\geq\delta_i\Lip(\gamma,t_0)\]
for almost every $t_0\in\dom\gamma$ and every $1\leq i\leq n$.
\item $\mathcal H^1(\gamma\cap S)=0$ for every $\gamma\in\Pi(X)$ in
the $\varphi$-direction of $C$ with
\begin{equation}\label{eq:prespeed}
  \|\psi_i(\gamma(t))-\psi_i(\gamma(t'))\| \geq \delta_i
  \|\gamma(t)-\gamma(t')\|
\end{equation}
for every $t,t'\in\dom\gamma$ and every $1\leq i\leq n$.
\end{itemize}
\end{corollary}

\begin{proof}
Let $\gamma_m\in\Pi(X)$ be a sequence of Lipschitz curves in the
$\varphi$-direction of $C$ and, for some $\gamma\in\Pi(X)$, suppose
that $\gamma_m\to \gamma$.
Then for any $t,t'\in\dom\gamma$ and for each $m\in\mathbb N$ there
exists $t_m,t'_m\in\dom\gamma_m$ such that
\[\gamma_m(t_m)\to\gamma(t) \text{ and } \gamma_m(t'_m)\to\gamma(t').\]
Then
\begin{align*}(\varphi(\gamma(t))-\varphi(\gamma(t')))\cdot w &=\lim_{m\to\infty} (\varphi(\gamma_m(t_m))-\varphi(\gamma_m(t_m)))\cdot w\\
&\geq \lim_{m\to\infty} (1-\theta)\|\varphi(\gamma_m(t_m))-\varphi(\gamma_m(t_m))\|\\
&= (1-\theta)\|\varphi(\gamma(t))-\varphi(\gamma(t'))\|
\end{align*}
so that $\gamma$ is in the $\varphi$-direction of $C$.  Similarly, if
each $\gamma_m$ satisfies \eqref{eq:prespeed} for every
$t,t'\in\dom\gamma$ and every $1\leq i\leq n$, then so does $\gamma$.
Therefore, the set $J$ of all $\gamma\in\Pi(X)$ in the
$\varphi$-direction of $C$ that satisfy \eqref{eq:prespeed} for every
$1\leq i \leq n$ and every $t,t'\in\dom\gamma$ is closed.  Moreover,
for any $\gamma\in J$,
\[\|(\psi_i\circ\gamma)'(t_0)\|\geq \delta_i\Lip(\gamma,t_0)\]
for almost every $t_0\in\dom\gamma$ and each $1\leq i \leq n$.  By
applying Lemma \ref{lem:generalrep} to $J$ we obtain a decomposition
$X=A\cup S$ of the required form.
\end{proof}

Observe that, if $X$ is a metric space such that $\Pi(X)$ only
contains Lipschitz curves $\gamma$ with $\dom\gamma$ a single point,
then the decomposition $A=\emptyset$, $S=X$ satisfies the conclusion
of the previous Lemma.  To obtain a meaningful result, we will instead
apply the Lemma to the closed, convex hull of an isometric copy of
$X$ within a Banach space $\mathcal B$.  In this case
$\Pi(\mathcal B)$ is very rich and so we obtain a much stronger conclusion.

So that we may describe such an Alberti representation purely
in the language of the metric space, we now investigate Lipschitz curves
in the singular set $S$.

\begin{lemma}\label{lem:nulloncone2}
Let $\mathcal B$ be a Banach space, $X\subset\mathcal B$ closed and
convex and let $\varphi,\psi\colon\mathcal B\to\mathbb R^n$ be
Lipschitz.  For a cone $C\subset \mathbb R^n$ and
$\delta_1,\ldots,\delta_n>0$ suppose that a Borel set
$S\subset\mathcal B$ satisfies $\mathcal H^1(\gamma\cap S)=0$ for all
$\gamma\in\Pi(X)$ in the $\varphi$-direction of $C$ with
\[\|\psi_i(\gamma(t))-\psi_i(\gamma(t'))\|\geq\delta_i\|\gamma(t)-\gamma(t')\|\]
for every $t,t'\in\dom\gamma$ and each $1\leq i \leq n$.  Then for
any measurable $D\subset \mathbb R$ and Lipschitz $\gamma\colon D\to
X$ the set
\[\{t_0\in D: (\varphi\circ\gamma)'(t_0)\in
C^\circ \text{ and } |(\psi_i\circ\gamma)'(t_0)|>\delta_i
\Lip(\gamma,t_0)\ \forall\ 1\leq i \leq n\}\]
is Lebesgue null.
\end{lemma}

\begin{proof}
Suppose that the conclusion does not hold for some measurable $D\subset
\mathbb R$ of positive measure and $\gamma\colon D\to X$.  We will
construct a new function $\widetilde\gamma\in \Pi(X)$ that agrees with
$\gamma$ on a set of positive measure and satisfies the conditions
given in the hypotheses of the Lemma, producing a contradiction.

By standard measure theoretic techniques, there exist
$\Phi,\Psi\in\mathbb R^n$, $b,\xi>0$ and $\delta'_i>\delta_i$ with
\[\overline B(\Phi,\xi)\subset C \text{ and each } |\Psi_i|\geq \delta_i' b\]
such that, for every $\epsilon>0$, there exists a bounded $D'\subset D$ of
positive measure with
\[\|(\varphi\circ\gamma)'(t_0)-\Phi\|<\epsilon,\ |(\psi_i\circ\gamma)'(t_0)|>|\Psi_i| \text{ and } \Lip(\gamma,t_0)< b\]
for every $1\leq i\leq n$ and $t_0\in D'$.  Let $D'$ be such a set for some choice of
\[0<\epsilon<\min\{\|\Phi\|,\xi\}.\]
Further, let $R>0$ and $D''\subset D'$ have positive measure such that,
for every $t_0\in D''$ and every $t\in\dom\gamma$ with $|t-t_0|\leq R$,
\[\|\varphi(\gamma(t))-\varphi(\gamma(t_0))-\Phi(t-t_0)\|\leq \epsilon
|t-t_0|,\]
for each $1\leq i \leq n$
\[|\psi_i(\gamma(t))-\psi_i(\gamma(t_0))|\geq |\Psi_i(t-t_0)|,\]
and
\[\|\gamma(t)-\gamma(t_0)\|\leq b|t-t_0|.\]

We set $D_0$ to be the intersection of $D''$ with an interval of
diameter $R$ chosen so that $D_0$ has positive measure, and $I$ to be
the smallest closed interval containing $D_0$.  Finally, we define
$\widetilde \gamma$ to equal $\gamma$ on $D_0$ and extend $\widetilde\gamma$
to $I$ whilst maintaining the Lipschitz constant by first extending
to the closure of $D_0$ and then linearly on the connected components
of the complement.

First observe that, since $X$ is closed and convex,
$\im\widetilde\gamma\subset X$.  Further, since $\gamma|_{D_0}$ is in the
$\varphi$-direction of $C$, $\mathcal H^1(\gamma(D_0))>0$.  Now, for
 any connected component $(c,d)$ of $I\setminus\overline{D}_0$, let $c_m,d_m\in D_0$ with $c_m\to c$ and $d_m\to d$.  Then
\begin{align*}\|\varphi(\widetilde\gamma(c))-\varphi(\widetilde\gamma(d))-\Phi(c-d)\| &= \lim_{m\to\infty} \|\varphi(\widetilde\gamma(c_m))-\varphi(\widetilde\gamma(d_m))-\Phi(c_m-d_m)\|\\
&\leq \epsilon |c-d|.\end{align*}
Similarly
\[|\psi_i(\widetilde\gamma(c))-\psi_i(\widetilde\gamma(d))|\geq |\Psi_i(c-d)|\]
for each $1\leq i \leq n$ and
\[\|\widetilde\gamma(c)-\widetilde\gamma(d)\| \leq b|c-d|.\]
In particular, since $\widetilde\gamma$ is extended linearly to $(c,d)$,
for almost every $t_0\in I$,
\[\|(\varphi\circ\widetilde\gamma)'(t_0)-\Phi\|\leq \epsilon\]
so that $\tilde\gamma$ is in the $\varphi$-direction of $C$.  Similarly
\[|(\psi_i\circ\widetilde\gamma)'(t_0)|\geq |\Psi_i|\]
for each $1\leq i \leq n$ and, for every $t,t'\in I$,
\[\|\widetilde\gamma(t)-\widetilde\gamma(t')\| \leq b|t-t'|.\]
Therefore, for any $t\leq t'\in I$,
\begin{align*}\|\widetilde\gamma(t)-\widetilde\gamma(t')\|&\geq \|\varphi(\widetilde\gamma(t))-\varphi(\widetilde\gamma(t'))\|\\
&= \left\|\int_{[t,t']} (\varphi\circ\widetilde\gamma)'(t_0)\right\|\\
&\geq (\|\Phi\|-\epsilon)|t-t'|\\
\end{align*}
so that $\widetilde\gamma$ is bi-Lipschitz and so belongs to $\Pi(X)$.
Similarly, for any $1\leq i \leq n$,
\begin{align*}
|\psi_i(\widetilde\gamma(t))-\psi_j(\widetilde\gamma(t'))| &=
\left|\int_{[t,t']} (\psi_i\circ\widetilde\gamma)'(t_0)\right|\\
&\geq \delta'_ib|t-t_0|\\
&\geq \delta_i\|\widetilde\gamma(t)-\widetilde\gamma(t')\|.\end{align*}
Therefore $\widetilde\gamma$ satisfies the hypotheses of the Lemma and so
\[0=\mathcal H^1(\tilde\gamma\cap S)\geq \mathcal H^1(\tilde\gamma(D_0))>0,\]
a contradiction.
\end{proof}

To apply these results to a metric measure space, we now define an
embedding into a Banach space that exposes additional structure
given by some Lipschitz functions defined on the metric space.

\begin{definition}\label{def:embedding}
For $n\in\mathbb N$ we define $\mathcal B$ to be the Banach space
\[\mathcal B = \mathbb R^n \times \mathbb R^n \times \ell_\infty\]
with norm
\[\|(u,v,s)\|=\max\{ \|u\|_{\mathbb R^n}, \|v\|_{\mathbb R^n},
\|s\|_{\ell_\infty}\}.\]

Further, suppose that $(X,d)$ is a metric space and
$\varphi,\psi\colon X\to\mathbb R^n$ Lipschitz.  We fix a bi-Lipschitz
embedding $\iota\colon X\to \ell_\infty$ and define the bi-Lipschitz
embedding $\iota^*\colon X\to \mathcal B$ by
\[x \mapsto (\varphi(x),\psi(x),\iota(x)).\]
We identify $X$ with its image in $\mathcal B$ under $\iota^*$ and
note that the projection $P_1$ onto the first factor
(respectively $P_2$ onto the second factor) in
$\mathcal B$ agrees with $\varphi$ (respectively $\psi$) on $X$.  In
particular, for any $\gamma\in\Gamma(X)$, the derivative
$(\varphi\circ\gamma)'$ (respectively $(\psi\circ\gamma)'$) agrees
with $(P_1\circ\gamma)'$ (respectively $(P_2\circ\gamma)'$) almost
everywhere.
\end{definition}

Finally, we define an additional property of an Alberti
representation.  The \emph{speed} of an Alberti representation
quantitatively describes how the partial derivative of a Lipschitz
function with respect to the Alberti representation compares to the
infinitesimal behaviour of such a function.  This notion will have a
particular importance when characterising Lipschitz differentiability
spaces when we will be interested in Alberti representations with
speed independent of the Lipschitz function.

\begin{definition}\label{def:speed}
Let $(X,d)$ be a metric space, $\varphi\colon X\to\mathbb R^n$
Lipschitz and $\delta>0$.  We say that $\gamma\in\Gamma(X)$ has
\emph{speed $\delta$} (or \emph{$\varphi$-speed $\delta$} whenever the
implied Lipschitz function is not clear) if for almost every
$t_0\in\dom\gamma$,
\[\|(\varphi\circ\gamma)'(t_0)\|\geq\delta\Lip(\varphi,\gamma(t_0))\Lip(\gamma,t_0).\]
Further, we say that an Alberti representation $\mathcal A$ has
\emph{speed $\delta$} (or $\varphi$-speed $\delta$) if almost every
$\gamma\in\mathcal A$ has speed $\delta$.
\end{definition}

\begin{corollary}\label{cor:onerep2}
Let $(X,d,\mu)$ be a metric measure space and $\varphi,\psi\colon
X\to\mathbb R^n$ Lipschitz such that $\Lip(\psi_i,x_0)>0$ for each
$1\leq i \leq n$.  Then, for any cone $C\subset\mathbb R^n$ and
$\delta_1,\ldots,\delta_n>0$, there exists a Borel decomposition
$X=A\cup S$ such that:
\begin{itemize}
\item There exists an Alberti representation of $\mu\llcorner A$ in
the $\varphi$-direction of $C$ with $\psi_i$-speed $\delta_i$ for each
$1\leq i\leq n$.
\item $\mathcal H^1(\gamma\cap S)=0$ for
any $\gamma\in\Gamma(X)$ in the $\varphi$-direction of $C^\circ$ with
$\psi_i$-speed strictly greater than $\delta_i$, for each $1\leq i
\leq n$.
\end{itemize}

In particular, for any cone $C\subset \mathbb R^n$ there exists a
Borel decomposition $X= A\cup S$ where $\mu\llcorner A$ has an Alberti
representation in the $\varphi$-direction of $C$ and $\mathcal
H^1(\gamma\cap S)=0$ for any $\gamma\in\Gamma(X)$ in the
$\varphi$-direction of $C$.
\end{corollary}

\begin{proof}
Fix $\lambda>1$, $k_1,\ldots,k_n\in\mathbb Z$ and let $Y\subset
X$ be compact with
\[\lambda^{-k_i}\geq\Lip(\psi_i,x_0)\geq \lambda^{-k_i-1}\]
for every $x_0\in Y$ and every $1\leq i \leq n$.  Since $Y$ is
compact, it's closed convex hull $\widetilde Y$ is also compact.  By
applying Corollary \ref{cor:basicrep} to $(\widetilde Y,d,\mu)$ we
obtain a Borel decomposition $Y=Y_a\cup Y_s$ such that:
\begin{itemize}
\item There exists an Alberti representation $\mathcal A$ of
  $\mu\llcorner Y_a$ in the $\varphi$-direction of $C$ such that
\[|(\psi_i\circ\gamma)'(t_0)|\geq \delta_i\lambda^{k_i}\Lip(\gamma,t_0)\]
for almost every $\gamma\in\mathcal A$, almost every
$t_0\in\dom\gamma$ and every $1\leq i \leq n$.
\item $\mathcal H^1(\gamma\cap Y_s)=0$ for every $\gamma\in
  \Pi(\widetilde Y)$ in the $\varphi$-direction of $C$ with
\[|\psi_i(\gamma(t))-\psi_i(\gamma(t'))|\geq \delta_i\lambda^{k_i}\|\gamma(t)-\gamma(t')\|\]
for each $t,t'\in\dom\gamma$ and every $1\leq i\leq n$.
\end{itemize}
By our choice of $\lambda$, $\mathcal A$ has $\psi_i$-speed $\delta_i$
for each $1\leq i \leq n$.  Further, by Lemma \ref{lem:nulloncone2},
for any $\gamma\in\Gamma(Y)$ in the $\varphi$-direction of $C^\circ$
with each $\psi_i$-speed strictly greater than $\delta_i\lambda$,
$\mathcal H^1(\gamma\cap S)=0$.

By varying the $k_i\in\mathbb Z$, there exists a countable cover of
$X$ by such $Y$ except for a $\mu$-null set and so, after combining
the respective representations using Lemma \ref{lem:sumreps}, for each
$\lambda>1$ there exists a Borel decomposition $X=A_\lambda\cup
S_\lambda$ where $A_\lambda$ has an Alberti representation in the
$\varphi$-direction of $C$ with $\psi_i$-speed $\delta_i$ for each
$1\leq i \leq n$ and such that $S$ satisfies $\mathcal H^1(\gamma\cap
S)=0$ for every $\gamma$ in the $\varphi$-direction of $C^\circ$ with
each $\psi_i$-speed greater than $\delta_i\lambda$.  Writing
$A=\cup_\lambda A_\lambda$ and $S=\cap_\lambda S_\lambda$ (where the union
and intersection are taken over $\lambda\in\mathbb Q,\ \lambda>1$),
Lemma \ref{lem:sumreps} completes the proof.
\end{proof}

\subsection{Multiple representations}
We now improve upon Corollary \ref{cor:onerep2} by producing multiple
independent Alberti representations of a metric measure space
$(X,d,\mu)$.

Suppose that $\varphi\colon X\to\mathbb R^n$ is Lipschitz and that
$C_1,\ldots,C_n$ is a collection of independent cones in $\mathbb
R^n$.  Then, by multiple applications of Corollary \ref{cor:onerep2},
there exists a Borel decomposition $X= A\cup S$ such that
$\mu\llcorner A$ has $n$ $\varphi$-independent Alberti representations
and $S$ has a Borel decomposition $S=S_1\cup\ldots\cup S_n$ such that,
for each $1\leq i \leq n$, $\mathcal H^1(\gamma\cap S_i)=0$ for each
$\gamma\in\Gamma(X)$ in the $\varphi$-direction of $C_i$.

However, as we will see in the next section, the method of
constructing a non-differentiable Lipschitz function on $S$ requires,
for any $\epsilon>0$, a decomposition as above such that each $C_i$
has width $1-\epsilon$.  Of course, for sufficiently small $\epsilon$,
these cones will not be independent.  So that we may produce
independent Alberti representations and also satisfy this condition
for $S$, we introduce a method that reduces the width of a cone that
defines the direction of an Alberti representation, at the expense of
a countable decomposition.

\begin{corollary}\label{cor:refine}
Let $(X,d,\mu)$ be a metric measure space, $A\subset X$ Borel and
$\varphi,\psi \colon X\to\mathbb R^n$ Lipschitz.  Suppose that, for some cone $C=C(w,\theta)\subset \mathbb
R^n$, $\mu\llcorner A$ has an Alberti representation $\mathcal A$ in
the $\varphi$-direction of $C$ such that $|(\psi_i\circ\gamma)'(t_0)|>0$
for almost every $\gamma\in\mathcal A$, almost every
$t_0\in\dom\gamma$ and every $1\leq i \leq n$.  Then, for any
countable collection of cones $C_m$ with
\[\bigcup_{m\in\mathbb N}C_m^\circ \supset C\setminus \{0\},\]
there exists a countable Borel decomposition $A=\cup_k A_k$ such that each
$\mu\llcorner A_k$ has an Alberti representation $\mathcal A_k$ in the
$\varphi$-direction of some $C_m$ with $\psi_i$-speed strictly greater
than $1/k$, for each $1\leq i \leq n$.

Moreover, if for some $\delta_1,\ldots,\delta_n>0$, $\mathcal A$ has
$\psi_i$-speed strictly greater than $\delta_i$ for each $1\leq i\leq
n$, then so does each $\mathcal A_k$.
\end{corollary}

\begin{proof}
By applying Corollary \ref{cor:onerep2} using each cone $C_m$ in the
hypotheses and each $\delta_i=1/k$, we obtain a countable Borel
decomposition of $A$ into a sets $A_k$ such that each $\mu\llcorner A_k$
has an Alberti representation of the required form and a set $S$ that
satisfies $\mathcal H^1(\gamma\cap S)=0$ for each $\gamma\in\Gamma(X)$
in the $\varphi$-direction of some $C_m^\circ$ and with positive
$\psi_i$-speed for each $1\leq i \leq n$.  Therefore, since the
$C_m^\circ$ cover $C\setminus\{0\}$, $S$ satisfies this for all
$\gamma$ in the $\varphi$-direction of $C$ with positive
$\psi_i$-speed for each $1\leq i \leq n$.  Since $\mu\llcorner A$ is
represented by $\mathcal A$, $\mu(S)=0$.

Now suppose that $\delta'_i>\delta_i$ and that $\mathcal A$ has
$\psi_i$-speed $\delta'_i$ for each $1\leq i \leq n$.  Then if we
repeat the same process but for each $1\leq i \leq n$ take each
$\mathcal A_k$ to have $\psi_i$-speed between $\delta'_i$ and
$\delta_i$, we obtain the same decomposition but with $S$ satisfying
$\mathcal H^1(\gamma\cap S)=0$ for each $\gamma$ in the
$\varphi$-direction of $C$ with $\psi_i$-speed $\delta'_i$ for each
$1\leq i \leq n$.  Again, since $\mu\llcorner A$ is represented
by $\mathcal A$, $\mu(S)=0$.
\end{proof}

\begin{definition}
We will refer to the above process of obtaining new Alberti
representations with the same properties as an existing representation
but in the direction of thinner cones as \emph{refining} an Alberti
representation.
\end{definition}

The requirement in the previous Corollary that both $\varphi$ and
$\psi$ take values in the same Euclidean space is not necessary.
Indeed, suppose that $(X,d,\mu)$ is a metric measure space satisfying
the hypotheses of Corollary \ref{cor:refine} for $\varphi\colon
X\to\mathbb R^n$ and $\psi\colon X\to\mathbb R^m$ Lipschitz.  If $m<n$
then by repeating any $n-m+1$ coordinate functions of $\psi$, we
obtain a Lipschitz function $\widetilde\psi$ into $\mathbb R^n$.
Similarly, if $n<m$, we may define a Lipschitz function
$\widetilde\varphi$ into $\mathbb R^m$ and $\widetilde w\in\mathbb
R^m$ by $\widetilde\varphi_i = \varphi_i$ and $\widetilde w_i=w_i$ for
$1\leq i \leq n$ and $\widetilde \varphi_i=\widetilde w_i=0$ for $n< i
\leq m$.  Then a Lipschitz curve is in the $\varphi$-direction of
$C(w,\theta)$ if and only if it is in the
$\widetilde\varphi$-direction of $C(\widetilde w,\theta)$.  In both
cases, the original hypotheses of Corollary \ref{cor:refine} are
satisfied and it's conclusion gives us Alberti representations of the
required form.

We will require refinements of Alberti representations that maintain
the speed of a vector valued function.  Suppose an Alberti
representation $\mathcal A$ of $(X,d,\mu)$ is in the
$\varphi$-direction of a cone $C$, so that
$\|(\varphi\circ\gamma)'(t_0)\|>0$ for almost every $\gamma\in\mathcal
A$ and almost every $t_0\in\gamma$.  Then we may refine $\mathcal A$
into Alberti representations, each with positive $\varphi$-speed.
Further, suppose that $\mathcal A$ has $\varphi$-speed strictly
greater than $\delta$, for some $\delta>0$.  Then there exists a
countable Borel decomposition $X=\cup_i X_i$ and for each $i$,
rational $\delta_1\ldots,\delta_n>0$ with $\delta_1^2+\ldots
+\delta_n^2>\delta^2$ such that $\mathcal A\llcorner X_i$ has
$\varphi_i$-speed strictly greater than $\delta_i$, for each $1\leq i
\leq n$.  Therefore, we may take arbitrary refinements of $\mathcal A$,
each with $\varphi$-speed strictly greater than $\delta$.

As in previous constructions, we will find multiple independent
Alberti representations of a metric measure space $(X,d,\mu)$ by
decomposing $X$ into two sets; one set on which $\mu$ has many
independent Alberti representations and another, singular, set.  We
now define the general form that such a singular set takes.  In fact,
we will later see that such sets play a fundamental role when
determining whether a metric measure space is a Lipschitz
differentiability space.  These sets, and the methods involving them
in the following section, are a natural generalisation of those
considered in \cite{acp-structurenullsets} in which the authors investigate measures on
Euclidean space with respect to which Rademacher's Theorem holds.

\begin{definition}\label{def:atilde}
Let $(X,d)$ be a metric space, $\varphi\colon X\to\mathbb R^n$
Lipschitz and $\delta,\theta,\lambda>0$.  We define $\widetilde
A(\varphi;\delta,\theta,\lambda)$ to be the set of $S\subset X$ such
that:
\begin{itemize}
\item For every $x_0\in S$,
\[\Lip(v\cdot\varphi,x_0) > \lambda\Lip(\varphi,x_0)\ \forall v\in\mathbb S^{n-1}.
\]
\item There exists a countable Borel decomposition $S=\cup_k
S_k$ and a countable collection of closed cones $C_k\subset \mathbb
R^n$ of width $1-\theta$ such that each $S_k$ satisfies $\mathcal
H^1(\gamma\cap S_k)=0$ for each $\gamma\in\Gamma(X)$ in the
$\varphi$-direction of $C_k$ with $\varphi$-speed $\delta$.
\end{itemize}
Further, we define $\widetilde A(\delta,\theta,\lambda)$ to be the set
of $S\subset X$ for which there exists a countable Borel decomposition
$S=\cup_k S_k$ and a sequence of Lipschitz functions $\varphi_k\colon
X\to\mathbb R^{n_k}$ such that $S_k\in\widetilde
A(\varphi_k;\delta,\theta,\lambda)$ for each $k\in\mathbb N$.
Finally, we define $\widetilde A(\varphi)$, respectively $\widetilde
A$, to be the set of $S\subset X$ for which there exists a $\lambda>0$
such that $S\in\widetilde A(\varphi;\delta,\theta,\lambda)$,
respectively $S\in\widetilde A(\delta,\theta,\lambda)$, for every
$\delta,\theta>0$.
\end{definition}

\begin{remark}
Observe that $\widetilde A(\varphi;\delta,\theta,\lambda)\subset
\widetilde A(\varphi;\delta',\theta',\lambda')$ whenever $\delta\leq
\delta'$, $\theta\geq\theta'$ and $\lambda\leq \lambda'$, and hence
$\widetilde A(\delta,\theta,\lambda)\subset \widetilde
A(\delta',\theta',\lambda')$.
\end{remark}

By combining the previous results of this section we may construct
many independent Alberti representations of a measure, one-by-one.

\begin{proposition}\label{prop:manyrep2}
Let $(U,\varphi)$ be a $\lambda$-structured chart of dimension $n$ in
a metric measure space $(X,d,\mu)$.  Then for any $0<\delta,\theta<1$
there exists a countable Borel decomposition
\[U=S\cup\bigcup_{i\in\mathbb N}U_i\]
such that each $\mu\llcorner U_i$ has $n$ $\varphi$-independent Alberti
representations with $\varphi$-speed strictly greater than $\delta$
and $S$ belongs to $\widetilde
A(\varphi;\delta,\theta,\lambda/\Lip\varphi)$.
\end{proposition}

\begin{proof}
Observe that it suffices to prove the result for $\Lip\varphi=1$.  In
this case, by applying Corollary \ref{cor:onerep2} using a cone of
width $1-\theta$, there exists a Borel decomposition $U=U'\cup S$ such
that $\mu\llcorner U'$ has an Alberti representation with
$\varphi$-speed strictly greater than $\delta$ and $S\in \widetilde
A(\varphi;\delta,\theta,\lambda)$.

Now suppose that, for some $1\leq m<n$, there exists a countable Borel
decomposition $U=\cup_i U_i\cup S$ such that each $\mu\llcorner U_i$ has
$m$ independent Alberti representations with $\varphi$-speed strictly
greater than $\delta$ and $S\in \widetilde
A(\varphi;\delta,\theta,\lambda)$.  Then, by refining if necessary, we
may suppose that the representations of each $\mu\llcorner U_i$ are in
the $\varphi$-direction of cones of width $\alpha$, for some
$0<\alpha<\sqrt{1-\theta^2}$.

Now fix an $i\in\mathbb N$ and let $w_1,\ldots,w_m\in\mathbb S^{n-1}$
such that $C(w_1,\alpha),\ldots,C(w_m,\alpha)$ are independent and
$\mu\llcorner U_i$ has an Alberti representation in the
$\varphi$-direction of $C(w_i,\alpha)$, for each $1\leq i \leq m$.  We
choose $w_{m+1}\in\mathbb S^{n-1}$ orthogonal to $w_1,\ldots,w_m$ so
that, by our choice of $\alpha$,
$C(w_1,\alpha),\ldots,C(w_m,\alpha),C(w_{m+1},\theta)$ are independent
cones in $\mathbb R^n$.  By applying Corollary \ref{cor:onerep2} using
$C(w_{m+1},\theta)$, there exists a Borel decomposition $U_i=U'_i\cup
S_i$ such that $\mu\llcorner U'_i$ has $m+1$ independent Alberti
representations with $\varphi$-speed strictly greater than $\delta$
and $S_i\in \widetilde A(\varphi;\delta,\theta,\lambda)$.

Since $S\cup S_1\cup S_2\cup\ldots \in \widetilde
A(\varphi;\delta,\theta,\lambda)$, we see that there exists a
countable Borel decomposition $X=\cup_i U'_i\cup S'$ such that each
$\mu\llcorner U'_i$ has $m+1$ independent Alberti representations and
$S'\in \widetilde A(\varphi;\delta,\theta,\lambda)$.  By repeating
this process $n-1$ times we obtain the required decomposition.
\end{proof}

\begin{theorem}\label{thm:manyrep}
Let $(U,\varphi)$ be a $\lambda$-structured chart of dimension $n$ in
a metric measure space $(X,d,\mu)$.  Then there exists a
countable Borel decomposition $U=\cup_i U_i$ such that each $\mu\llcorner
U_i$ has $n$ $\varphi$-independent Alberti representations if and only
if, for every $S\in\widetilde A(\varphi)$ with $S\subset U$, $\mu(S)=0$.
\end{theorem}

\begin{proof}
First suppose that, for every $S\in\widetilde A(\varphi)$ with
$S\subset U$, $\mu(S)=0$.  By the previous Proposition, for each $m\in\mathbb N$
there exists a Borel decomposition $U=U_m\cup S_m$ such that $U_m$ has a
countable decomposition of the required form and $S_m\in\widetilde
A(\varphi;1/m,1/m,\lambda)$.  This gives a Borel decomposition
$U=S\cup U_1\cup U_2\cup\ldots$ where each $U_i$ is of the required
form and $S=\cap_m S_m\in\widetilde A(\varphi)$, as required.

Conversely, suppose that such a decomposition exists.  Then by
refining each Alberti representation if necessary, for each $U_j$ we
may suppose that there exists a $\delta>0$ and independent cones
$C_1,\ldots,C_n$ such that for each $1\leq i \leq n$, $\mu\llcorner U_j$
has an Alberti representation in the $\varphi$-direction of $C_i$ with
$\varphi$-speed $\delta$.  For such a $U_j$ there exists an
$m\in\mathbb N$ such that $1/m<\delta$ and such that any cone of width
$1-1/m$ must completely contain one of the $C_i$.  Since $\mu\llcorner
U_j$ has the above Alberti representations, for any
$S\in\widetilde A(\varphi;1/m,1/m,\lambda)$, $\mu(S\cap U_j)=0$.  In
particular, for any $S\in\widetilde A(\varphi)$, $\mu(S\cap U_j)=0$
and hence $\mu(S\cap U)=0$.
\end{proof}


\section{First Alberti representations in differentiability spaces}\label{sec:spanreps}
By combining our previous results, we now prove the existence of
independent Alberti representations in Lipschitz differentiability
spaces.  We will apply Theorem \ref{thm:manyrep} to construct
independent representations and will show, for any chart $(U,\varphi)$
and $S\in\widetilde A(\varphi)$, that $\mu(S\cap U)=0$ by constructing
a Lipschitz function that is differentiable almost nowhere in such a
set, via Proposition \ref{prop:cons}.

\subsection{Constructions in Banach spaces}
We will again embed our metric space into a Banach space.  This
provides many Lipschitz curves that we may use to define the Lipschitz
functions required for Proposition \ref{prop:cons}.

In this subsection we fix the following notation for simplicity.

\begin{notation}\label{not:banachcurves}
For $n\in\mathbb N$ we let $\mathcal B$ to be the Banach space
defined in Definition \ref{def:embedding}, let $\varphi\colon
\mathcal B\to \mathbb R^n$ be the projection onto the first factor and fix
$w\in\mathbb S^{n-1}$.  We also fix $X$ a compact subset of $\mathbb
R^{n} \times \{0\} \times \ell_{\infty} \subset \mathcal B$,
$0<\delta,\theta<1$ and $S\subset X$ closed with $\mathcal
H^1(\gamma\cap S)=0$ for every $\gamma\in\Gamma(X)$ in the
$\varphi$-direction of $C(w,\theta)$ with $\varphi$-speed $\delta$.

For each $t\in\mathbb R$ we define $\mathcal P_t\colon \mathbb
R^n\to\mathbb R^n$ to be the orthogonal projection onto
\[\{v\in\mathbb R^n:(v-tw)\cdot w =0\}\]
and define $P_t\colon \mathcal B\to\mathcal
B$ by
\[P_t(u,v,s)=(\mathcal P_t(u),v,s).\]
We also define $\mathcal P=\mathcal P_0$,
$\varphi_0=\inf\{\varphi(x)\cdot w:x\in X\}$, $P=P_{\varphi_0}$ and
$\Omega$ to be the closed, convex hull of $P(X)\cup X$, a compact,
convex set.

Further, for any measurable $I\subset \mathbb R$ and Lipschitz
$\gamma\colon I\to\Omega$, we express
\begin{equation*}
\gamma = (\varphi\circ\gamma, 0, \gamma^{*})
\end{equation*}
and write $\mathcal V_{\gamma}(t)= \Lip(\gamma^{*},t)$, so that
\begin{equation*}
  \Lip(\gamma, t) = \operatorname{max}\{\|(\varphi\circ\gamma)'(t)\|,
  \mathcal V_{\gamma}(t)\} 
\end{equation*}
for almost every $t \in I$.

In addition, for any Borel $V\subset\mathcal B$ and $\epsilon>0$, we
define
\begin{equation*}
  Q(V,\gamma, \epsilon) = \int_{\gamma^{-1}(\Omega\setminus V)}
  (\varphi\circ\gamma)' \cdot w + \int_{\gamma^{-1}(\Omega)}
  \left(  K(\theta) \|\mathcal
  P((\varphi\circ\gamma)') \| + \delta \mathcal V_{\gamma}
\right)
 + \epsilon\mathcal H^1(\gamma),
\end{equation*}
where $K(\theta)=(1-\theta)/\sqrt{\theta(2-\theta)}$.
\end{notation}

\begin{lemma}\label{lem:fundobs}
For any $\epsilon>0$ there exists a $\mathcal B$-open set $V\supset S$
such that, for any compact interval $I$ and Lipschitz $\gamma\colon
I\to\Omega$ with $(\varphi\circ\gamma)'\cdot w\geq 0$ almost
everywhere,
\[Q(V,\gamma, \epsilon) \geq (\varphi(\gamma_e)-\varphi(\gamma_s))\cdot w -\epsilon,\]
where $\gamma_s$ and $\gamma_e$ are the endpoints of $\gamma$.
\end{lemma}

\begin{proof}
Suppose that the conclusion is false.  Then there exists an
$\epsilon>0$ and for each $m\in\mathbb N$ a Lipschitz $\gamma_m$ that
violates the inequality for the open set $V_m=B(S,1/m)$.  Since
$\Omega$ is compact, by replacing $\epsilon$ by $\epsilon/2$ if
necessary, we may suppose that each $\gamma_m$ has the same endpoints,
$\gamma_s$ and $\gamma_e$.  In particular, for each $m\in\mathbb N$,
\[\mathcal H^1(\gamma_m)\leq Q(V_m,\gamma_m, \epsilon)/\epsilon \leq
(\varphi(\gamma_e))-\varphi(\gamma_s))\cdot w/\epsilon.
\]
Therefore, there exists an $L>0$ and a reparametrisation of each
$\gamma_m$ such that each is a 1-Lipschitz function defined on
$[0,L]$.  Further, by the Arzel\`a-Ascoli theorem and after possibly
choosing a suitable subsequence, there exists a 1-Lipschitz
$\gamma\colon [0,L]\to\Omega$ such that $\gamma_m\to\gamma$ uniformly.
Moreover, since $(\varphi\circ\gamma_m)'\geq 0$ almost everywhere for
each $m\in\mathbb N$, $(\varphi\circ\gamma)'\geq 0$ almost everywhere.
We will show that the image of $\gamma$ intersects $S$ in a set of
positive measure from which we deduce a contradiction to our
hypothesis.

Fix an $m\in\mathbb N$ and $\eta>0$.  Since
$\beta=\gamma^{-1}(\Omega\setminus \overline{V_m})$ is an open subset
of $\mathbb R$ of finite measure, there exist disjoint compact intervals
$\beta_1,\ldots,\beta_N\subset\beta$ such that $\mathcal
L^1(\beta\setminus \cup_i \beta_i)\leq \eta$.  Therefore,
\begin{equation*}
   \int_{\gamma^{-1}(\Omega\setminus \overline{V_{m}})}
   (\varphi\circ\gamma)' \cdot w \leq \int_{\cup_{i}\beta_{i}}(\varphi
   \circ \gamma)' \cdot w + \eta \Lip\varphi.
\end{equation*}
Moreover, by the lower semicontinuity of the total variation of
Lipschitz functions under uniform convergence,
\begin{equation*}
   \int_{\cup_{i\beta_{i}}} (\varphi\circ\gamma)' \cdot w \leq
   \liminf_{k \to \infty}\int_{\cup_{i}\beta_{i}}(\varphi
   \circ \gamma_{k})' \cdot w.
\end{equation*}
Now, since $\gamma_k\to\gamma$ uniformly and each $\gamma(\beta_i)$
is compact, there exists a $K\in\mathbb N$ such that, for all $k\geq
K$ and $1\leq i \leq N$, $\gamma_k(\beta_i)\subset \Omega \setminus
\overline{V_m}$.  Therefore, since $(\varphi \circ \gamma_{k})' \cdot
w \geq 0$ almost everywhere for each $k$,
\begin{equation*}
   \int_{\gamma^{-1}(\Omega\setminus \overline{V_{m}})}
   (\varphi\circ\gamma)' \cdot w \leq \liminf_{k\to\infty}\int_{\gamma_{k}^{-1}(\Omega\setminus \overline{V_{m}})}(\varphi
   \circ \gamma_{k})' \cdot w + \eta \Lip\varphi.
\end{equation*}
Further, by again using the lower semicontinuity of total variation
and of $\alpha \mapsto \mathcal H^{1}(\alpha)$,
\[Q(\overline{V_m},\gamma,\epsilon) \leq \liminf_{k\to\infty}
Q(\overline{V_m},\gamma_k,\epsilon) +\eta \Lip\varphi.\]
Finally, since $(\varphi\circ\gamma)'\geq 0$ almost everywhere and
$V_k\subset V_m$ for every $k\geq m$,
\[Q(\overline{V_m},\gamma, \epsilon) \leq (\varphi(\gamma_e)-\varphi(\gamma_s))
\cdot w -\epsilon + \eta \Lip\varphi.
\]

This is true for all $m\in\mathbb N$ and $\eta>0$.  Moreover, since
$S$ is closed, $\cap_m \overline{V_m}=S$ and so
\[Q(S,\gamma,\epsilon)\leq (\varphi(\gamma_e)-\varphi(\gamma_s))\cdot w -\epsilon.\]
In particular, by the fundamental theorem of calculus,
\begin{equation}\label{eq:fundobspt1} \epsilon\leq \int_{\gamma^{-1}(S)} (\varphi\circ\gamma)'\cdot w -
\int_{\gamma^{-1}(\Omega)} \left( K(\theta)\|\mathcal
  P((\varphi\circ\gamma)')\| +\delta\mathcal V_{\gamma}\right) -\epsilon\mathcal H^1(\gamma).
\end{equation}
On the other hand, when $\gamma$ is restricted to
\[D:=\{t_0: (\varphi\circ\gamma)'(t_0) \cdot w >
(1-\theta)\|(\varphi\circ\gamma)'(t_0)\| \geq (1-\theta)\delta
\Lip(\gamma,t_0)\},\] it is a Lipschitz function in the
$\varphi$-direction of $C$ with $\varphi$-speed $\delta$.  Moreover,
$\Lip(\gamma,t_0)>0$ for every $t_0\in D$ and so we may decompose $D$
into a Lebesgue null set and a countable collection of compact sets
$K_i$ on which $\gamma$ is bi-Lipschitz.  In particular, each
$\restr{\gamma}{K_i}\in\Gamma(X)$ and so $D\cap\gamma^{-1}(S)$ is
Lebesgue null.  Therefore, for almost every $t_{0} \in \gamma^{-1}(S)$,
either
\begin{equation*}
  (\varphi \circ \gamma)'(t_{0}) \leq (1-\theta) \|(\varphi \circ
  \gamma)'(t_{0}) \| \leq K(\theta) \|\mathcal P(\varphi\circ\gamma)'(t_{0})\|,
\end{equation*}
or
\begin{equation*} \|(\varphi\circ\gamma)'(t_0)\| \leq \delta
\Lip(\gamma,t_0) \text{ and so }
\Lip(\gamma,t_0) = \mathcal V_{\gamma}(t_{0}).
\end{equation*}
Thus
\begin{equation}\label{eq:fundobspt2}\int_{\gamma^{-1}(S)}(\varphi\circ\gamma)'\cdot w\leq
  \int_{\gamma^{-1}(S)}\left( K(\theta) \|\mathcal
    P((\varphi\circ\gamma)')\|+ \delta\mathcal V_{\gamma}\right).
\end{equation}
By combining equations \eqref{eq:fundobspt1} and \eqref{eq:fundobspt2}
we obtain $\epsilon\leq 0$, giving the required contradiction.
\end{proof}

We use this result to construct a Lipschitz function on $\Omega$
defined via Lipschitz curves connecting points in $\Omega$.  We will
see that functions of this form have the properties required to apply
Proposition \ref{prop:cons}.

\begin{lemma}\label{lem:firstfunct}
For any $\epsilon>0$ there exists a $\mathcal B$-open set $V\supset S$
and a $(1+K(\theta)+\delta+\epsilon)$-Lipschitz function $f\colon
\Omega\to\mathbb R$ such that:
\begin{enumerate}
\item \label{item:funcprop2} For every $x,x_0\in \Omega$ with $(\varphi(x)-\varphi(x_0))\cdot w\geq 0$,
\[f(x)-f(x_0)\geq(\varphi(x)-\varphi(x_0))\cdot w-\epsilon.\]
\item \label{item:funcprop1}For every $y,z$ contained in a ball $B\subset V$,
\[|f(y)-f(z)|\leq K(\theta)\|\mathcal
P(\varphi(y)-\varphi(z))\|+(\delta+\epsilon)\|y-z\|.\]
\end{enumerate}
\end{lemma}

\begin{proof}
For $\epsilon>0$ let $V\supset S$ be the $\Omega$-open set obtained
from an application of Lemma \ref{lem:fundobs}.  By our definition of
$\Omega$, for every $x\in \Omega$ the straight line segment joining
$P(x)$ to $x$ lies in $\Omega$.  Therefore we may define a function
$f\colon \Omega \to\mathbb R$ by
\[f(x)=\inf Q(V,\gamma, \epsilon)\]
where the infimum is taken over all $l\geq 0$ and all Lipschitz
$\gamma\colon [0,l]\to\Omega$ with $(\varphi\circ\gamma)'\cdot w \geq
0$ almost everywhere such that $\gamma(0)\in P(\Omega)$ and
$\gamma(l)=x$.  We will call such a curve \emph{admissible for $x$}.

We now use the conclusion of Lemma \ref{lem:fundobs} to deduce the
required properties of $f$.  We first show that $f$ is Lipschitz and
satisfies \eqref{item:funcprop1}.  Indeed, let $y,z\in\Omega$ with
$(\varphi(z)-\varphi(y))\cdot w\geq 0$ and let $\gamma\colon
[0,l]\to\Omega$ be any admissible curve for $y$.  Define
$\widetilde\gamma\colon [0,l+1]\to \mathcal B$ by
\[\widetilde\gamma(t)=\begin{cases}\gamma(t) & t\in[0,l]\\
y+(t-l)(z-y) & t\in (l,l+1].\end{cases}
\]
Then since $(\varphi(z)-\varphi(y))\cdot w\geq 0$,
$(\varphi\circ\widetilde\gamma)' \geq 0$ almost everywhere and so
$\widetilde\gamma$ is admissible for $z$.  Therefore
\begin{align}\nonumber f(z)&\leq f(y) + Q(V,\restr{\widetilde\gamma}{[l,l+1]})\\
\nonumber&\leq f(y) + \mathcal H^1(\widetilde\gamma([l,l+1])\setminus V)+
  K(\theta) \|\mathcal P(\varphi(y)-\varphi(z))\| + (\delta+\epsilon)\mathcal H^1(\widetilde\gamma([l,l+1]))\\
\label{eq:firstfunct0}&\leq f(y) + K(\theta)\|\mathcal P(\varphi(y)-\varphi(z))\|+(\delta+\epsilon)\|y-z\| \text{ if }y,z\in\text{Ball}\subset V\\
\label{eq:firstfunct1}&\leq f(y) + K(\theta)\|\mathcal
  P(\varphi(y)-\varphi(z))\|+(1 +\delta+\epsilon)\|y-z\| \text{ otherwise.}
\end{align}

To bound $f(y)$, let $\gamma\colon [0,l]\to \Omega$ be admissible
for $z$ and set
\[t_0=\inf\{t\in I:\varphi(\gamma(t))\cdot w\geq \varphi(y)\cdot w\}.\]
Then since $(\varphi\circ\gamma)'\geq 0$ almost everywhere,
$(\varphi(\gamma(t))- \varphi(y))\cdot w \geq 0$ for all $t\geq t_0$.
We define $\widetilde\gamma$ by
\[\widetilde\gamma(t)=\begin{cases} \gamma(t) & \text{if }0\leq t\leq t_0\\
P_{\varphi(y)\cdot w}(\gamma(t)) & t_0<t\leq
l\\ \widetilde\gamma(l)+(t-l)(y-\widetilde\gamma(l)) & t\in
(l,l+1]. \end{cases}
\]
Then $\widetilde\gamma(l+1)=y$ and $(\varphi\circ \widetilde
\gamma)'(t)\cdot w = 0$ for almost every $t\in [t_0,l+1]$, so that
$\widetilde\gamma$ is admissible for $y$.  Further, for almost every
$t\in [t_0,l]$,
\[\|\mathcal P((\varphi\circ\widetilde\gamma)'(t))\| \leq \|\mathcal
P((\varphi\circ\gamma)'(t))\| \text{ and }
\mathcal V_{\widetilde\gamma}(t) = \mathcal V_{\gamma}(t).
\]
Therefore
\[Q(V,\restr{\widetilde\gamma}{[t_0,l]}) \leq Q(V,\restr{\gamma}{[t_0,l]})
\]
and so
\begin{equation}\label{eq:firstfunct2}f(y)\leq f(z) + K(\theta)\|\mathcal P(\varphi(y)-\varphi(z))\|+(\delta+\epsilon)\|y-z\|.\end{equation}
By combining equations \eqref{eq:firstfunct0}, \eqref{eq:firstfunct1} and
\eqref{eq:firstfunct2} we see that $f$ is a
$(1+K(\theta)+\delta+\epsilon)$-Lipschitz function such that, for every $y,z$
belonging to a ball contained in $V$,
\[|f(y)-f(z)|\leq K(\theta)\|\mathcal P(\varphi(y)-\varphi(z))\|+(\delta+\epsilon) \|y-z\|.\]

To prove \eqref{item:funcprop2} let $x,x_0\in\Omega$ with
$(\varphi(x)-\varphi(x_0))\cdot w \geq 0$.  By the definition of
$\Omega$ and $P$, the straight line segment $\gamma$ joining $P(x_0)$
to $x_0$ lies entirely within $\Omega$ and satisfies
$(\varphi\circ\gamma)'\geq 0$ almost everywhere and so is
admissible for $x_0$.  Further, $\mathcal
P((\varphi\circ\gamma)'(t_{0}))=0= \mathcal V_{\gamma}(t_{0})$ for almost
every $t_{0}$.  Therefore,
\[f(x_0)\leq (\varphi(x_0)-\varphi(P(x_0)))\cdot w + \epsilon \|x_0- P(x_0)\|.\]
Further, by the conclusion of Lemma \ref{lem:fundobs},
\[f(x)\geq (\varphi(x)-\varphi(P(x)))\cdot w -\epsilon.\]
Therefore, since $(\varphi(P(x))-\varphi(P(x_0)))\cdot w=0$,
\[f(x)-f(x_0)\geq (\varphi(x)-\varphi(x_0))\cdot w - \epsilon
(1+\operatorname{diam}X).\]
This proves the result for $(1 + \operatorname{diam} X)\epsilon$,
which suffices.
\end{proof}

\subsection{Application to Lipschitz differentiability spaces}
We first apply our previous construction to subsets of metric measure spaces.

\begin{lemma}\label{lem:firstfunlds}
Let $(X,d,\mu)$ be a metric measure space, $U\subset X$ compact and
$\varphi\colon X\to\mathbb R^n$ Lipschitz.  Suppose that, for some
$w\in\mathbb S^{n-1}$, $0<\theta<1$ and $\delta>0$, $S\subset U$ is
closed and satisfies $\mathcal H^1(\gamma\cap S)=0$ for any
$\gamma\in\Gamma(X)$ in the $\varphi$-direction of $C(w,\theta)$ with
$\varphi$-speed $\delta$.  Then for any $\epsilon>0$ there exists a
$(1+K(\theta)+\delta+\epsilon)\Lip\varphi$-Lipschitz function $f\colon U\to\mathbb R$
and a $\rho>0$ such that:
\begin{itemize}
\item For every $x_0\in S$ and $x\in U$ with $(\varphi(x)-\varphi(x_0))\cdot w\geq 0$,
\[f(x)-f(x_0)\geq (\varphi(x)-\varphi(x_0))\cdot w -\epsilon\]
\item For every $x_0\in S$ and $y,z\in B(x_0,\rho)$,
\[|f(y)-f(z)|\leq K(\theta)\|\mathcal P(\varphi(y)-\varphi(z))\|+(\delta+\epsilon)\Lip\varphi d(y,z).\]
\end{itemize}
Here $\mathcal P\colon\mathbb R^n\to\mathbb R^n$ is the orthogonal
projection onto the hyperplane orthogonal to $w$ passing through the
origin.
\end{lemma}

\begin{proof}
Let us identify $U$ with its image in $\mathcal B$ via the bi-Lipschitz
isomorphism $\iota^*$ defined in Definition \ref{def:embedding} (for
$\psi=0$).  Note that $\varphi$ agrees with the projection onto the
first factor on $U$ and so we may extend $\varphi$ to all of $\mathcal
B$ by defining it to be this projection.  Then $S$ is also a subset of
$\mathcal B$ that satisfies $\mathcal H^1(\gamma\cap S)=0$ for every
$\gamma\in\Gamma(U)$ in the $\varphi$-direction of $C(w,\theta)$ with
$\varphi$-speed $\delta$.  We denote by $\|x-y\|$ the distance between
$x$ and $y$ in $\mathcal B$ and by $d(x,y)$ the original distance
between $x$ and $y$ in $U$, so that
\[d(x,y)\leq \|x-y\|\leq \Lip\varphi d(x,y).\]

Define $f$ to be the Lipschitz function and $V$ the open set
obtained from an application of Lemma \ref{lem:firstfunct}.  Since
$U\subset V$ is compact there exists a $\rho>0$ such that, for every
$x_0\in S$, $B(x_0,\rho)\subset V$.  This function has the
required properties.
\end{proof}

We now combine the functions constructed in the previous Lemma to
satisfy the hypotheses of Proposition \ref{prop:cons} for arbitrarily
large subsets of any $S\in\widetilde A(\varphi)$, for any structured
chart $(U,\varphi)$.

\begin{lemma}\label{lem:aptolds}
Let $(U,\varphi)$ be a structured chart in a metric measure space
$(X,d,\mu)$ and suppose that $S \in \widetilde A(\varphi)$ with
$S\subset U$.  Then
there exists a $\beta>0$ and, for any $\delta>0$, a
1-Lipschitz function $f\colon X \to\mathbb R$, a compact
$S'\subset S$ with $\mu(S')\geq\mu(S)-\delta$
and a $\rho>0$ such that:
\begin{itemize}
\item For every $x_0\in S'$ there exists an $x\in U$ with $0<d(x,x_0)<\delta$ and
\[|f(x)-f(x_0)|\geq \beta d(x,x_0).\]
\item For every $x_0\in S$ and $y,z\in B(x_0,\rho)$,
\[|f(y)-f(z)|\leq \delta d(y,z).\]
\end{itemize}
\end{lemma}

\begin{proof}
  Note that it suffices to prove the result for any 1/2-Lipschitz
  $\varphi$ and $0<\delta<1/2$.  If so, let $\lambda>0$ and
  $0<\theta<1$ such that $(U,\varphi)$ is a $\lambda$-structured chart
  and $K(\theta)\leq \delta/2$.  Then there exists a countable Borel
  decomposition $S=\cup_i S_i$ and for each $i\in\mathbb N$ a
  $w_i\in\mathbb S^{n-1}$ such that $\mathcal H^1(\gamma\cap S_i)=0$
  for any $\gamma\in\Gamma(X)$ in the $\varphi$-direction of
  $C(w_i,\theta)$ with $\varphi$-speed $\delta$.

For every $i\in\mathbb N$, let $Q_i\subset S_i$ be disjoint and compact
and let $N\in\mathbb N$ such that
\[\mu(Q_1\cup\ldots\cup Q_N)> \mu(S)-\delta.\]
Then there exists a $h>0$ such that the $B(Q_i,2h)$ are disjoint.
Further, let $0<4R<\min\{h,\delta\}$ such that $(1-4R/h)\mu(Q_1\cup\ldots\cup
Q_N)>\mu(S)-\delta$.  Then, since $(U,\varphi)$ is a structured chart,
there exists an $r>0$ such that, for every $x_0\in U$, there exists an $x\in X$ with
$r<d(x,x_0)<R$ and
\[\frac{|(\varphi(x)-\varphi(x_0))\cdot w_i|}{d(x,x_0)}\geq
\lambda.
\]
We also set $\epsilon = \lambda r/2$ and, for
each $1\leq i \leq N$, define
\[g_i\colon B(Q_i,h)\to\mathbb R\]
to be the 1-Lipschitz function obtained from Lemma
\ref{lem:firstfunlds} restricted to $B(Q_i,h)$ for this choice of
$\epsilon$, $\theta$ and $w_i$.  Further, we let $f_i\colon
B(Q_i,h)\to\mathbb R$ and $P_i\subset Q_i$ be obtained from applying
Lemma \ref{lem:modifications} to $g_i$ and $Q_i$ with the choice of
$\epsilon=R$.  We extend each $f_i$ to a 1-Lipschitz
function defined on $X$ with value zero outsize
$B(Q_i,2h)$.

Then for any $x_0\in P_i$ there exists an $x\in X$ with $r<d(x,x_0)<R<\delta$ and
\begin{align*}|f_i(x)-f_i(x_0)| &= |g_i(x)-g_i(x_0)|\\
&\geq |(\varphi(x)-\varphi(x_0))\cdot w_i| -\epsilon\\
&\geq \lambda d(x,x_0)-\lambda r/2\\
&\geq \lambda d(x,x_0)/2.\end{align*}
Further, we let $\rho>0$ such that, for any $y,z\in \text{Ball}\subset B(Q_i,\rho)$,
\begin{align*} |f_i(y)-f_i(z)| &\leq |g_i(y)-g_i(z)|\\
&\leq K(\theta)\|\mathcal P(\varphi(y)-\varphi(z))\|+(\delta+\epsilon) d(y,z)\\
&\leq 2\delta d(y,z).\end{align*}
Finally we set $S'=P_1\cup\ldots\cup P_N$ so that
\[\mu(S')\geq (1-4R/h)\mu(Q_1\cup\ldots\cup Q_N)\geq \mu(S)-\delta\]
and
\[f=\sum_{1\leq i\leq N} g_i.\]
Then since the $g_i$ have disjoint support, $f$ and $S'$ satisfy the
conclusion of the Lemma for $2\delta$ and $\beta=\lambda/2$.
\end{proof}

We now combine all of our previous results
to obtain our first statement on the structure of measures in
Lipschitz differentiability spaces.

\begin{theorem}\label{thm:spanrep}
Let $(U,\varphi)$ be an $n$-dimensional chart in a Lipschitz
differentiability space $(X,d,\mu)$.  Then any $\widetilde A(\varphi)$
subset of $U$ is $\mu$-null.  Therefore, there exists a countable
Borel decomposition
\[U=\bigcup_{k\in\mathbb N} U_k\]
such that, for each $k\in\mathbb N$, $\mu\llcorner U_k$ has $n$
$\varphi$-independent Alberti representations.

Moreover, for any Lipschitz $\psi\colon X\to\mathbb R^m$ and
$U'\subset U$ of positive measure, if $U'$ has $m$ $\psi$-independent
Alberti representations then $m\leq n$.
\end{theorem}

\begin{proof}
By Lemma \ref{lem:decompstruct} it suffices to prove that any
$\widetilde A(\varphi)$ subset $S$ of a structured chart $(U,\varphi)$
is $\mu$-null.  If $S$ is such a set, for any $\epsilon>0$ and each
$m\in\mathbb N$ let $0<\delta_m<1/m$ such that $\sum_m\delta_m<\epsilon$
and let $f_m\colon X\to\mathbb R$ and $S_m\subset S$ be obtained by
applying Lemma \ref{lem:aptolds} with $\delta_m$.  Then $S':=\cap_m
S_m$ satisfies $\mu(S')\geq \mu(S)-\epsilon$ and the $f_m$ satisfy the
hypotheses of Proposition \ref{prop:cons} for $S'$.  Therefore there
exists a Lipschitz function differentiable $\mu$-almost nowhere on
$S'$.  Thus $S'$ and hence $S$ are $\mu$ null.

In particular, any $\widetilde A(\varphi)$ subset of $U$ is $\mu$-null and
so, by Theorem \ref{thm:manyrep}, there exists the required collection
of Alberti representations.

Now suppose that $(U,\varphi)$ is any chart, $\psi\colon X\to\mathbb
R^m$ is Lipschitz and $U'\subset U$ has $m$ $\psi$-independent Alberti
representations.  Then for almost every $x_0\in U'$ we have
$\Lip(v\cdot\psi,x_0)>0$ for each $v\in\mathbb S^{m-1}$.  However, for
almost every $x_0\in U$, each $D\psi_i(x_0)$ exists.  Therefore, if
$m>n$, there exists a $v\in\mathbb S^{m-1}$ such that $\sum_i v_i
D\psi_i(x_0)=0$.  In particular we have $\Lip(v\cdot\psi,x_0)=0$ and
so $U'$ must be $\mu$-null.
\end{proof}

We now apply our theory of Alberti representations to charts, relating
the behaviour of Lipschitz differentiability spaces to the existing
differentiability theory of Euclidean spaces.  Recall the notion of a
gradient given in Definition \ref{def:gradient}.

\begin{corollary}\label{cor:tangentcurves}
Let $(U,\varphi)$ be an $n$-dimensional chart in a Lipschitz
differentiability space $(X,d,\mu)$.  Then for almost every $x\in U$
there exists $\gamma_1^x,\ldots,\gamma_n^x\in\Gamma(X)$ such that each
$(\gamma_i^x)^{-1}(x)=0$ is a density point of $(\gamma_i^x)^{-1}(U)$
and the $(\varphi\circ\gamma_i^x)'(0)$ are linearly
independent.

Moreover, for any such $\gamma_i^x$, for any Lipschitz $f\colon
X\to\mathbb R$ and almost every $x\in U$, the gradient of $f$ at $x$
with respect to $\varphi$ and $\gamma_1^x,\ldots,\gamma_n^x$ equals $Df(x)$.
\end{corollary}

\begin{proof}
By the previous Theorem there exists a countable Borel decomposition $U=\cup_i
U_i$ of $U$ into sets with $n$ $\varphi$-independent Alberti
representations.  Therefore, by applying Proposition
\ref{prop:curvefullmeas} to each representation, for almost every $x$
in any $U_i$ there exists such $\gamma_1^x,\ldots,\gamma_n^x$.

Moreover, if $Df(x)$ exists, then since
\[(f\circ\gamma_i^x)'(0)= Df(x)\cdot (\varphi\circ\gamma_i^x)'(0),\]
$Df(x)$ equals the gradient of $f$ at $x$ with respect to $\varphi$
and $\gamma_1^x,\ldots,\gamma_n^x$.
\end{proof}

The previous Corollary should be compared to \cite{cheegerkleiner-rnp},
Theorem 3.3.  This theorem asserts that, for any $n$ dimensional chart
in a doubling Lipschitz differentiability space that satisfies the
Poincar\'e inequality, and for any collection $f_1,f_2 \ldots$ of
Lipschitz functions, for almost every $x \in U$ there exist
$\gamma_1^x,\ldots, \gamma^x_n \in\Gamma(X)$ (whose domains are in
fact intervals) such that, for each $i\in\mathbb N$, the gradient of
$f_i$ at $x$ with respect to $\varphi$ and $\gamma^x_1,\ldots,
\gamma^x_n$ equals $Df_i(x)$.

We now use the derivative of a Lipschitz function to show
that any $\widetilde A$ subset of a Lipschitz differentiability space
has measure zero.  For this, we first investigate how the direction of
a Lipschitz curve varies with respect to different Lipschitz
functions in a Lipschitz differentiability space.

\begin{lemma}\label{lem:preatildenull}
Let $(U,\varphi)$ be an $n$-dimensional $\lambda$-structured chart in
a Lipschitz differentiability space $(X,d,\mu)$, $S\subset U$ Borel
and $\eta>0$.  Suppose that $\psi\colon X\to\mathbb R^m$ is Lipschitz
and
\[\Lip(v\cdot\psi,x_0)> \eta\Lip(\psi,x_0)\]
for every $x_0\in S$ and $v\in\mathbb S^{m-1}$.  Then for any
$\widetilde w\in\mathbb S^{m-1}$ and $0<\epsilon<\theta<1$, there exists a
countable Borel decomposition $S=\cup_i S_i\cup N$ where $\mu(N)=0$
and, for each $i\in\mathbb N$, there exists a $w_i\in\mathbb
S^{n-1}$ such that, for any $\delta>0$, any $\gamma\in\Gamma(S_i)$ in
the $\varphi$-direction of $C(w_i,\theta-\epsilon)$ with $\varphi$-speed
$\delta$ is in the $\psi$-direction of $C(\widetilde w,
1-(1-\theta)\lambda\eta/\Lip\varphi)$ with $\psi$-speed $\delta\eta$.
\end{lemma}

\begin{proof}
We use the derivative of each $\psi_i$ to transform the direction of
such a Lipschitz curve.  Indeed, for $x_0\in S$ suppose that the
derivative, $D\psi_i(x_0)$, of each component of $\psi$ exists at
$x_0$.  If we write $D\psi\colon \mathbb R^m\to\mathbb R^n$ for the
linear map whose columns are the $D\psi_i(x_0)$, then for each
$v\in\mathbb R^m$,
\[\Lip(D\psi(v)\cdot\varphi,x_0)=\Lip(v\cdot\psi,x_0)\geq
\eta \|v\| \Lip(\psi,x_0)\]
and so
\begin{equation}\label{eq:psione}\|D\psi(v)\|\Lip(\varphi,x_0)\geq
  \eta \|v\| \Lip(\psi,x_0).\end{equation}
We also obtain
\begin{equation}\label{eq:psitwo}\lambda\|D\psi(v)\|\leq\|v\|\Lip(\psi,x_0).\end{equation}
In particular, $D\psi$ is injective and it's inverse
$D\psi^{-1}\colon \im D\psi\to\mathbb R^n$ satisfies the inequalities
corresponding to \eqref{eq:psione} and \eqref{eq:psitwo}.

Now let $w\in\mathbb S^{n-1}\cap\im D\psi$, $0<\theta<1$ and suppose that
$\gamma\in\Gamma(X)$ satisfies $\gamma(t_0)=x_0$ and that both
$(\varphi\circ\gamma)'(t_0)$ and $(\psi\circ\gamma)'(t_0)$ exist.
Then, for every $v\in\mathbb R^m$,
\[v\cdot(\psi\circ\gamma)'(t_0)=D\psi(v)\cdot(\varphi\circ\gamma)'(t_0).\]
Therefore, if $(\varphi\circ\gamma)'(t_0)\in C(w,\theta)$, by equation
\eqref{eq:psitwo},
\begin{align*}D\psi(D\psi^{-1}(w))\cdot(\varphi\circ\gamma)'(t_0) &=
  w\cdot(\varphi\circ\gamma)'(t_0)\\
&\geq (1-\theta)\|(\varphi\circ\gamma)'(t_0)\|\\
&= (1-\theta) \|D\psi^{-1}\cdot (\psi\circ\gamma)'(t_0)\|\\
&\geq (1-\theta) \lambda \|(\psi\circ\gamma)'(t_0)\|/\Lip(\psi,x_0)
\end{align*}
If we let $\widetilde w\in\mathbb R^m$ be a scalar multiple of
$D\psi^{-1}(w)$ with norm 1, then since
\[\|D\psi^{-1}\|\leq \Lip(\varphi,x_0) /\eta\Lip(\psi,x_0),\]
the previous inequality gives
\begin{align*} \widetilde w\cdot (\psi\circ\gamma)'(t_0) &\geq
  \frac{(1-\theta)\lambda\|(\psi\circ\gamma)'(t_0)\|}{\|D\psi^{-1}(w)\|\Lip(\psi,x_0)}\\
&\geq \frac{(1-\theta)\lambda\eta
    \|(\psi\circ\gamma)'(t_0)\|}{\Lip(\varphi,x_0)}.
\end{align*}
Therefore $(\psi\circ\gamma)'(t_0)\in C(\widetilde
w,1-(1-\theta)\lambda\eta/\Lip\varphi)$.

Similarly, if $\|(\varphi\circ\gamma)'(t_0)\|\geq
\delta\Lip(\varphi,x_0)\Lip(\gamma,t_0)$, then by equation
\eqref{eq:psione} there exists a
$v\in\mathbb S^{n-1}$ such that
\begin{align*}\|(\varphi\circ\gamma)'(t_0)\| &=
  |v\cdot(\varphi\circ\gamma)'(t_0)\|\\
&= |D\psi^{-1}(v)\cdot(\psi\circ\gamma)'(t_0)|\\
&\leq \|(\psi\circ\gamma)'(t_0)\|\Lip(\varphi,x_0)/\eta\Lip(\psi,x_0).
\end{align*}
Therefore
\[\|(\psi\circ\gamma)'(t_0)\|\geq
\delta\eta\Lip(\psi,x_0)\Lip(\gamma,t_0).\]

Let $S=\cup_i S_i\cup
N$ be a countable Borel decomposition where $\mu(N)=0$ and such that,
for each $i\in\mathbb N$, there exists a $w_i\in\mathbb S^{n-1}$ with
\[\left\|\frac{D\psi(x_0)(\widetilde w)}{\|D\psi(x_0)(\widetilde
  w)\|}-w_i\right\|<\epsilon.\]
Then if $\gamma\in\Gamma(S_i)$ is in the
$\varphi$-direction of $C(w_i,\theta-\epsilon)$ with $\varphi$-speed $\delta$,
for almost every $t_0\in\dom\gamma$,
\[(\varphi\circ\gamma)'(t_0)\in C(w_i,\theta-\epsilon)\subset
C\left(\frac{D\psi(x_0)(\widetilde w)}{\|D\psi(x_0)(\widetilde w)\|},
\theta\right)\]
and
\[\|(\psi\circ\gamma)'(t_0)\|\geq \eta\delta
\Lip(\psi,x_0)\Lip(\gamma,t_0).\] Therefore, by the above estimates,
$\gamma$ is in the $\psi$-direction of $C(\widetilde
w,1-(1-\theta)\lambda\eta/\Lip\varphi)$ with $\psi$-speed
$\delta\eta$.
\end{proof}

As a consequence, we obtain the following Theorem.

\begin{theorem}\label{thm:atildenull}
Any $\widetilde A$ subset of a Lipschitz differentiability space has
measure zero.
\end{theorem}

\begin{proof}
For $\lambda>0$ let $(U,\varphi)$ be a $\lambda$-structured chart and
$U=\cup_i U_i$ be a countable Borel decomposition such that, for each
$i\in\mathbb N$, there exists independent cones $C_1,\ldots,C_n$ and a
$\delta>0$ such that $\mu\llcorner U_i$ has Alberti representations in
the $\varphi$-direction of each $C_i$ with $\varphi$-speed $\delta$.
Further, let $0<\epsilon<\theta<1$ such that any cone in $\mathbb R^n$
of width $\theta-\epsilon$ completely contains one of the $C_i$.

We work with a fixed $U_i$.  Suppose that for some Lipschitz
$\psi\colon X\to\mathbb R^m$ and $\widetilde w\in\mathbb S^{m-1}$,
$S\subset U_i$ satisfies $\mathcal H^1(\gamma\cap S)=0$ for every
$\gamma\in\Gamma(X)$ in the $\psi$-direction of $C(\widetilde
w,1-(1-\theta)\lambda\eta/ \Lip\varphi)$ with $\psi$-speed
$\delta\eta$.  Then by the previous Lemma there exists a countable
Borel decomposition $S=\cup_i S_i\cup N$ and for each $i\in\mathbb N$
cones $C_i$ of width $\theta-\epsilon$ such that $\mu(N)=0$ and $\mathcal
H^1(\gamma\cap S_i)=0$ for any $\gamma\in\Gamma(X)$ in the
$\varphi$-direction of $C_i$ with speed $\delta$.  However, one of the
Alberti representations of $\mu\llcorner U_i$ is in the
$\varphi$-direction of this cone with $\varphi$-speed $\delta$ and so
$\mu(S_i)=0$ and hence $\mu(S)=0$.

Finally, let $S'\in \widetilde A$ and $\eta>0$ such that
$S'\in\widetilde A(\delta',\theta',\eta)$ for any
$0<\delta',\theta'<1$.  In particular $S\in\widetilde
A(\eta\delta,1-(1-\theta)\lambda\eta/\Lip\varphi,\eta)$.  Therefore,
there exists a countable decomposition $S'=\cup_j S_j$ and for each
$j\in\mathbb N$ a Lipschitz $\psi_j\colon X\to\mathbb R^{n_j}$ such that
each $S_j$ has the form of $S$ above.  Therefore, each $S_j$ has
measure zero.  In particular, any $\widetilde A$ subset of $U_i$ and
hence of $(X,d,\mu)$ must be $\mu$-null.
\end{proof}

Finally, as a consequence of the existence of the above Alberti
representations, we give a partial, positive answer to
\cite{cheeger-diff}, Conjecture 4.63 regarding the image of a chart under
the chart map.

\begin{corollary}\label{cor:chartleb}
For $n=1$ or 2 let $(U,\varphi)$ be an $n$-dimensional chart in a
Lipschitz differentiability space $(X,d,\mu)$.  Then the pushforward
$\varphi_* (\mu\llcorner U)$ is absolutely continuous with respect to
$n$-dimensional Lebesgue measure.

In particular, if $\mu(U)>0$ then $\mathcal L^n(\varphi(U))>0$.
\end{corollary}

\begin{proof}
Suppose that a measure $\nu$ has an Alberti representation in the
$\varphi$-direction of a cone $C\subset \mathbb R^n$.  Then an easy
application of Corollary \ref{cor:onerep2} shows that
$\varphi_*\nu$ also has an Alberti representation in the direction of
$C$.  By Theorem \ref{thm:spanrep} there exists a countable Borel
decomposition $U = \cup_m U_m$ such that each $\mu\llcorner U_m$ has
$n$-independent Alberti representations, so that each
$\varphi_*(\mu\llcorner U_m)$ has $n$ independent Alberti
representations.

Any measure on $\mathbb R$ with an Alberti representation is
absolutely continuous with respect to Lebesgue measure.  Further,
results from \cite{acp-structurenullsets} show that any measure on $\mathbb R^2$ with
two independent Alberti representations must also be absolutely
continuous with respect to Lebesgue measure.  Therefore, in either
case, $\varphi_*(\mu\llcorner U)$ is absolutely continuous with
respect to Lebesgue measure.
\end{proof}

This improves the known cases proved by Keith ($n=1$, see \cite{keith}, page
282) and Gong ($n=2$, see \cite{gong-posmeas}, Theorem 1.4) who prove
Corollary \ref{cor:chartleb} for Lipschitz differentiability spaces
that satisfy the Poincar\'e inequality and possess a doubling measure
(see Definition \ref{defn:poincare}).

\begin{remark}
Using a recent announcement of Cs\"ornyei and Jones (see
\url{www.math.sunysb.edu/Videos/dfest/PDFs/38-Jones.pdf}, pages 15-23)
one may show that, for any $n\in\mathbb N$, any measure on $\mathbb
R^n$ with $n$ independent Alberti representations is absolutely
continuous with respect to Lebesgue measure.  Therefore, the previous
argument also proves the statements of Corollary \ref{cor:chartleb} for
any $n\in\mathbb N$.

One may also use this announcement to generalise the techniques of
Keith and Gong to prove Corollary \ref{cor:chartleb} for all $n\in
\mathbb N$ for Lipschitz differentiability spaces that satisfy the
Poincar\'e inequality and possess a doubling measure.
\end{remark}


\section{A characterisation via Alberti representations}\label{sec:albertichar}
We now show how the Alberti representations found in the previous
section characterise the differentiability properties of a metric
measure space. 

We first introduce the notion of a \emph{universal} set of Alberti
representations for when a set of representations describe all of the
Lipschitz functions on a metric measure space.

\begin{definition}\label{def:universal}
Let $(U,\varphi)$ be a chart in a metric measure space $(X,d,\mu)$ and
$\mathcal A_1,\ldots,\mathcal A_n$ be a collection of
$\varphi$-independent Alberti representations of $\mu\llcorner U$,
each with strictly positive $\varphi$-speed.  For $\rho>0$ we say that
this collection of Alberti representations is \emph{$\rho$-universal}
(or just \emph{universal} if such a $\rho$ exists) if, for any
Lipschitz $f\colon X\to\mathbb R$, there exists a finite Borel
decomposition $X=X_1\cup\ldots\cup X_n$ such that the Alberti
representation of $\mu\llcorner X_i$ induced by $\mathcal A_i$ has
$f$-speed $\rho$.
\end{definition}

\begin{remark}
A universal collection of Alberti representations is a
much stronger concept than a \emph{maximal} collection of
representations (i.e. a collection for which there are no other
independent representations).  For example, any purely unrectifiable
metric measure space has a maximal (empty) collection of Alberti
representations.  However, by Corollary \ref{cor:zerodim}, this collection is
not universal as there exists a Lipschitz function with positive
pointwise Lipschitz constant on a set of positive measure.  (For a
slightly less trivial example, one may consider the Cartesian product
of $\mathbb R$ and a purely unrectifiable space.)
\end{remark}

We will see that a universal collection of Alberti
representations is precisely the required concept so that the gradient
(see Definition \ref{def:gradient}) $\nabla f$ of a Lipschitz
function $f$ forms a derivative.  We prove this in a very natural way:
observe that the derivative of $f-\nabla f\cdot\varphi$ along the Lipschitz
curves defining the gradient is zero.  Further, if the Alberti
representations are sufficiently refined, then the derivative along
almost every curve in the representation is also sufficiently small.
Therefore, after a suitable limit of such gradients
obtained from a sequence of Alberti representations, each refining the
previous, we obtain a derivative.

We begin with a simple result that bounds the magnitude of the
gradient using properties of Alberti representations.  For this we must
quantitatively describe an independent collection of Alberti
representations.

\begin{definition}\label{def:separated}
For $\xi>0$ we say $v_1,\ldots,v_m\in\mathbb R^n$ are
\emph{$\xi$-separated} if, for any $\lambda\in \mathbb
R^m\setminus\{0\}$,
\[\left\|\sum_{i=1}^m \lambda_i v_i\right\| > \xi \max_{1\leq i\leq
  m}\|\lambda_i v_i\|\]
and that closed cones $C_1,\ldots, C_m$ are
$\xi$-separated if any choice of $v_i\in C_i\setminus\{0\}$ are
$\xi$-separated.  Further, we say that Alberti representations
$\mathcal A_1,\ldots,\mathcal A_m$ are \emph{$\xi$-separated} if there
exists $\xi$ separated cones $C_1,\ldots,C_m$ such that each $\mathcal
A_i$ in the $\varphi$-direction of $C_i$.
\end{definition}

Observe that cones $C_1,\ldots,C_m$ are $\xi$-separated if and only
if, for every $1\leq i\leq m$, the distance of $C_i\cap \mathbb
S^{n-1}$ from those $v$ in the symmetric convex hull of the $C_j$ (for
$j\neq i$) with $\|v\|\leq 1$ is strictly greater than $\xi$.  In
particular, any independent cones are $\xi$-separated for some
$\xi>0$.

Suppose that a metric measure space $(X,d,\mu)$ has Alberti
representations in the $\varphi$-direction of $\xi$-separated cones
$C(w_1,\theta),\ldots,C(w_m,\theta)$.  Then there exists an
$\epsilon>0$ such that
$C(w_1,\theta+\epsilon),\ldots,C(w_m,\theta+\epsilon)$ are also
$\xi$-separated and a finite cover of each $C(w_i,\theta)$ by cones
$C_1^i,\ldots, C_{N_i}^i$ of width $\epsilon$ that are contained
within $C(w_i,\theta+\epsilon)$.  By applying Corollary
\ref{cor:refine} using the $C_j^i$ we obtain \emph{arbitrary}
refinements of these Alberti representations that are also
$\xi$-separated.  Moreover, if for some Lipschitz $\psi\colon
X\to\mathbb R^n$ and $\delta_1,\ldots,\delta_n>0$, the original
Alberti representations have $\psi_i$-speed strictly greater than
$\delta_i$ for each $1\leq i \leq N$, then so do the refinements.

For this section we fix the following notation.

\begin{notation}\label{not:albchar}
We fix a metric measure space $(X,d,\mu)$ and for
$\rho,\delta,\lambda,\xi>0$ let $(U,\varphi)$ be an $n$-dimensional
$\lambda$-structured chart in $(X,d,\mu)$ such that $\mu\llcorner U$
has a $\rho$-universal collection of $n$, $\xi$-separated Alberti
representations with $\varphi$-speed strictly greater than $\delta$.
\end{notation}

Using the notion of separated Alberti representations, we may bound
the magnitude of a gradient.

\begin{lemma}\label{lem:gradientbound}
Let $f\colon X\to\mathbb R$ be Lipschitz, $x_0\in U$ and
$\gamma_1,\ldots,\gamma_n \in\Gamma(X)$ such that the
$(\varphi\circ\gamma_i)'(0)$ are $\xi$-separated and, for each $1\leq
i \leq n$, $\gamma_i^{-1}(x_0)=0$ is a density point of $\dom\gamma_i$
and
\[(\varphi\circ\gamma_i)'(0)\geq
\delta\Lip(\varphi,x_0)\Lip(\gamma_i,0).\]
Then the gradient $\nabla f(x_0)$ of $f$ at $x_0$ with
respect to $\varphi$ and $\gamma_1,\ldots,\gamma_n$ satisfies
\[\|\nabla f(x_0)\|\leq n\Lip(f,x_0)/\xi\delta\lambda.\]
\end{lemma}

\begin{proof}
For any $\xi$-separated $v_1,\ldots,v_n$ and
$G\in\mathbb R^n$, there exists a $1\leq i \leq n$ such that
\[\xi\|G\|/n\leq |G\cdot v_i|.\]
Therefore,
\begin{align*}
\Lip(f,x_0)\Lip(\gamma_i,0) &\geq |(f\circ\gamma_i)'(0)|\\
&= |\nabla f(x_0)\cdot(\varphi\circ\gamma_i)'(0)|\\
&\geq \xi \|\nabla f(x_0)\|\|(\varphi\circ\gamma_i)'(0)\|/n\\
&\geq \xi\delta \|\nabla f(x_0)\|\Lip(\varphi,x_0)\Lip(\gamma_i,0)/n.
\end{align*}
That is,
\[\|\nabla f(x_0)\|\leq \Lip f n/\xi\delta\Lip(\varphi,x_0)\]
as required.
\end{proof}

Next we refine a universal collection of Alberti representations
whilst maintaining their speed with respect to a finite collection of
Lipschitz functions.

\begin{lemma}\label{lem:univrefine}
Let $\mathcal F$ be a finite collection of real valued Lipschitz
functions defined on $X$, $f\colon X\to\mathbb R$ Lipschitz and
$\epsilon>0$.  Then there exists a finite Borel decomposition
$X=X_1\cup\ldots\cup X_N$ and for each $1\leq i \leq N$, Alberti
representations $\mathcal A'_1,\ldots,\mathcal A'_n$ of $\mu\llcorner
X_i$ such that:
\begin{enumerate}
\item \label{item:vecrefinesep} The $\mathcal A'_k$ are $\xi$-separated and have $\varphi$-speed
  strictly greater than $\delta$.
\item \label{item:vecrefinespeed} For every $g\in\mathcal F$ there exists an $\mathcal A'_k$ with
  $g$-speed $\rho$.
\item \label{item:vecrefinewidth} If we write $\Phi\colon X\to\mathbb R^{n+1}$ to be the function
  obtained by appending $f$ to $\varphi$, then each $\mathcal A'_k$
  is in the $\Phi$-direction of a cone of width $\epsilon$.
\end{enumerate}
\end{lemma}

\begin{proof}
The collection $\mathcal A_1,\ldots,\mathcal A_n$ is
$\rho$-universal and so there exists a finite Borel
decomposition $X=Z_1\cup\ldots\cup Z_L$, such that, for any
$g\in\mathcal F$ and $1\leq i\leq L$, there exists a $1\leq k\leq n$ such
that $\mathcal A_k\llcorner Z_i$ has $g$-speed $\rho$.

We fix $1\leq i \leq L$ and for each $1\leq k\leq n$ let $\mathcal
F(k)$ be the set of $g\in\mathcal F$ such that $\mathcal A_k\llcorner
Z_i$ has $g$-speed $\rho$.  Then by Corollary \ref{cor:refine} we may
refine each $\mathcal A_k\llcorner Z_i$ to obtain a finite
decomposition $Z_i=Y_k^1\cup\ldots\cup Y_k^{M_k}$ and Alberti
representations $\mathcal A_k^1,\ldots,\mathcal A_k^{M_k}$ of
$\mu\llcorner Y_k^1,\ldots,\mu\llcorner Y_k^{M_k}$ respectively, in
the $\Phi$-direction of cones of width $\epsilon$ and with
$\varphi$-speed strictly greater than $\delta$, and such that each
$\mathcal A_k^i$ has $g$-speed $\rho$, for each $g\in \mathcal F(k)$.
Moreover, we may make these refinements such that the representations
are $\xi$-separated, with respect to $\varphi$.  Therefore, for these
representations, both \eqref{item:vecrefinesep} and
\eqref{item:vecrefinewidth} are satisfied.

By taking the intersection of $Y_k^m$ for $1\leq k\leq n$, we obtain a
finite Borel decomposition $Z_i=X_1^i\cup\ldots\cup X_{N_i}^i$ such
that each $X_j^i$ is a subset of some $Y_k^m$, for each $1\leq k\leq
n$.  Therefore, for each $1\leq m\leq N_i$ and any $g\in\mathcal F$,
there exists a $1\leq k \leq n$ such that $\mathcal A_k\llcorner
X^i_m$ has $g$-speed $\rho$.  In particular, \eqref{item:vecrefinespeed}
is satisfied and so $X=\cup_{i,m} X_m^i$ is a decomposition of the
required form.
\end{proof}

Finally, we use this refinement to show that the gradient obtained
from such Alberti representations is an $\epsilon$-derivative at
almost every point.

\begin{lemma}\label{lem:univgivesapproxder}
Let $f\colon X\to\mathbb R$ be Lipschitz.  Then there exists an
$M\in\mathbb R$ such that, for every $\epsilon>0$ and almost every
$x_0\in U$, there exists a $\nabla f(x_0)\in\mathbb R^n$ with
\[\Lip(f-\nabla f(x_0)\cdot\varphi,x_0)\leq M\sqrt{\epsilon}.\]
\end{lemma}

\begin{proof}
Fix $\epsilon>0$ and let $\mathcal Q$ be a finite $\epsilon$-net of
$\overline B(0,\Lip f n/\xi\delta\lambda)\subset\mathbb R^n$ and
\[\mathcal F = \{f-D\cdot\varphi:D\in\mathcal Q\}.\]
Then by the previous Lemma there exists a countable decomposition
$U=\cup_i U_i$ such that each $\mu\llcorner U_i$ has Alberti
representations $\mathcal A'_1,\ldots,\mathcal A'_n$ satisfying the
properties of the previous Lemma for $\epsilon$, $f$ and $\mathcal F$.
We fix such a $U_i$ and define $\Phi\colon X\to\mathbb R^{n+1}$ to be
the Lipschitz function obtained by appending $f$ to $\varphi$ and let
$C_1,\ldots,C_n\subset\mathbb R^{n+1}$ have width $\epsilon$ such that
each $\mathcal A_k$ is in the $\Phi$-direction of $C_k$.

By combining Proposition \ref{prop:curvefullmeas} and Lemma
\ref{lem:gradientbound}, for almost every $x_0\in U_i$ there exist
$\gamma_1,\ldots,\gamma_n\in\Gamma(X)$ in the $\Phi$-direction of
$C_1,\ldots,C_n$ respectively such that the gradient $\nabla f(x_0)$ of $f$
at $x_0$ with respect to $\varphi$ and $\gamma_1,\ldots,\gamma_n$
satisfies
\[\|\nabla f(x_0)\|\leq n\Lip(f,x_0)/\xi\delta\lambda.\]
In particular, there exists a
$D\in\mathcal Q$ with $\|D-\nabla f(x_0)\|<\epsilon$.  We write
$\mathbf D$ and $\mathbf{\nabla f}(x_0)\in\mathbb R^{n+1}$ for the
vectors obtained by appending $-1$ to $D$ and $\nabla f(x_0)$
respectively.

Suppose that $v,v'\in\mathbb R^{n+1}$ belong to a cone of width
$\epsilon$.  Then
\[\left\|\frac{v}{\|v\|}-\frac{v'}{\|v'\|}\right\|\leq
  \sqrt{\epsilon(2-\epsilon)}\]
and so, for any $a\in\mathbb R^{n+1}$,
\[|a\cdot v|\leq (|a\cdot
v'|/\|v'\|+\sqrt{\epsilon(2-\epsilon)})\|v\|.
\]
For almost every $x_0\in U$, $\nabla f(x_0)$ is the gradient of $f$
with respect to $\varphi$ and $\gamma_1,\ldots,\gamma_n$, and so
$\mathbf{\nabla f}(x_0)\cdot (\Phi\circ\gamma_i)'(0)=0$ for each
$1\leq i \leq n$.  Therefore, if $\widetilde\gamma\in\Gamma(X)$ is in
the $\Phi$-direction of some $C_i$ with $\widetilde\gamma(t_0)=x_0$
such that $(\Phi\circ\widetilde\gamma)'(t_0)$ exists, then
\[
|\mathbf{\nabla f}(x_0)\cdot (\Phi\circ\widetilde\gamma)'(t_0)| \leq
\sqrt{\epsilon(2-\epsilon)}) \|\mathbf{\nabla f}(x_0)\|\|
(\Phi\circ\widetilde\gamma)'(t_0)\|.
\]
In particular,
\[|\mathbf{D}\cdot (\Phi\circ\widetilde\gamma)'(t_0)| \leq
\sqrt{\epsilon(2-\epsilon)} \|\mathbf{D}\|
\|(\Phi\circ\widetilde\gamma)'(t_0)\| +2\epsilon
(\Phi\circ\widetilde\gamma)'(t_0).\]
However, for each $1\leq i \leq
n$, almost every $\widetilde\gamma\in\mathcal A_i$ is of this form.
Further, since $D\cdot\varphi-f\in\mathcal F$, one of the $\mathcal
A'_k$ has $D\cdot\varphi-f$ speed $\rho$.  Therefore,
\[\rho \Lip(D\cdot\varphi-f,x_0)\leq
\sqrt{\epsilon(2-\epsilon)}\|\mathbf
D\|\Lip\Phi+2\epsilon\Lip\Phi\]
for almost every $x_0\in U$.  Finally, since
\[\|\mathbf D\|\leq 1+\|\nabla f(x_0)\|\leq 1+n\Lip f/\xi\delta\lambda,\]
we have
\[\Lip(f-\nabla f(x_0)\cdot \varphi,x_0)\leq
(\Lip\varphi+\Lip f)(2\epsilon+\sqrt{\epsilon(2-\epsilon)}(1+n\Lip
f/\xi\delta\lambda))\]
almost everywhere in $U$, as required.
\end{proof}

Using this construction we may give our first characterisation of Lipschitz
differentiability spaces.  We now work without the fixed quantities
given in Notation \ref{not:albchar}.

\begin{theorem}\label{thm:albchar}
A metric measure space $(X,d,\mu)$ is a Lipschitz differentiability
space if and only if there exists a countable Borel decomposition
$X=\cup_i U_i$ such that each $\mu\llcorner U_i$ has a finite universal
collection of Alberti representations.

In this case, for each $i\in\mathbb N$ let $\varphi_i\colon X\to\mathbb
R^{n_i}$ be Lipschitz such that the Alberti representations of
$\mu\llcorner U_i$ are $\varphi_i$-independent.  Then each
$(U_i,\varphi_i)$ is a chart and the derivative of a Lipschitz
function $f\colon X\to\mathbb R$ is given by any gradient of $f$ with
respect to $\varphi_i$ at almost every point.
\end{theorem}

\begin{proof}
We first show that the condition is necessary for a metric measure
space to be a Lipschitz differentiability space.  For any $i\in\mathbb
N$ let $\varphi_i\colon X\to\mathbb R^{n_i}$ be Lipschitz such that the
Alberti representations of $\mu\llcorner U_i$ described in the
hypotheses are $\varphi_i$-independent. Then by applying Proposition
\ref{prop:curvefullmeas}, for almost every $x_0\in U_i$ there exist
$\gamma_1,\ldots,\gamma_{n_i}\in\Gamma(X)$ such that the
$(\varphi\circ\gamma_i)'(0)$ form a basis of $\mathbb R^{n_i}$ and, for
each $1\leq i \leq n_i$, $\gamma_i^{-1}(x_0)=0$ is a density point of
$\dom\gamma_i$.  Therefore, for almost every $x_0\in U_i$, there exists
a $\lambda(x_0)>0$ such that, for any $v\in\mathbb S^{n_i-1}$,
\[\limsup_{X\ni x\to x_0} \frac{|(\varphi(x)-\varphi(x_0))\cdot
  v|}{d(x,x_0)}\geq \lambda(x_0).\] In particular, by Lemma
\ref{lem:discuniq}, any possible derivative of a function at $x_0$ is
unique.

Now let $\delta,\xi>0$ such that the Alberti representations of
$\mu\llcorner U_i$ are $\xi$-separated with respect to $\varphi_i$ and
have $\varphi_i$-speed strictly greater than $\delta$.  By taking a
further decomposition (and allowing for a possible $\mu$-null subset
of $X$), we may suppose that $(U_i,\varphi_i)$ is
$\lambda$-structured, for some $\lambda>0$. Then we are in the
situation described by Notation \ref{not:albchar} and so, for any
Lipschitz $f\colon X\to\mathbb R$, by Lemma
\ref{lem:univgivesapproxder} there exists an $M>0$ such that, for any
$\epsilon>0$ and almost every $x_0\in U$, there exists a $\nabla
f(x_0)\in\mathbb R^{n_i}$ with $\|\nabla f(x_0)\|\leq n_i\Lip
f/\xi\delta\lambda$ and
\[\Lip(f-\nabla f(x_0)\cdot \varphi_i,x_0)\leq M\sqrt{\epsilon}.\]
Applying this with some sequence $\epsilon_m\to 0$, for almost every
$x_0\in U_i$ we obtain a bounded sequence of such $\nabla f(x_0)$ and
so, after passing to a convergent subsequence, we obtain a
$Df(x_0)\in\mathbb R^{n_i}$ that is a derivative of $f$ at $x_0$ with
respect to $(U_i,\varphi_i)$.  Further, by Corollary
\ref{cor:tangentcurves}, this derivative is given by any gradient of $f$
with respect to $\varphi_i$.

Conversely, let $(U,\varphi)$ be an $n$-dimensional,
$\lambda$-structured chart in a Lipschitz differentiability space.
Then by Theorem \ref{thm:spanrep} and Corollary \ref{cor:refine},
there exists a countable Borel decomposition $U=\cup_i U_i$ and for
every $i\in\mathbb N$ a $\delta>0$ such that, for each $i\in\mathbb
N$, $\mu\llcorner U_i$ has $n$ $\varphi$-independent Alberti
representations with $\varphi$-speed $\delta$.

Fix such a $U_i$ and suppose such representations $\mathcal
A_1,\ldots,\mathcal A_n$ of $\mu\llcorner U_i$ are in the
$\varphi$-direction of $\xi$-separated cones $C_1,\ldots,C_n$.  Then
for any Lipschitz $f\colon X\to\mathbb R$ and almost every $x_0\in
U_i$, $Df(x_0)$ exists.  Moreover, there exists a $1\leq j\leq n$ such
that, for any $v\in C_j$,
\[\xi\|Df(x_0)\|\|v\|/n \leq |Df(x_0)\cdot v|.\]
For $1\leq j\leq n$ let $V_j$ be the set of $x_0$ that satisfy this for
$C_j$.  Therefore, if $x_0\in V_j$ and $\gamma\in\Gamma(X)$ with
$\gamma(t_0)=x_0$, $\|(\varphi\circ\gamma)'(t_0)\|\geq\delta
\Lip(\varphi,x_0)\Lip(\gamma,t_0)$ and $(\varphi\circ\gamma)'(t_0)\in
C_j$ then
\begin{align*}
(f\circ\gamma)'(t_0) &= |Df(x_0)\cdot(\varphi\circ\gamma)'(t_0)|\\
&\geq \|Df(x_0)\|\|(\varphi\circ\gamma)'(t_0)\|\\
&\geq \delta\Lip(f,x_0)\Lip(\varphi,x_0)\Lip(\gamma,t_0)\\
&\geq \delta\lambda \Lip(f,x_0)\Lip(\gamma,t_0).
\end{align*}
That is, $\mathcal A_j\llcorner V_j$ has $f$-speed
$\delta\lambda$.  Therefore, the $\mathcal A_j$ are universal.
\end{proof}


\section{Characterisations via null sets}\label{sec:char}
We now present other characterisations that prescribe a class of sets
that determine if a metric measure space is a Lipschitz
differentiability space.  Such sets will be the singular sets found
earlier when constructing Alberti representations of a measure.
Therefore, in a metric measure space where such sets have measure
zero, there exist many Alberti representations.  Further, if these
representations are sufficiently different, for almost every
$x_0$, particular points of the Lipschitz curves obtained from the
representations form a separated subset of balls centred at $x_0$.
By combining this with a doubling condition on the metric
space, we obtain a bound on the total number of such different
Alberti representations.  In turn, this leads to the existence of a
derivative of a Lipschitz function at almost every point.

We begin by giving a method that constructs a chart structure in
a metric measure space.

\begin{lemma}\label{lem:boundgivesdiff}
Let $(X,d,\mu)$ be a metric measure space, $Y\subset X$ Borel and
$N\in\mathbb N$.  Then the following are equivalent:
\begin{itemize}\item For any Lipschitz $\psi\colon X\to\mathbb
  R^{N+1}$, for almost every $x_0\in Y$ there exists an
  $a(x_0)\in\mathbb S^N$ such that
\[\Lip(a(x_0)\cdot\psi,x_0)=0.\]
\item There exists a countable Borel decomposition $Y=\cup_i U_i$ and
  Lipschitz functions $\varphi_i\colon X\to\mathbb R^{n_i}$ such that
  each $(U_i,\varphi_i)$ is a chart of dimension at most $N$, with
  respect to which any Lipschitz $f\colon X\to\mathbb R$ is
  differentiable at almost every point of $U_i$.
\end{itemize}
\end{lemma}

\begin{proof}
First suppose that the second statement is true, $(U_i,\varphi_i)$ a
chart of the decomposition and $\psi\colon X\to\mathbb R^{N+1}$
Lipschitz.  Then for almost every $x_0\in U_i$ and each $1\leq j \leq N+1$,
the derivative of $\psi_j$ exists at $x_0$ and belongs to $\mathbb
R^{n_i}$, for $n_i<N+1$.  Therefore, if we write $D\psi$ for the
matrix whose columns are the $D\psi_i(x_0)$, there exists an
$a(x_0)\in\mathbb S^{N}$ such that $D\psi\cdot a(x_0)=0$.  Then by the
definition of a derivative, $\Lip(a(x_0)\cdot\psi,x_0)=0$.

Now suppose that the first statement holds and let $U\subset Y$ be a
Borel set of positive measure.  First observe that either any
Lipschitz function $\psi\colon X\to\mathbb R$ satisfies
$\Lip(\varphi,x_0)=0$ for almost every $x_0\in U$ in which case $U$
is a chart of dimension zero (as described in Corollary
\ref{cor:zerodim}) or there exists a Borel set $U'\subset U$ of
positive measure and a Lipschitz $\varphi\colon X\to\mathbb R$ such that
$\Lip(\varphi, x_0)>0$ for every $x_0\in U'$.

By our hypotheses on $Y$, there exists a maximal $n\in\mathbb N$ such
that there exists a Lipschitz $\varphi\colon X\to\mathbb R^n$ and a
$U'\subset U$ of positive measure with
$\Lip(v\cdot\varphi,x_0)>0$ for every $x_0\in U'$ and every
$v\in\mathbb S^{n-1}$.  Further, we have $n\leq N$.  Then for such an
$n$, $\varphi$ and $U'$, for any Lipschitz $f\colon X\to\mathbb R$ and
almost every $x_0\in U'$ there exists a $Df(x_0)\in\mathbb R^n$ such
that $\Lip(f-Df(x_0)\cdot\varphi,x_0)=0$.  Moreover, by Lemma
\ref{lem:discuniq}, the condition $\Lip(v\cdot\varphi,x_0)>0$ for
every $v\in\mathbb S^{n-1}$ is equivalent to the uniqueness of such a
$Df(x_0)$.

Therefore, within any $U$ of positive measure, there exists
a chart of positive measure with dimension less than $N$, with respect
to which any real valued Lipschitz function is differentiable almost
everywhere.  Since $\mu$ is finite we obtain a countable decomposition
of $U$ into charts with respect to which Lipschitz functions are
differentiable almost everywhere.
\end{proof}

We now introduce additional properties of metric measure spaces that
will allow us to satisfy the hypotheses of Lemma
\ref{lem:boundgivesdiff}.

\begin{definition}\label{def:ptwisedoubling}
We say that a metric measure space $(X,d,\mu)$ is \emph{pointwise
  doubling} if
\[\limsup_{r\to 0} \frac{\mu(B(x_0,r))}{\mu(B(x_0,r/2))}<\infty\]
for almost every $x_0\in X$.

Further, for $0<\delta<1$ and $\eta\geq 1$, we define
$M(\delta,\eta)=\eta^{2-\log_2 \delta/5}$ and $D_\eta$ to be the set
of Borel subsets $Y$ of $X$ such that, for each $x_0\in Y$,
$0<\delta<1$ and $r>0$, there exist $x_1,\ldots,x_M\in Y$ with
\[B(x_0,r)\cap Y\subset \bigcup_{i=1}^M B(x_i,\delta r),\]
for some $M\in\mathbb N$ less than $M(\delta,\eta)$.
\end{definition}

\begin{lemma}\label{lem:ptwisedoubling}
For any pointwise doubling metric measure space $(X,d,\mu)$,
there exists a Borel decomposition
\[X=N\cup \bigcup_{m=1}^\infty Y_m\]
where $\mu(N)=0$ and for each $m\in\mathbb N$ there exists an
$\eta\geq 1$ such that $Y_m\in D_\eta$.
\end{lemma}

\begin{proof}
For $\eta\geq 1$ and $R>0$ let
\[X_{R,\eta}:=\{ x_0\in X : \eta\mu(B(x_0,r/2)) >
\mu(B(x_0,r))>0,\ \forall\ 0<r<R\}.
\]
For any $x_0\in X$ and $r>0$,
\[\frac{\mu(B(x_0,r))}{\mu(B(x_0,r/2))} = \lim_{r_m \nearrow r}
\frac{\mu(B(x_0,r_m))}{\mu(B(x_0,r_m/2))}
\]
and so, in the definition of $X_{R,\eta}$, it is equivalent to satisfy
the condition for rational $0<r<R$.  Further, for a fixed $r>0$,
$x_0\mapsto \mu(B(x_0,r))$ is lower semicontinuous.  Therefore, each
$X_{R,\eta}$ is a Borel set.

Now fix $\eta\geq 1$, $R>0$ and $x_0\in X_{R,\eta}$.  Then for any
$m\in\mathbb N$ and $0<r<R$,
\[\frac{\mu(B(x_0,r))}{\mu(B(x_0,2^{-m}r))} \leq \eta^m\]
and so, for any $0<\delta<1$,
\[\frac{\mu(B(x_0,r))}{\mu(B(x_0,\delta r))}\leq \eta^M,\]
for $M$ the least integer greater than $-\log_2\delta$.

Now let $0<r<R/2$.  Then
\[B(x_0,r) \cap X_{R,\eta}\subset \bigcup \{B(x,\delta
r/5):x \in B(x_0,r) \cap X_{R,\eta}\}\]
and so there exists $F\subset B(x_0,r)\cap X_{R,\eta}$ such that
\[B(x_0,r) \cap X_{R,\eta}\subset \bigcup \{B(x,\delta
r):x \in F\}\]
and such that the $B(x_0,\delta r/5)$ are disjoint.  Firstly, since
$\mu(B(x_0,2r))$ is finite and for $x\in F$ the $B(x,\delta r/5)$ are disjoint
subsets of $B(x_0,2r)$ with positive measure, $F$ is countable.  Further,
\[\eta^M \mu(B(x_0,r)) \geq \eta^M \sum_{x\in F} \mu(B(x,\delta
  r/5)) \geq \sum_{x\in F} \mu(B(x, 2r)) \geq \sum_{x\in F}
  \mu(B(x_0,r)),
\]
for $M$ the least integer greater than $2-\log_2 \delta/5$.  In
particular, the cardinality of $F$ is less than
$M(\delta,\eta)$.  Therefore, for any $y\in X$, $X_{R,\eta} \cap
B(y,R/4)$ belongs to $D_\eta$.

Finally, for $R_i\to 0$ and $\eta_i\to\infty$, the $X_{R_i,\eta_i}$
cover almost all of $X$ and so there exists a decomposition of $X$ of
the required form.
\end{proof}

Metric measure spaces in which all porous sets (see Definition
\ref{def:porous}) have measure zero are pointwise doubling (for
example, see \cite{porosityandmeasures}, Theorem 3.6).  Therefore, we
obtain the following consequence of Corollary \ref{cor:porousnull}.

\begin{corollary}[\cite{porousdiff}, Corollary 2.6]\label{cor:doubling}
Any Lipschitz differentiability space is pointwise doubling.
\end{corollary}

In addition, we also require concepts related to an $\widetilde
A$ set (see Definition \ref{def:atilde}).

\begin{definition}\label{def:btilde}
Let $(X,d)$ be a metric space and $\delta>0$.  We define $\widetilde
B_\delta$ to be the set of $S\subset X$ for
which there exist a countable Borel decomposition $S=\cup_i S_i$ and
Lipschitz functions $f_i\colon X\to\mathbb R$ such that, for every
$i\in\mathbb N$,
\[S_i\subset \{x_0\in X: \Lip(f_i,x_0)>0\}\]
and $\mathcal H^1(\gamma\cap S_i)=0$ for every $\gamma\in\Gamma(X)$
with $f_i$-speed $\delta$.  Further, we define $\widetilde B$ to be
the set of $S\subset X$ that belong to $\widetilde B_\delta$ for every
$\delta>0$.
\end{definition}

\begin{remark}
Note that for any $0<\delta\leq \delta '$ we have $\widetilde
B_\delta\subset \widetilde B_{\delta'}$.
\end{remark}

Observe that $\widetilde B_\delta\subset\widetilde
A(\delta,\theta,\lambda)$ for any $0<\delta,\theta,\lambda<1$ and so
$\widetilde B\subset \widetilde A$.  In particular, any $\widetilde B$
subset of a Lipschitz differentiability space has measure zero.  Also,
by Corollary \ref{cor:onerep2}, metric measure spaces in which
$\widetilde B_\delta$ sets are $\mu$-null have many Alberti
representations with speed $\delta$.  We now show how the curves
obtained from such Alberti representations interact with the doubling
properties of a metric measure space.  This is done by creating a
separated set from points belonging to curves obtained from such
representations.

\begin{lemma}\label{lem:albertibound}
Let $(X,d,\mu)$ be a metric measure space, $\varphi\colon X\to\mathbb
R^n$ Lipschitz, $Y\subset X$ Borel and $\delta>0$.  Suppose that
$\mathcal A_1,\ldots,\mathcal A_n$ are Alberti representations of
$\mu\llcorner Y$ and $w_1,\ldots,w_n\in\mathbb S^{n-1}$ such that, for
every $1\leq i\neq j\leq n$, either
\[\frac{|w_i\cdot (\varphi\circ\gamma_i)'(t_i)|}{\Lip(\gamma_i, t_i)}
> \frac{|w_i\cdot (\varphi\circ\gamma_j)'(t_j)|}{\Lip(\gamma_j, t_j)}
+ \delta \Lip (w_i\cdot \varphi, \gamma_i(t_i))>0,\]
or the corresponding inequalities with $i$ and $j$ exchanged.  Then,
for almost every $x_0\in X$ and any $r>0$, there exist
$x_1,\ldots,x_n\in B(x_0,r)$ such that
\[d(x_i,x_j)\geq \delta r- \phi(r)\]
for each $1\leq i\neq j\leq n$, where $\phi(r)/r\to 0$ as $r\to 0$.

In particular, if $Y\subset W\in D_\eta$, for some $\eta \geq 1$, then
$n\leq M(\delta/2,\eta)$.
\end{lemma}

\begin{proof}
Let $\mathcal Q$ be a countable dense subset of $\mathbb S^{n-1}$.
Then by standard techniques, there exists a Borel decomposition
\[Y=N\cup \bigcup_{m\in\mathbb N} Y_m\]
where $\mu(N)=0$ and for each $m\in\mathbb N$ and $q\in\mathcal Q$,
$\Lip (q\cdot\varphi)$ is continuous on each $Y_m$.  Therefore, for any
$\epsilon,\rho>0$, $m\in\mathbb N$ and $q\in\mathcal Q$, there exists
a $Z\subset Y_m$ with $\mu(Y_m\setminus Z)<\rho$ and an $R>0$ such that
\begin{equation}\label{eq:albbounddecomp}
|q\cdot(\varphi(x)-\varphi(y))|\leq (\Lip (q\cdot\varphi,
x_0)+\epsilon)d(x,y)
\end{equation}
for each $0<r<R$, $x_0\in Z$, $y\in Z\cap B(x_0,r)$ and $x\in B(y,r)$.
Therefore, by decomposing each $Y_m$ further, we may suppose that for
every $\epsilon>0$, $q\in\mathcal Q$, $m\in\mathbb N$ there exists an
$R>0$ such that \eqref{eq:albbounddecomp} holds for each $0<r<R$,
$x_0\in Y_m$, $y\in Z\cap B(x_0,r)$ and $x\in B(y,r)$.

We now work with a fixed $Y_m$.  By Proposition
\ref{prop:curvefullmeas}, for almost every $x_0\in Y_m$ there exists
$\gamma_1,\ldots,\gamma_n\in\Gamma(X)$ such that, for each $1\leq
i\leq n$, $\gamma_i^{-1}(x_0)=t_0$ is a density point of
$\gamma_i^{-1}(Y_m)$ and $\Lip(\gamma_i,t_0)=1$.  Further, there
exists $w_1,\ldots,w_n\in\mathcal Q$ such that, for each $1\leq i\neq
j\leq n$, either
\begin{equation}\label{eq:albboundeither}|w_i\cdot (\varphi\circ\gamma_i)'(t_0)| > |w_i\cdot
(\varphi\circ\gamma_j)'(t_0)| + \delta \Lip (w_i\cdot \varphi,
\gamma_i(t_0)) >0,
\end{equation}
or the corresponding inequalities with $i$ and $j$ exchanged.

Let $\epsilon>0$.  Then there exists an $R_1>0$ such that, for any
$0<r<R_1$, if $t$ belongs to each $\dom\gamma_i$ with $0<|t-t_0|<r$
then, after setting $x_k=\gamma_k(t)$,
\begin{align*}\Lip (w_i\cdot\varphi, x_0)d(x_i,x_j) &\geq |w_i\cdot(\varphi(x_i)-\varphi(x_j))|-\epsilon r\\
&\geq
  |w_i\cdot(\varphi(x_i)-\varphi(x_0))|-|w_i\cdot(\varphi(x_j)-\varphi(x_0)|-\epsilon
  r\end{align*}
for each $1\leq i, j\leq n$.  Further, there exists a
$0<R_2<R_1$ such that, for any $0<r<R_2$, $1\leq k\leq n$ and
$x_k=\gamma_k(t)$,
\[\|(\varphi\circ\gamma)'(t_0)(t_k-t_0)-\varphi(x_k)-\varphi(x_0)\|\leq \epsilon r.\]
In particular,
\[\Lip (w_i\cdot\varphi, x_0)d(x_i,x_j)\geq |w_i\cdot(\varphi\circ\gamma_i)'(t_0)(t-t_0)|-|w_i\cdot(\varphi\circ\gamma_j)'(t_0)(t-t_0)|-3\epsilon r.\]
Further, since $t_0$ is a density point of each $\dom\gamma_k$, there
exists a $0<R_3<R_2$ such that, for any $0<r<R_3$ and $1\leq k \leq n$,
there exists a $t$ in each $\dom\gamma_k$ such that $\gamma_k(t)\in
B(x_0,r)$ and $|t-t_0-r|\leq \epsilon r$.  Therefore, if
$x_k=\gamma_k(t)$ for each $1\leq k\leq n$,
\[\Lip (w_i\cdot\varphi, x_0)d(x_i,x_j)\geq r|w_i\cdot(\varphi\circ\gamma_i)'(t_0)|-r|w_i\cdot(\varphi\circ\gamma_j)'(t_0)|-5\epsilon r.\]
By exchanging $i$ and $j$ if necessary, we may apply
\eqref{eq:albboundeither} as stated so that
\[\Lip (w_i\cdot\varphi, x_0)d(x_i,x_j)\geq \delta r \Lip (w_i\cdot\varphi, x_0)-5\epsilon r.\]
Finally, also by our initial hypotheses, $\Lip
(w_i\cdot\varphi,\gamma_i(t_0))>0$ and so we deduce that there exists
$\phi\colon \mathbb R\to \mathbb R$ such that $\phi(r)/r\to 0$ as
$r\to 0$ and
\[d(x_i,x_j)\geq \delta r - \phi(r).\]
The $Y_m$ cover almost all of $Y$ and so such $x_1,\ldots,x_n$ may be
found for almost every $x_0\in Y$, as required.

Now suppose that for some $\eta\geq 1$, $Y\subset W\in D_\eta$.  Then
for any $0<\epsilon<\delta$, if $R>0$ such that $|\phi(r)|<\epsilon r$
for each $0<r<R$, the $x_i$ found above all belong to $B(x_0,r)$ but
are separated by a distance of at least $(\delta-\epsilon)r$.
Therefore no two $x_k$ can belong to the same ball of radius
$(\delta-\epsilon)r/2$.  In particular $n\leq
N((\delta-\epsilon)/2,\eta)$ for all $0<\epsilon<\delta$ and so $n\leq
N(\delta/2,\eta)$.
\end{proof}

We use this construction to satisfy the hypotheses of Lemma
\ref{lem:boundgivesdiff}.

\begin{proposition}\label{prop:firstcharts}
Let $(X,d,\mu)$ be a metric measure space, $\delta>0$ and, for
$\eta\geq 1$, let $Y$ be a Borel subset of some $W\in
D_\eta$.  Suppose that, for any Lipschitz $f\colon X\to\mathbb R$,
there exists an Alberti representation of $\mu\llcorner Y$ with
$f$-speed $\delta$.  Then, for any $N>M(\delta/2,\eta)$ and Lipschitz
$\varphi\colon X\to\mathbb R^N$, for almost every $x_0\in Y$ there
exists an $a(x_0)\in\mathbb R^N$ such that
\[\Lip(a(x_0)\cdot\varphi,x_0)=0.\]
\end{proposition}

\begin{proof}
Suppose that $\varphi\colon X\to\mathbb R^N$ is Lipschitz and, for
some Borel $Z\subset Y$ of positive measure, $\Lip(v\cdot \varphi, x_0)>0$
for every $x_0\in Z$ and every $v\in\mathbb S^{N-1}$.  By taking a countable
decomposition of $Z$ if necessary, we may suppose that there exists a
$\lambda>0$ such that
\[\Lip(v\cdot\varphi,x_0)\geq \lambda \Lip(\varphi,x_0)\]
for every $v\in\mathbb S^{N-1}$ and $x_0\in Z$.  We will prove the
Proposition by proving $N\leq M(\delta,\eta)$, using the assumption
that there exists an Alberti representation of $\mu\llcorner Z$ with
$v\cdot\varphi$-speed $\delta$, for any $v\in\mathbb S^{N-1}$.

Let $\epsilon>0$ and $\mathcal A_1$ be such a representation for an
arbitrary choice of $v_1\in\mathbb S^{N-1}$.  Inductively, suppose
that $m\leq N$ and that there exists a countable Borel decomposition
$Z=\cup_i Z_i$ such that each $\mu\llcorner Z_i$ has $m-1$ Alberti
representations $\mathcal A_1^i,\ldots,\mathcal A_{m-1}^i$.  Then, by
refining the representations if necessary, we may suppose that there
exists $w_1^i,\ldots,w_{m-1}^i\in \mathbb S^{N-1}$ such that $\mathcal
A_k^i$ is in the $\varphi$-direction of $C(w_k^i,\epsilon)$ for each
$1\leq k< m$.  We choose $w_m^i$ orthogonal to each $w_j^i$ and let
$\mathcal A_m^i$ be an Alberti representation of $\mu\llcorner Z_i$
with $w_m^i\cdot\varphi$-speed $\delta$.  Then $w_m^i$ satisfies
\[w_m^i\cdot(\varphi\circ\gamma_j)'(t)\leq
\frac{\sqrt{\epsilon(2-\epsilon)}}{\lambda}\Lip(\varphi,\gamma_j(t_j))\Lip(\gamma_j,t_j)\]
for almost every $\gamma_j\in\mathcal A_j^i$ and almost every
$t_j\in\dom\gamma_j$, for each $1\leq j<m$.  In particular,
\[\frac{w_m^i\cdot(\varphi\circ\gamma_m)'(t_m)}{\Lip(\gamma_m,t_m)} >
\frac{w_m^i\cdot(\varphi\circ\gamma_j)'(t_j)}{\Lip(\gamma_j,t_j)} + \left(\delta-\frac{2\sqrt{\epsilon(2-\epsilon)}}{\lambda}\right)\Lip(w_m^i\cdot\varphi,\gamma_m(t_m))\]
for every $1\leq j<m$, almost every $\gamma_m\in\mathcal A_m^i$ and
$\gamma_j\in\mathcal A_j^i$ and almost every $t_m\in\gamma_m$ and
$t_j\in\gamma_j$.

By repeating this process $N-1$ times, we obtain $N$ Alberti
representations that satisfy the hypotheses of
Lemma \ref{lem:albertibound} and so $N\leq
M(\delta/2-\sqrt{\epsilon(2-\epsilon)}/\lambda,\eta)$.  This is
true for every $\epsilon>0$ and so $N\leq M(\delta/2,\eta)$, as required.
\end{proof}

Using this result we may also bound the dimension of charts in Lipschitz
differentiability spaces.

\begin{corollary}\label{cor:dimensionbound}
Let $(U,\varphi)$ be a chart in a Lipschitz differentiability space
and $\eta\geq 1$.  Suppose that $Y\subset U$ has
positive $\mu$ measure, is contained within some $W\in D_\eta$ and
that any $\widetilde B_\delta$ subset of $Y$ is $\mu$-null.  Then the
dimension of $(U,\varphi)$ is bounded above by $M(\delta/2,\eta)$.
\end{corollary}

\begin{proof}
Suppose that $(U,\varphi)$ is a chart of dimension $n$.  Then by Lemma
\ref{lem:discuniq}, $\Lip(v\cdot\varphi,x_0)>0$ for all
$v\in\mathbb S^{n-1}$ and almost every $x_0\in U$.  Therefore, by
Proposition \ref{prop:firstcharts}, $n\leq M(\delta/2,\eta)$.
\end{proof}

Using the previous construction we may give our characterisation of
Lipschitz differentiability spaces via null sets.

\begin{theorem}\label{thm:doublingchar}
For $(X,d,\mu)$ a metric measure space, the following are equivalent:
\begin{enumerate}
\item \label{item:diffspace} $(X,d,\mu)$ is a Lipschitz
  differentiability space.
\item \label{item:atilde} Any $\widetilde A$ subset of $X$ is
  $\mu$-null and $X$ is pointwise doubling.
\item \label{item:btilde} Any $\widetilde B$ subset of $X$ is
  $\mu$-null and $X$ is pointwise doubling.
\item \label{item:speedreps} There exists a countable Borel
  decomposition $X=\cup_i X_i$ and sequences $\eta_i \geq 1$ and
  $\delta_i>0$ such that:
\begin{itemize}\item For any
  Lipschitz $f\colon X\to\mathbb R$ and $i\in\mathbb N$, there exists
  an Alberti representation of $\mu\llcorner X_i$ with $f$-speed
  $\delta_i$.
\item For every $i\in\mathbb N$ there exists a $W\in D_{\eta_i}$ such
  that $\mu(X_i\setminus W)=0$.
\end{itemize}
\end{enumerate}
\end{theorem}

\begin{proof}
We will prove the implications $\eqref{item:diffspace}\Rightarrow
\eqref{item:atilde}\Rightarrow\eqref{item:btilde}\Rightarrow
\eqref{item:speedreps}\Rightarrow \eqref{item:diffspace}$.  Firstly, by
Theorem \ref{thm:atildenull}, any $\widetilde A$ subset of a Lipschitz
differentiability space is $\mu$-null.  Further, by Corollary \ref{cor:doubling}, any Lipschitz differentiability space is
pointwise doubling, giving the first implication.  The second
implication is true since $\widetilde B\subset \widetilde A$ for any
metric measure space.

To prove $\eqref{item:btilde}\Rightarrow\eqref{item:speedreps}$
observe that, for any $\delta>0$, $\widetilde B_\delta$ is closed
under taking countable unions.  Therefore there exists a countable
Borel decomposition $X=\cup_i X_i\cup N$ where $N\in\widetilde B$
and such that, for each $i\in\mathbb N$, every $\widetilde B_{1/i}$
subset of $X_i$ has measure zero.  In particular, given a Lipschitz
function $f\colon X\to\mathbb R$, we may apply Corollary
\ref{cor:onerep2} to obtain an Alberti representation of $\mu\llcorner
X_i$ with $f$-speed $1/i$.  Further, by Lemma
\ref{lem:ptwisedoubling}, for each $i\in\mathbb N$ there exists a
countable Borel decomposition $X_i=\cup_j (X_i^j \cup N_i)$ where $N_i$
is $\mu$-null and for each $j\in\mathbb N$ there exists a $W\in
D_{1/j}$ such that $X_i^j\subset W$.  Therefore, the decomposition
$X=\cup_{i,j} (X_i^j \cup N_i)$ is of the required form.

For the final implication, by applying Lemma
\ref{lem:boundgivesdiff} and Proposition \ref{prop:firstcharts}, we see
that $(X,d,\mu)$ is a Lipschitz differentiability space.
\end{proof}

Finally, as mentioned above, metric measure spaces in which porous
sets have measure zero are pointwise doubling.  Moreover,
by Corollary \ref{cor:porousnull}, porous sets in Lipschitz
differentiability spaces have measure zero.  Therefore, we obtain an
additional characterisation of Lipschitz differentiability spaces
entirely via null sets.

\begin{theorem}\label{thm:nullchar}
For $(X,d,\mu)$ a metric measure space, the following are equivalent:
\begin{itemize}
\item $(X,d,\mu)$ is a Lipschitz differentiability space.
\item Any $\widetilde A$ and any porous set in $X$ is $\mu$-null.
\item Any $\widetilde B$ and any porous set in $X$ is $\mu$-null.
\end{itemize}
\end{theorem}

\section{Arbitrary Alberti representations}\label{sec:arbreps}
We now extend the results from the previous sections to find Alberti
representations in the direction of an arbitrary cone $C$ in a chart
$(U,\varphi)$ of a Lipschitz differentiability space $(X,d,\mu)$.  As
before, we will produce this representation using Corollary
\ref{cor:onerep2} and so we are required to show that any Borel
$S\subset U$ that satisfies $\mathcal H^1(\gamma\cap S)=0$ for any
$\gamma\in\Gamma(X)$ in the $\varphi$-direction of $C$ has measure
zero.  This is achieved by giving a decomposition $S=\cup_i S_i$ such
that each $S_i$ satisfies $\mathcal H^1(\gamma\cap S_i)=0$ for any
$\gamma\in\Gamma(X)$ in the $\varphi$-direction of a cone $C_i$ that
defines the direction of an existing Alberti representation.

For this we will need to deduce properties of the derivative of a
function obtained from Lemma \ref{lem:firstfunlds}.  In particular, we
require a bound on the directional derivative of the function in the
direction $w$, for $w$ as in the hypotheses of the Lemma.  This is
possible for any $x_0$ for which there exists $x_m\to x_0$ such that
$\varphi(x_m)-\varphi(x_0)$ is parallel to $w$.

To begin we concatenate the curves obtained from our existing Alberti
representations, in suitable ratios, to show the existence of such a
sequence.  Recall the notion of separated Alberti representations
given in Definition \ref{def:separated}.

\begin{lemma}\label{lem:pointtangents}
Let $(X,d,\mu)$ be a metric measure space, $\varphi\colon X\to\mathbb
R^n$ Lipschitz and $\delta,\xi>0$.  Suppose that $X$ has $n$
$\xi$-separated Alberti representations with speed strictly greater
than $\delta$.  Then there exists an $\eta>0$ such that, for any
measurable $S\subset X$ and for almost every $x_0\in S$, given any
$v\in\mathbb S^{n-1}$ there exists $S\ni x_m\to x_0$ with
\[\limsup_{m\to\infty}\left\|\frac{\varphi(x_m)-\varphi(x_0)}{d(x_m,x_0)}-v \frac{\|\varphi(x_m)-\varphi(x_0)\|}{d(x_m,x_0)}\right\|=0\]
and
\[\limsup_{m\to\infty}\frac{\|\varphi(x_m)-\varphi(x_0)\|}{d(x_m,x_0)}\geq\eta\Lip(\varphi,x_0).\]
\end{lemma}

\begin{proof}
Let $S\subset X$ be measurable, $w\in\mathbb S^{n-1}$ and
$0<\epsilon<1$.  By refining if necessary, we may suppose that the
above Alberti representations are in the $\varphi$-direction of
$\xi$-separated cones $C_1=C(w_1,\epsilon), \ldots,
C_n=C(w_n,\epsilon)$.  Further, for each $1\leq i \leq n$ let
$C'_i\supset C_i$ be open cones that are also $\xi$-separated.

For $1\leq i\leq n$ let $\Gamma^i$ be the set of $\gamma\in
\Gamma(X)$ in the $\varphi$-direction of $C_i$ with speed $\delta$.
Further, for each $0<R<\epsilon$, let $G^i_{R}(S)$ be the set of
$x_0\in S$ for which there exist a $\gamma\in\Gamma^i$ and a
$t_0\in\dom\gamma$ with $\gamma(t_0)=x_0$ such that, for every
$0<r<R/\|(\varphi\circ\gamma)'(t_0)\|$ and every $t\in \dom\gamma$
with $|t-t_0|\leq 5R/\|(\varphi\circ\gamma)'(t_0)\|$, the following
conditions hold:
\begin{enumerate}
\item \label{item:arbseqdensity} $\mathcal L^1(\gamma^{-1}(S)\cap B(t_0,2r))\geq 4r(1-\epsilon)$.
\item \label{item:arbseqderapprox}$\|\varphi(\gamma(t))-\varphi(\gamma(t_0))-(\varphi\circ\gamma)'(t_0)\|\leq \epsilon \|(\varphi\circ\gamma)'(t_0)\||t-t_0|$.
\item \label{item:arbseqspeedbound}$\|(\varphi\circ\gamma)'(t_0)\||t-t_0|\leq 2\|\varphi(\gamma(t))-\varphi(\gamma(t_0))\|$.
\item \label{item:arbseqincone}$\varphi(\gamma(t))-\varphi(\gamma(t_0))\in C'_i$.
\item \label{item:arbseqspeed}$\|(\varphi\circ\gamma)'(t_0)\||t-t_0|\geq \delta \Lip(\varphi,\gamma(t_0))d(\gamma(t),\gamma(t_0))/2$.
\end{enumerate}
Since $\mu$ has an Alberti representation in the
$\varphi$-direction of each $C_i$ with speed $\delta$, by Proposition
\ref{prop:curvefullmeas} each $G^i_R(S)$ is measurable and
monotonically increases to a set of full measure in $S$ as $R\searrow 0$.

Now let $x_0\in G^i_{R}(S)$, let $\gamma\in\Gamma^i$ be as in the
definition of $G^i_R$ for $x_0$ and write $w=\sum_i \lambda_i w_i$.
Then for any $0<r<R/\max \lambda_i$,
\[r\lambda_i/\|(\varphi\circ\gamma)'(t_0)\|\leq
R/\|(\varphi\circ\gamma)'(t_0)\|\]
and so, by \eqref{item:arbseqdensity}, there exists a $t\in\dom\gamma$ with
\[0\leq t-t_0-\frac{r\lambda_i}{\|(\varphi\circ\gamma)'(t_0)\|}\leq \frac{4r\epsilon}{\|(\varphi\circ\gamma)'(t_0)\|}\leq 4\epsilon(t-t_0).\]
In particular
\[|t-t_0|\leq 5R/\|(\varphi\circ\gamma)'(t_0)\|\]
and
\[\left|\|(\varphi\circ\gamma)'(t_0)\|-\frac{r\lambda_i}{|t-t_0|}\right|\leq 4\epsilon.\]
Note also that, since $(\varphi\circ\gamma)'(t_0)$ belongs to $C_i$,
\[\left\|\frac{(\varphi\circ\gamma)'(t_0)}{\|(\varphi\circ\gamma)'(t_0)\|}- w_i \right\| \leq \sqrt{2\epsilon}.\]
Therefore, by the triangle inequality, \eqref{item:arbseqderapprox}
and \eqref{item:arbseqspeedbound},
\begin{align*}\|\varphi(\gamma(t))-\varphi(\gamma(t_0))-r\lambda_i w_i\|&\leq 10\sqrt{\epsilon}\|(\varphi\circ\gamma)'(t_0)\||t-t_0|\\
&\leq 20\sqrt{\epsilon}\|\varphi(\gamma(t))-\varphi(\gamma(t_0))\|
\end{align*}
Moreover, by \eqref{item:arbseqspeed},
\[\delta\Lip(\varphi,\gamma(t_0))d(\gamma(t),\gamma(t_0))\leq 5R\leq 5\epsilon.\]
That is, for every $x_0\in G^i_R$, there exists an $x\in S$ with
\begin{equation}\label{eq:tangentcurve1}\|\varphi(x)-\varphi(x_0)-r\lambda_i w_i\|\leq 20\sqrt{\epsilon}\|\varphi(x)-\varphi(x_0)\|\end{equation}
and
\begin{equation}\label{eq:tangentcurve2}(x,x_0)\leq 5\epsilon/\delta\Lip(\varphi,\gamma(t_0)).\end{equation}

Now define $S_R^0=S$ and for each $0\leq i < n$ define
\[S_{R}^{i+1}=G^{i+1}_{R}\left(S^i_R\right).\]
Then for any $x_0\in S^n_R$ and $0\leq i <n$ there exists $x_i\in
S^{n-i}_R$ such that the relations \eqref{eq:tangentcurve1} and
\eqref{eq:tangentcurve2} hold between $x_{i}$ and
$x_{i+1}$.  Therefore
\[\|\varphi(x_n)-\varphi(x_0)-r w\|\leq 20\sqrt{\epsilon}\sum_{0\leq i<n} \|\varphi(x_i)-\varphi(x_{i+1})\|\]
and so, by the triangle inequality,
\[\|\varphi(x_n)-\varphi(x_0)-\|\varphi(x_n)-\varphi(x_0)\| w\|\leq
40\sqrt{\epsilon}\sum_{0\leq i<n}  \|\varphi(x_i)-\varphi(x_{i+1})\|.
\]
However each $\gamma$ above was chosen so that
$\varphi(x_{i})-\varphi(x_{i-1})\in C'_i$ for each $0\leq i <n$.
Since the $C'_i$ form a collection of $\xi$-separated cones,
\begin{align*}\|\varphi(x_0)-\varphi(x_n)\|&=\left\|\sum_{0\leq i <n} \varphi(x_i)-\varphi(x_{i+1})\right\|\\
&\geq \xi\max_{0\leq i < n}\|\varphi(x_i)-\varphi(x_{i+1})\|\\
&\geq \xi\delta \Lip(\varphi,x_0)\max_{0\leq i <n} d(x_i,x_{i+1})/4>0.\end{align*}
Therefore
\[\|\varphi(x_n)-\varphi(x_0)-\|\varphi(x_n)-\varphi(x_0)\| w\|\leq 40\sqrt{\epsilon} n\Lip\varphi d(x_n,x_0)/\xi\]
and by the triangle inequality
\[\|\varphi(x_n)-\varphi(x_0)\|\geq\xi\delta\Lip(\varphi,x_0) d(x_n,x_0)/4n.\]
Note also that we must have $x_0\neq x_n$ and
\[d(x_n,x_0)\leq \sum_{0\leq i < n} d(x_i,x_{i+1})\leq 5n\epsilon.\]

Finally, $\cup_{m\in\mathbb N}S^n_{1/m}$ is a set of full measure in
$S$ and so for almost every $x_0\in S$ there exists an $x_n\in S$ with
the above properties, for a given $0<\epsilon<1$.  Taking a
countable intersection over $\epsilon\in (0,1)\cap\mathbb Q$ completes
the proof for $\eta=\xi\delta/n$.
\end{proof}

\begin{corollary}\label{cor:thinfunct}
Let $(U,\varphi)$ be a $\lambda$-structured chart in a Lipschitz
differentiability space and for some $\xi,\delta>0$, let $V\subset U$
be Borel such that $\mu\llcorner V$ has $n$ $\xi$-separated Alberti
representations with speed strictly greater than $\delta$.  Suppose
that, for some $w\in\mathbb S^{n-1}$ and $0<\epsilon<1$, $S\subset V$ is
closed and satisfies $\mathcal H^1(\gamma\cap S)=0$ for any
$\gamma\in\Gamma(X)$ in the $\varphi$-direction of $C(w,\epsilon)$. Then
there exists an $\eta>0$ and, for any $\zeta>0$, a
$(K(\epsilon)+1+\zeta)\Lip\varphi$-Lipschitz function $f\colon U\to\mathbb R$
such that
\begin{enumerate}
\item For every $x_0\in S$ and $x\in U$ with $(\varphi(x)-\varphi(x_0))\cdot w\geq 0$,
\[f(x)-f(x_0)\geq (\varphi(x)-\varphi(x_0))\cdot w-\zeta.\]
\item \label{item:fundercon}For almost every $x_0\in S$
\[Df(x_0)\cdot w\leq \zeta/\eta\lambda.\]
\end{enumerate}
\end{corollary}

\begin{proof}
Let $f\colon U\to\mathbb R$ be the Lipschitz function obtained from an
application of Lemma \ref{lem:firstfunlds} (for
$\delta=\zeta\min\{1,1/\operatorname{diam}U,1/\Lip\varphi\}$), so
that $f$ automatically satisfies all of the required properties (for
$2\zeta$) except for \eqref{item:fundercon}.  Further, there exists
a $\rho>0$ such that, for every $x_0\in S$ and $y,z\in B(x_0,\rho)$,
\[|f(y)-f(z)|\leq K(\epsilon)\|\mathcal P(\varphi(y)-\varphi(z))\|+\zeta d(y,z)\]
for $K(\epsilon)=(1-\epsilon)/\sqrt{\epsilon(2-\epsilon)}$ and $\mathcal
P\colon \mathbb R\to\mathbb R$ the orthogonal projection onto the
plane orthogonal to $w$ passing through the origin.

Now let $\eta>0$ such that the conclusion of Lemma
\ref{lem:pointtangents} holds.  Then for almost every $x_0\in V$,
$Df(x_0)$ exists and there exist $U\ni x_m\to x_0$ with
$\|\varphi(x_m)-\varphi(x_0)\|\geq \eta d(x_m,x_0)$ such that
\[\frac{\|\varphi(x_m)-\varphi(x_0) -\|\varphi(x_m)-\varphi(x_0)\|w\|}{d(x_m,x_0)}\to 0.\]
In particular, by the triangle inequality,
\[\limsup_{m\to\infty}\frac{\|\mathcal
    P(\varphi(x_m)-\varphi(x_0))\|}{d(x_m,x_0)} \leq \|\mathcal P(w)\|
  \limsup_{m\to\infty}\frac{\|\varphi(x_m)-\varphi(x_0)\|}{d(x_m,x_0)} =0\]
and so
\[\limsup_{m\to\infty}\frac{|f(x_m)-f(x_0)|}{d(x_m,x_0)}\leq \zeta.\]
Therefore
\begin{align*}\eta\Lip(\varphi,x_0)Df(x_0)\cdot w &\leq \limsup_{m\to\infty} \frac{\|\varphi(x_m)-\varphi(x_0)\|}{d(x_m,x_0)}Df(x_0)\cdot w\\
&\leq\limsup_{m\to\infty}\frac{Df(x_0)\cdot(\varphi(x_m)-\varphi(x_0))}{d(x_m,x_0)}\\
&= \limsup_{m\to\infty}\frac{|f(x_m)-f(x_0)|}{d(x_m,x_0)}\\
&\leq \zeta.
\end{align*}
Since $(U,\varphi)$ is a $\lambda$-structured chart,
$\Lip(\varphi,x_0)\geq \lambda$ and so dividing by $\eta\lambda$
completes the proof.
\end{proof}

Given a set $S$ as in the hypotheses of Corollary \ref{cor:thinfunct},
we will decompose it using the following Lemma.  In our application,
the $D_m$ in the hypotheses will be the derivatives of Lipschitz
functions obtained from Corollary \ref{cor:thinfunct}.

\begin{lemma}\label{lem:decompder}
Let $(X,d,\mu)$ be a metric measure space, $U\subset X$ Borel and
$w\in\mathbb S^{n-1}$.  Suppose that, for some $\beta \geq 0$ and
$E>0$, there exists a sequence of measurable functions $D_m\colon
X\to\mathbb R^n$ with essential supremum $E$ such that
\[\limsup_{m\to\infty}D_m(x)\cdot w +\beta \leq w\cdot w=1\]
for almost every $x\in U$.  Suppose further that, for some $\xi>0$ and
$\xi$-separated $w_1,\ldots,w_k\in\mathbb S^{n-1}$, $w$ belongs to the
convex cone of the $w_i$.  Then there exists a Borel decomposition
\[S=N\cup\bigcup_{j\in\mathbb N} S_j\]
with $\mu(N)=0$ and, for every $j\in\mathbb N$, a $1\leq i
\leq n$ such that, for every $0<\theta<1$,
\[\lim_{m\to\infty}\frac{1}{m}\sum_{1\leq k\leq
  m}D_k(x)\cdot v + \frac{\|v\|\beta\xi}{n} \leq w\cdot v +
\sqrt{2\theta}(E+1)\|v\|\]
uniformly for $x\in S_j$ and $v\in C(w_i,\theta)$.
\end{lemma}

\begin{proof}
Since each $D_m$ is bounded, there exists a $g\in L^2(U)$ such that,
after passing to a subsequence, $D_m\to g$ weakly.  Then, by the
Banach-Saks theorem, there exists a further subsequence such that the
functions
\[\widetilde D_m:=\frac{1}{m}\sum_{1\leq j\leq m} D_j\]
converge pointwise almost everywhere to $g$.  Moreover, since
\[\limsup_{m\to\infty} D_m(x)\cdot w +\beta\leq w\cdot w,\]
$g(x)\cdot w +\beta\leq w\cdot w$ almost everywhere.

Let
\[S=N\cup\bigcup_{j\in\mathbb N} S_j\]
be a Borel decomposition of $S$ where $N$ is $\mu$-null and such that
$\widetilde D_m\to g$ uniformly on each $S_j$.  Further, since the
$w_i$ are $\xi$-separated and $w$ lies in the convex cone of the
$w_i$, there exists $0\leq \lambda_i\leq 1/\xi$ with $w=\sum_i\lambda_i
w_i$.  Therefore, for each $x_0$ in some $S_j$, there exists a $1\leq
i\leq n$
with
\[g\cdot w_i + \beta\xi/n \leq w\cdot w_i.\]
We may therefore decompose each $S_j$ into the sets
\[S_j^i=\{x_0\in S_j: g\cdot w_i + \beta\xi/n \leq w\cdot w_i\}\]
for $1\leq i \leq k$.  By taking a further decomposition, we may
suppose that each $S_j^i$ is compact.

Finally, for any $0<\theta<1$, $j\in\mathbb N$, $1\leq i\leq n$ and
$v\in C(w_i,\theta)$,
\begin{align*}g(x)\cdot v +\|v\|\beta\xi/n &= \|v\|(g(x)\cdot w_i +\beta\xi/n) + g(x)\cdot(v-\|v\| w_i)\\
&\leq \|v\|(w\cdot w_i \sqrt{2\theta}E)\\
&\leq w\cdot v + \sqrt{2\theta}(E+1)\|v\|.\end{align*}
Since $\widetilde D_m\to g$ uniformly on each $S_j^i$, this
decomposition is of the required form.
\end{proof}

Finally, we will show that a set with certain properties intersects a
Lipschitz curve $\gamma$ in a set of measure zero by applying the
following Lemma to the domain of $\gamma$.

\begin{lemma}\label{lem:nondiffonr}
Let $S\subset \mathbb R$ be measurable and $L>0$.  Suppose that there
exists a sequence of $L$-Lipschitz functions $f_m\colon S\to\mathbb R$
and measurable functions $\Phi,\Psi\colon S\to \mathbb R$ such that,
for almost every $t_0\in S$, $\Phi(t_0)<\Psi(t_0)$ and:
\begin{enumerate}\item \label{item:rdiff1}There exists $m_j\to\infty$ such that, for every $t\in S$ with $t\geq t_0$,
\[f_{m_j}(t)-f_{m_j}(t_0)\geq \Psi(t_0)(t-t_0)-1/{m_j}.\]
\item \label{item:rdiff2}There exists an $M\in\mathbb N$ such that, for every $m\geq M$,
\[Df_m(t_0)\leq \Phi(t_0).\]
\end{enumerate}
Then $S$ is Lebesgue null.
\end{lemma}

\begin{proof}
Suppose that such an $S$ has positive Lebesgue measure.  Then there
exists an $M\in\mathbb N$, $\alpha<\beta\in\mathbb R$ and $S'\subset
S$ of positive measure such that \eqref{item:rdiff1} and
\[Df_m(t_0)\leq \Phi(t_0) \leq \alpha <\beta \leq \Psi(t_0)\]
for all $t_0\in S'$ and $m\geq M$.  Let $t_0$ be a density point of
$S'$ and $R>0$ such that, for every $t\in(t_0,t_0+R)$,
\[\frac{\mathcal L^1(S'\cap (t_0,t))}{t}\geq 1-(\beta-\alpha)/2L.\]
Then, for any $m\geq M$ and $t\in(t_0,t_0+R)\cap S$,
\begin{align*}f_m(t)-f_m(t_0) &= \int_{(t_0,t)\cap S'}Df_m+\int_{(t_0,t)\setminus S'}Df_m\\
&\leq \alpha (t-t_0) +(t-t_0)L(\beta-\alpha)/2L.\\
&\leq (\beta+\alpha)(t-t_0)/2.
\end{align*}
However, if $m_j\geq M$ such that \eqref{item:rdiff1} holds for $t_0$,
then for any $t\in (t_0,t_0+R)\cap S$,
\[(\beta+\alpha)(t-t_0)/2 \geq f_{m_j}(t)-f_{m_j}(t_0) \geq \beta(t-t_0) -1/{m_j}.\]
In particular, $|t-t_0|\leq 2/(\beta-\alpha)m_j$ for $m_j\to \infty$ and so $t=t_0$,
contradicting our assumption that $t_0$ is a density point of $S$.
\end{proof}

We now apply these results to a chart in a Lipschitz differentiability
space.

\begin{theorem}\label{thm:arbrep}
Let $(U,\varphi)$ be an $n$-dimensional chart in a Lipschitz
differentiability space $(X,d,\mu)$, $w\in\mathbb S^{n-1}$ and
$0<\epsilon<1$.  Then there exists an Alberti representation of $\mu
\llcorner U$ in the $\varphi$-direction of $C(w,\epsilon)$.
\end{theorem}

\begin{proof}
Since any chart in a Lipschitz differentiability space has a countable
decomposition into a $\mu$-null set and structured charts, and since
we may combine Alberti representations using Lemma \ref{lem:sumreps},
it suffices
to prove the result for $(U,\varphi)$ a $\lambda$-structured chart,
for some $\lambda>0$.  Further, by Corollary \ref{cor:onerep2}, there
exists a decomposition $U=A\cup S$ where $\mu\llcorner A$ has an
Alberti representation in the $\varphi$-direction of $C(w,\epsilon)$
and $S$ satisfies $\mathcal H^1(\gamma\cap S)=0$ for any
$\gamma\in\Gamma(X)$ in the $\varphi$-direction of $C(w,\epsilon)$.
Finally, by Theorem \ref{thm:spanrep}, there exists a countable
decomposition $U=\cup_m U_m$ such that each $\mu\llcorner U_m$ has $n$
$\varphi$-independent Alberti representations.  By refining these
representations if necessary, we may suppose that for any $m\in\mathbb
N$ there exist $w_1,\ldots,w_n\in\mathbb S^{n-1}$ and
$0<\theta,\xi,\delta<1$ such that the Alberti representations of
$\mu\llcorner U_m$ are in the $\varphi$-direction of $\xi$-separated
cones $C(w_1,\theta),\ldots,C(w_n,\theta)$, with speed $\delta$.  By
refining these representations further, we may suppose that
\[\frac{\sqrt{2\theta}(2+K(\epsilon)+\lambda)}{\lambda}\leq \frac{\xi}{4n}\]
for $K(\epsilon)=(1-\epsilon)/\sqrt{\epsilon(2-\epsilon)}$.
Therefore, we are required to show that, for any $m\in\mathbb N$, any
compact $S'\subset S\cap U_m$ is $\mu$-null.

We apply Corollary \ref{cor:thinfunct} (with $\zeta=1/m$) to obtain a sequence of
$(2+K(\epsilon))$-Lipschitz functions $f_m\colon U\to\mathbb R$ such that:
\begin{itemize}\item For every $x_0\in S'$ and $x\in U$ with $(\varphi(x)-\varphi(x_0))\cdot w\geq 0$,
\[f_m(x)-f_m(x_0)\geq (\varphi(x)-\varphi(x_0))\cdot w-1/m.\]
\item For almost every $x_0\in S'$,
\[\limsup_{m\to\infty} Df_m(x_0)\cdot w + 1\leq w\cdot w.\]
\end{itemize}
Note that, since $(U,\varphi)$ is a $\lambda$-structured chart, by
Lemma \ref{lem:nondiffcond},
\[\|Df_m(x_0)\|\leq (2+K(\epsilon))/\lambda.\]
Therefore, by Lemma \ref{lem:decompder}, there exists a countable
Borel decomposition
\[S=N\cup\bigcup_{j\in\mathbb N} S_j\]
where $\mu(N)=0$ and for each $j\in\mathbb N$ there exists a $1\leq
i \leq n$ and an $M\in\mathbb N$ such that, for every $m\geq M$ and
$x_0\in S_j$,
\begin{equation}\label{eq:arbrep}\frac{1}{m}\sum_{1\leq k\leq
    m}Df_k(x_0)\cdot v +\frac{\xi}{n} \|v\|\leq w\cdot v +\frac{\xi}{2n} \|v\|\end{equation}
for all $v\in C(w_i,\theta)$.  Define, for each $m\in\mathbb N$
\[F_m=\frac{1}{m}\sum_{1\leq k\leq m} f_k.\]
For each $j\in\mathbb N$ and almost every $x_0\in S_j$ each
$Df_k(x_0)$ exists and so
\[DF_m=\frac{1}{m}\sum_{1\leq k\leq m}Df_k.\]
Further, for every $x_0\in S$ and $x\in U$ with
$(\varphi(x)-\varphi(x_0))\cdot w\geq 0$,
\[F_m(x)-F_m(x_0)\geq(\varphi(x)-\varphi(x_0))\cdot w-1/m.\]
We fix $j\in\mathbb N$ and let $1\leq i\leq n$ and $M\in\mathbb N$ such that
\eqref{eq:arbrep} holds for every $m\geq M$ and $x_0 \in S_j$.  It suffices to prove
that $\mu(S_j)=0$.  Given the above Alberti representations of
$\mu\llcorner U_m$, it suffices to prove that $\mathcal H^1(\gamma\cap
S_j)=0$ for any $\gamma\in\Gamma(X)$ in the $\varphi$-direction of
$C(w_i,\theta)$.

To show this, fix $\gamma\in\Gamma(X)$ in the $\varphi$-direction of
$C(w_i,\theta)$ and define, for each $m\in\mathbb N$, the
$(2+K(\epsilon))\Lip\varphi\Lip\gamma$-Lipschitz function
\[g_m=F_m\circ\gamma\colon \dom\gamma\to\mathbb R.\]
Then, for any $t_0\in\dom\gamma$, if $(\varphi\circ\gamma)'(t_0)$
exists and $\gamma(t_0)\in S_j$,
\[Dg_m(t_0)=DF_m(\gamma(t_0))\cdot (\varphi\circ\gamma)'(t_0).\]
Therefore, if $(\varphi\circ\gamma)'(t_0)\in C(w_i,\theta)$,
\begin{align*} Dg_m(t_0) &\leq w\cdot(\varphi\circ\gamma)'(t_0) - \frac{\xi}{2n}\|(\varphi\circ\gamma)'(t_0)\|\\ 
&\leq w\cdot(\varphi\circ\gamma)'(t_0) -
  \frac{\xi}{2n}|w\cdot(\varphi\circ\gamma)'(t_0)|\\
&:= \Phi(t_0).
\end{align*}
Also, for any $t,t_0\in\dom\gamma$,
\[g_m(t)-g_m(t_0) \geq (\varphi(\gamma(t))-\varphi(\gamma(t_0)))\cdot
w -1/m.\]

Suppose that $\gamma^{-1}(S_j)$ has positive measure.   Then there exists an $R>0$ and a
$T\subset\gamma^{-1}(S_j)$ of positive measure such that, for every
$t_0\in T$ and $t\in\dom\gamma$ with $|t-t_0|\leq R$,
$(\varphi\circ\gamma)'(t_0)\in C(w_i,\theta)$ and
\begin{align*}(\varphi(\gamma(t))-\varphi(\gamma(t_0)))\cdot w &\geq
\left(w\cdot(\varphi\circ\gamma)'(t_0)-
\frac{\xi}{4n} |w\cdot(\varphi\circ\gamma)'(t_0)|\right)(t-t_0)\\
&:=\Psi(t_0)(t-t_0).\end{align*}
We choose $s\in T$ such that $T':=T\cap B(s,R)$ has positive measure.
Then for almost every $t_0\in T'$, $\Psi(t_0)>\Phi(t_0)$ and:
\begin{itemize}\item For any $m\geq M$ and every $t\in T'$ with $t\geq t_0$,
\[g_m(t)-g_m(t_0)\geq \Psi(t_0)(t-t_0)-1/m.\]
\item For every $m\geq M$,
\[Dg_m(t_0)\leq \Phi(t_0).\]
\end{itemize}
Therefore, by Lemma \ref{lem:nondiffonr}, $T'$ is Lebesgue null, a
contradiction.
\end{proof}

We may use this Theorem to improve our description of the local
structure of a Lipschitz differentiability space.

\begin{corollary}
Let $(U,\varphi)$ be an $n$-dimensional chart in a Lipschitz
differentiability space $(X,d,\mu)$.  Then for almost every $x\in U$
and any cone $C\subset \mathbb R^n$, there exists a
$\gamma^x\in\Gamma(X)$ such that $(\gamma^x)^{-1}(x)=0$ is a density point
of $(\gamma^x)^{-1}(U)$ and such that $(\varphi\circ\gamma^x)'(0)\in C$.
\end{corollary}

We also obtain another characterisation of Lipschitz differentiability
spaces corresponding to arbitrary Alberti representations.

\begin{definition}
For a metric measure space $(X,d,\mu)$ we define $\widetilde C$ to be
the set of Borel $S\subset X$ for which there exist
$0<\epsilon,\eta<1$ and, for every $\delta>0$, there exist an
$n\in\mathbb N$, a Lipschitz $\psi\colon X\to\mathbb R^n$ and a cone
$C\subset \mathbb R^n$ of width $\epsilon$ such that:
\begin{itemize}
\item For every $x_0\in S$,
\[\Lip(v\cdot\psi,x_0)>\eta \Lip(\psi,x_0).\]
\item For every $\gamma\in\Gamma(X)$ in the $\psi$-direction of $C$
  with $\psi$-speed $\delta$, $\mathcal H^1(\gamma\cap S)=0$.
\end{itemize}
\end{definition}

\begin{theorem}
For a metric measure space $(X,d,\mu)$ the following are equivalent:
\begin{itemize}
\item $(X,d,\mu)$ is a Lipschitz differentiability space.
\item Every $\widetilde C$ subset of $X$ is $\mu$-null and $X$ is
  pointwise doubling.
\item Every $\widetilde C$ subset of $X$ and every porous subset of
  $X$ is $\mu$-null.
\end{itemize}
\end{theorem}

\begin{proof}
These two conditions are sufficient since there exists a countable
decomposition of any $\widetilde A$ subset of
$X$ into $\widetilde C$ sets.  Therefore, if any $\widetilde C$ subset
of $X$ is $\mu$-null, so is any $\widetilde A$ subset.

Conversely, for any $S\in \widetilde C$ that is contained within a
structured chart of a Lipschitz differentiability space,
by Lemma \ref{lem:preatildenull} there exists a
$0<\epsilon<1$ and a sequence of cones $C_m$ of width $\epsilon$ such
that $\mathcal H^1(\gamma\cap S)=0$ for any $\gamma\in\Gamma(X)$ in
the $\varphi$-direction of $C_m$ with speed $1/m$.  Then there exists
a cone $C$ of width $\epsilon/2$ and $m_j\to\infty$ such that $C\subset
C_{m_j}$ for each $j\in\mathbb N$.  Therefore $\mathcal H^1(\gamma\cap
S)=0$ for any $\gamma\in\Gamma(X)$ in the $\varphi$-direction of $C$
and so, by Theorem \ref{thm:arbrep}, $S$ is $\mu$-null.
\end{proof}

\section{Relationship with other works}\label{sec:cheeger}
As an example of the use of our theory, we give an alternate proof of
Cheeger's Differentiation Theorem.  We begin by introducing the notion
of a Poincar\'e inequality in a metric measure space.

\begin{definition}\label{defn:poincare}
Let $(X,d)$ be a metric space and $f\colon X\to\mathbb R$ Lipschitz.
We say that a measurable function $\rho\colon X\to\mathbb R$ is an
\emph{upper gradient of $f$} if, for any $\gamma\in\Pi(X)$ that is
parametrised by arc length,
\[|f(\gamma_e)-f(\gamma_s)|\leq \int_{\dom\gamma}\rho\circ\gamma \text{ d}\mathcal L^1\]
where $\gamma_e$ and $\gamma_s$ are the end points of $\gamma$.

Further, for $p,P\geq 1$ we say that a metric measure space
$(X,d,\mu)$ satisfies the $p$-Poincar\'e inequality with constant $P$ if,
for every closed ball $B=\overline B(x_0,r)\subset X$, $\mu(B)>0$ and
\[\dashint_B |f-f_B|\leq Pr \left(\dashint_{PB} \rho^p\right)^{1/p}.\]
Here
\[f_B=\dashint_B f := \frac{1}{\mu(B)}\int_B f.\]

Finally, for $C\geq 1$, we say that a metric measure space $(X,d,\mu)$
is \emph{$C$-doubling} if
\[\mu(\overline B(x_0,r)) \leq C\mu(\overline B(x_0,r/2))\]
for every $x_0\in X$ and every $r>0$.
\end{definition}

\begin{theorem}[Cheeger \cite{cheeger-diff}]\label{thm:cheeger}
Any $C$-doubling metric measure space $(X,d,\mu)$ that satisfies the
$p$-Poincar\'e inequality with constant $P$ is a Lipschitz
differentiability space.  Moreover, each chart is of a dimension
bounded above by an integer depending only upon $C$ and $P$,
independent of the metric measure space.
\end{theorem}

\begin{proof}Since a $C$-doubling metric measure space is a $D_\eta$
set, for some $\eta\geq 1$ depending only upon $C$, by Theorem \ref{thm:nullchar}
and Corollary \ref{cor:dimensionbound} it suffices to prove the
existence of a $\delta>0$ depending only upon $C$ and $P$ such that
any $\widetilde B_\delta$ subset of $X$ is $\mu$-null.

Firstly, by Proposition 4.3.3 of \cite{keith}, for any $C$-doubling
metric measure space, there exists a $C'$ depending only upon $C$ such
that, for any Lipschitz $f\colon X\to\mathbb R$,
\[\Lip (f,x_0)\leq C'\lim_{r\to 0} \frac{1}{r}\dashint_{\overline B(x_0,r)}|f-f_B|\]
for almost every $x_0\in X$.  Therefore, if $\rho$ is an upper gradient
of $f$, by using the Poincar\'e inequality and applying Lebesgue's
differentiation theorem to $\rho$, we see that there exists a $C''>0$
depending only upon $C$ and $P$ such that
\begin{equation}\label{eq:cheeger}\rho(x_0)\geq C'' \Lip (f,x_0)
\end{equation}
for almost every $x_0\in X$.  We set $\delta=C''/2$.

Now suppose that $f\colon X\to\mathbb R$ is Lipschitz and $S\subset
\{x_0:\Lip(f,x_0)>0\}$ satisfies $\mathcal H^1(\gamma\cap S)=0$ for
every $\gamma\in\Gamma(X)$ with $f$-speed $\delta$.  Then for any
$\gamma\in\Gamma(X)$ and almost every $t_0\in\gamma^{-1}(S)$,
\[(f\circ\gamma)'(t_0)\leq \delta \Lip (f,\gamma(t_0))\Lip(\gamma,t_0).\]
In particular this is true for any $\gamma\in \Pi(X)$ that is
parametrised by arc length and so
\[\rho(x)=\begin{cases} \delta\Lip(f,x) & x\in S\\
\Lip(f,x) & \text{otherwise}\end{cases}\] is an upper gradient of $f$.
However, for almost every $x\in S$, by equation \eqref{eq:cheeger},
\[\delta\Lip(f,x) = \rho(x)\geq 2\delta \Lip(f,x)>0.\]
Therefore $S$ must be $\mu$-null.  In particular, any
$\widetilde B_\delta$ subset of $X$ is $\mu$-null, as required.
\end{proof}

As pointed out in the introduction, the existence of an Alberti
representation of any doubling Lipschitz differentiability space
$(X,d,\mu)$ that satisfies the Poincar\'e inequality may be deduced
from a theorem of Cheeger and Kleiner.  To see this, suppose that $S$
is compact, contained within a chart and satisfies $\mathcal
H^1(\gamma \cap S)=0$ for every $\gamma \in\Pi(X)$.  Then by applying
\cite{cheegerkleiner-rnp}, Theorem 4.2 (and adopting it's terminology)
with $f$ a component of the chart map and the negligible set
$N=\{(\gamma,t) : \gamma(t) \in S\}$, we see that the minimal upper
gradient of $f$ equals zero almost everywhere in $S$.  However, the
minimal upper gradient of $f$ equals $\Lip(f,.)>0$ almost everywhere
and so $S$ must be $\mu$-null.  An application of Lemma
\ref{lem:generalrep} gives the required Alberti representation.  In
fact, almost every curve in this Alberti representation is defined on
an interval.

In \cite{keith}, Keith introduced the Lip-lip condition (see below)
on a metric measure space and showed that any doubling metric measure
space with a Lip-lip condition is a Lipschitz differentiability space
(via the Poincar\'e inequality).  We now use our theory to give an
alternate proof of this fact and to prove the converse statement.
This gives a characterisation of Lipschitz differentiability spaces
via the Lip-lip condition.

\begin{definition}
We say that a metric measure space $(X,d,\mu)$ satisfies a
\emph{Lip-lip} condition if there exists a countable Borel
decomposition $X=\cup_i X_i$ and for each $i\in\mathbb N$ a
$\delta_i>0$ such that, for any Lipschitz $f\colon X\to\mathbb R$,
\begin{align*}\delta_i\Lip(f,x_0) &\leq \liminf_{r\to
  0}\sup\left\{\frac{|f(x)-f(x_0)|}{r}:0<d(x,x_0)<r\right\}\\
&:=\operatorname{lip}(f,x_0)\end{align*}
for almost every $x_0\in X_i$.
\end{definition}

\begin{theorem}
A metric measure space $(X,d,\mu)$ satisfies a Lip-lip condition if
and only if every $\widetilde B$ subset of $X$ is $\mu$-null.
\end{theorem}

\begin{proof}
First suppose that any $\widetilde B$ subset of $X$ is $\mu$-null.
Then by Theorem \ref{thm:nullchar}, there exists a countable Borel
decomposition $X=\cup_i X_i$ and for each $i\in\mathbb N$ a
$\delta_i>0$ such that, for any Lipschitz $f\colon X\to\mathbb R$ and
every $i\in\mathbb N$, there exists an Alberti representation of
$\mu\llcorner X_i$ with speed $\delta_i$.  In particular, for every
$i\in\mathbb N$ and almost every $x_0\in X_i$, there exist
$\gamma\in\Gamma(X)$ and $t_0$ a density point of $\dom\gamma$ such that
$\gamma(t_0)=x_0$ and
\[\delta_i \Lip(f,x_0) \leq (f\circ\gamma)'(t_0).\]
However, since $t_0$ is a density point of $\dom\gamma$,
\[ (f\circ\gamma)'(t_0) \leq \operatorname{lip}(f,x_0)\]
and so $X$ satisfies a Lip-lip condition.

Conversely, suppose that $S\subset X$ belongs to $\widetilde B$, so
that $S$ belongs to $\widetilde A$.  Then for any $\epsilon>0$ we may
apply Lemma \ref{lem:firstfunlds} to construct a sequence of functions
that satisfy the hypotheses of Lemma \ref{lem:prenondiff} on some
Borel $S'\subset S$ with $\mu(S')\geq \mu(S)-\epsilon$.  Therefore, by
applying Lemma \ref{lem:prenondiff} with positive sequences
$R_i,r_i\to 0$ such that $r_i/R_i\to 0$ as $i\to\infty$, for any
$\delta>0$ we obtain a Lipschitz function $f\colon X\to\mathbb R$ with
\[\delta \Lip(f,x_0)> \operatorname{lip}(f,x_0)\]
for almost every $x_0\in S'$.  In particular, if $X$ satisfies a
Lip-lip condition, $S'$ and hence $S$ are $\mu$-null.
\end{proof}

By combining this Theorem with Theorem \ref{thm:nullchar}, we obtain
the following Corollary.  Very shortly after the first preprint of
this paper appeared, Gong gave a second, independent proof of this
Corollary in \cite{gong-liplip}.

\begin{corollary}
A metric measure space $(X,d,\mu)$ is a Lipschitz differentiability
space if and only if it satisfies a Lip-lip condition and is pointwise
doubling.
\end{corollary}

\bibliography{../../references,../../mypapers}\bibliographystyle{amsalpha}
\end{document}